\pdfoutput=1
\RequirePackage{ifpdf}
\ifpdf 
\documentclass[pdftex]{sigma}
\else
\documentclass{sigma}
\fi

\usepackage{array}
\newcolumntype{P}[1]{>{\raggedright\let\newline\\\arraybackslash\hspace{0pt}}m{#1}}
\usepackage{amsxtra,ifthen,mathtools,enumitem,bbm}	

\newcommand{\integers}{\mathbb Z}
\newcommand{\rationals}{\mathbb Q}

\newcommand{\reals}{\mathbb R}

\newcommand{\thet}{\vartheta}

\newcommand{\sgn}{\operatorname{sgn}}
\newcommand{\vsgn}{\mathbf{sgn}}

\newcommand{\uf}{{\operatorname{uf}}}
\newcommand{\fr}{{\operatorname{fr}}}

\newcommand{\set}[1]{{\lbrace #1 \rbrace}}
\newcommand{\sett}[1]{{\bigl\lbrace #1 \bigr\rbrace}}

\newcommand{\br}[1]{{\langle #1 \rangle}}

\newcommand{\F}{{\mathcal F}}
\newcommand{\D}{{\mathfrak D}}
\newcommand{\N}{{\mathcal N}}
\newcommand{\p}{{\mathfrak p}}

\newcommand{\g}{\mathbf{g}}

\renewcommand{\c}{\mathbf{c}}
\renewcommand{\b}{\mathbf{b}}
\renewcommand{\k}{\mathbbm{k}}
\newcommand{\kk}{{\boldsymbol{k}}}
\newcommand{\ellell}{{\boldsymbol{\ell}}}

\renewcommand{\v}{\mathbf{v}}
\renewcommand{\u}{\mathbf{u}}
\newcommand{\w}{\mathbf{w}}
\newcommand{\tB}{{\tilde{B}}}

\newcommand{\V}{\mathcal{V}}
\newcommand{\U}{\mathcal{U}}

\newcommand{\Dom}{\operatorname{Dom}}

\newcommand{\Hom}{\operatorname{Hom}}
\newcommand{\Scat}{\operatorname{Scat}}
\newcommand{\Fan}{\operatorname{Fan}}

\newcommand{\can}{\underline{\operatorname{can}}}

\renewcommand{\d}{{\mathfrak d}}

\renewcommand{\c}{{\mathbf c}}

\newcommand{\bl}{\gamma}

\newcommand{\const}{c}

\newcommand{\RSChar}{\Phi}
\newcommand{\RS}{\RSChar}

\numberwithin{equation}{section}

\newtheorem{Theorem}{Theorem}[section]
\newtheorem*{Theorem*}{Theorem}
\newtheorem{Corollary}[Theorem]{Corollary}
\newtheorem{Lemma}[Theorem]{Lemma}
\newtheorem{Proposition}[Theorem]{Proposition}

\theoremstyle{definition}

\newtheorem{Example}[Theorem]{Example}
\newtheorem{Remark}[Theorem]{Remark}

\begin{document}

\allowdisplaybreaks

\newcommand{\arXivNumber}{2603.19391}

\renewcommand{\PaperNumber}{109}

\FirstPageHeading

\ShortArticleName{Mutation of Theta Functions}

\ArticleName{Mutation of Theta Functions}

\Author{Nathan READING~$^{\rm a}$ and Salvatore STELLA~$^{\rm b}$}

\AuthorNameForHeading{N.~Reading and S.~Stella}

\Address{$^{\rm a)}$~Department of Mathematics, North Carolina State University, USA}
\EmailD{\mail{reading@math.ncsu.edu}}
\URLaddressD{\url{http://nreadin.math.ncsu.edu}}

\Address{$^{\rm b)}$~Dipartimento di Ingegneria e Scienze dell'Informazione e Matematica,\\
\hphantom{$^{\rm b)}$}~Universit\`{a} degli Studi dell'Aquila, Italy}
\EmailD{\mail{salvatore.stella@univaq.it}}

\ArticleDates{Received March 30, 2026, in final form October 22, 2026; Published online November 08, 2026}

\Abstract{We give an account of mutation of theta functions in cluster scattering diagrams, starting with a notion of mutation that is related to, but different from, the notion of mutation defined by Gross, Hacking, Keel, and Kontsevich.
This different approach to mutation leads to several applications. Three of the applications simplify the process of computing structure constants for multiplication of theta functions, and these are used in another paper on cluster scattering diagrams of affine type. Notable in these three applications is the appearance of mutation symmetries and dominance regions. The other two applications have to do with pointed reduced bases, a variation on the pointed bases of Fan Qin. We give a characterization of pointed reduced bases analogous to Qin's characterization of pointed bases. All of these applications take place in a version of Gross, Hacking, Keel, and Kontsevich's canonical algebra that can be constructed for an arbitrary exchange matrix.}


\Classification{13F60}

\section{Introduction}
Cluster scattering diagrams were defined by Gross, Hacking, Keel, and Kontsevich~\cite{GHKK} to prove longstanding structural conjectures about cluster algebras.

The key fact connecting cluster scattering diagrams to cluster algebras is that cluster variables (and more generally cluster monomials) can be computed as theta functions for the cluster scattering diagram.
Theta functions, in turn, are computed as sums of monomials obtained from piecewise-linear curves called broken lines.

In this paper, we give a detailed account of how broken lines and theta functions change when the initial seed is mutated.
This account of mutation can in principle be recovered from \cite[Proposition~3.6]{GHKK} and vice versa, but we follow a very different convention for what mutation should mean in the context of scattering diagrams.
(For a comparison of the two notions of mutation, see \cite[Section~4]{scatfan}.
Briefly, the difference is that mutation in~\cite{GHKK} moves the ``positive chamber'', while mutation in this paper fixes the positive chamber.)

The notion of mutation contemplated here is motivated by the notion of ``initial seed mutations'' of cluster algebras and also motivated by applications to two related basic problems:
Determining theta functions for a cluster scattering diagram and determining the structure constants for multiplication of theta functions.

Given an $n\times n$ exchange matrix $B$ and a sequence $\kk$ of indices, the associated \emph{mutation map}~$\eta_\kk^B$ takes a vector in $\reals^n$, places it as a coefficient row under~$B$, mutates at indices~$\kk$, and reads the new coefficient row.
Mutation of cluster scattering diagrams amounts to applying a~mutation map to walls and adjusting the scattering terms on the walls appropriately.
(A~precise statement is reproduced here as Theorem~\ref{mut thm}. Versions of this fact are in \cite[Section~3]{CGMMRSW}, \mbox{\cite[Lemma~5.2.1]{Muller}} and \cite[Theorem~4.2]{scatfan}.
The equivalent fact using the other notion of mutation is \cite[Theorem~1.24]{GHKK}.)

The key technical lemma in this paper (Lemma~\ref{mut broken line}) is that theta functions also mutate by applying mutation maps and adjusting the (Laurent monomial) labels appropriately.
Rather than ``translating'' the result from \cite[Proposition~3.6]{GHKK} which proves it for the other notion of mutation, it is more straightforward to prove it here directly.

To make Lemma~\ref{mut broken line} apply to an arbitrary initial exchange matrix $B$, we work in an algebra (the \emph{small canonical algebra}, defined in Section~\ref{canon sec}) that, loosely speaking, is ``the algebra generated by the theta functions''.
In \cite{GHKK}, a canonical algebra is only defined when products of theta functions expand as finite sums of theta functions.
Here, we appeal to the fact that theta functions can in any case be multiplied as formal power series to define an algebra (however poorly behaved) in which our result on structure constants makes sense.
To make the small canonical algebra a reasonable setting for mutation of theta functions, we need nondegeneracy conditions on the extended exchange matrix~$\tB$ (which is taken to be wide rather than tall).
We say that $\tB$ has \emph{nondegenerate coefficients} if the rows that are adjoined right of~$B$ to make~$\tB$ are linearly independent.
This is stronger than the requirement in~\cite{GHKK} that the rows of~$\tB$ are linearly independent.
(However, we point out in Section~\ref{coeff sec} that there is a reasonable way to define some theta functions under weaker assumptions.)
For our results on mutating theta functions, we assume \emph{signed-nondegenerating coefficients}, meaning that any mutation of~$\tB$ has nondegenerate coefficients, with consistent signs in coefficient rows.
A sufficient condition for signed-nondegenerating coefficients is that there exists a seed at which the coefficients are principal.\looseness=-1

Lemma~\ref{mut broken line} implies the following result, stated explicitly later as Theorem~\ref{2 muts}.
\begin{Theorem}
 Theta functions relative to mutation-equivalent exchange matrices with signed-nondegenerating coefficients can be obtained from one another by multiplication with an appropriate Laurent monomial in frozen variables.
\end{Theorem}

Section~\ref{app sec} contains various applications of Theorem~\ref{2 muts}.
One can think of these as applications of the notion of mutation used in this paper.

Three of these applications simplify the computation of structure constants for multiplication of theta functions.
One application (Theorem~\ref{finite orbit}) exploits mutation-symmetry.
(A~\emph{mutation-symmetry} of an exchange matrix is a sequence of mutations that preserves the exchange matrix).
Roughly, the simplification is as follows:
When we multiply theta functions indexed by vectors in finite orbits under the mutation-symmetry, the product expands as a sum of theta functions indexed by vectors in finite orbits.
Slightly more precisely, we must also assume that the theta functions being multiplied are indexed by vectors in a special subset of the ambient lattice (called $\Theta$ in \cite{GHKK}), so that the product expands as a finite sum.
Another application (Proposition~\ref{mut pair}) points out a scenario where pairs of broken lines mutate together.
A third application (Theorem~\ref{B cone prod}) shows that the product of theta functions whose $\g$-vectors are all in the same cone of the mutation fan expands as a combination of theta functions whose $\g$-vectors are in the dominance region of the $\g$-vector of the product.
(The dominance region is a formulation due to Rupel and Stella of Fan Qin's dominance partial order on $\g$-vectors \cite[Section~3.1]{FanQin}.)
Theorem~\ref{B cone prod} is a corollary of the more precise but more technical Theorem~\ref{B cone prod N}.

These simplifications in computing structure constants along with direct applications of the technical lemma are crucial in \cite{afftheta}, where we give a complete description of theta functions in the acyclic affine case, in the context of a combinatorial model of the cluster scattering diagram developed in \cite{affscat}, discover ``imaginary'' exchange relations among theta functions, and identify an ``imaginary subalgebra'' of the small canonical algebra that is related to (and in some cases isomorphic to) a generalized cluster algebra of finite type.

The remaining applications of Theorem~\ref{2 muts} have to do with \emph{pointed reduced bases} of the small canonical algebra, a variation on Qin's pointed bases \cite[Theorem~1.2.1]{FanQin}.
We give a~characterization (Theorem~\ref{point precise}) of pointed reduced bases analogous to Qin's characterization \cite[Theorem~1.2.1]{FanQin} of pointed bases.
We give both precise necessary and sufficient conditions for a~set to be a pointed reduced basis and also simpler necessary conditions, the latter being phrased in terms of dominance regions.
For more details on the relationship between \cite[Theorem~1.2.1]{FanQin} and Theorem~\ref{point precise}, see Remark~\ref{FQ remark}.
We also define, in some cases, a basis for the small canonical algebra called the \emph{ray basis}.
In the case of a marked surface, the ray basis is the bangles basis of~\cite{MSWbases}.

Some ideas in this paper are inspired by the ideas in Fan Qin's paper \cite{FanQin}, although our conventions, methods, and goals are different enough that there is little actual overlap between the two papers.

\section{Scattering diagrams and theta functions}
We now introduce background material on scattering diagrams and their theta functions.

\subsection{Context for scattering diagrams}\label{context sec}
The underlying data for a scattering diagram is a finite-dimensional lattice with a distinguished basis, a distinguished subset of the basis, and a skew-symmetric bilinear form on the lattice.
In~\cite{GHKK}, the underlying data is divided into ``fixed data'' and ``seed data''.
Later in the paper, we will (crucially) take a different point of view on what is fixed and what changes, but for now we make no distinction between fixed data and seed data.

The scattering diagram itself lives in the dual space, and involves a set of indeterminates in bijection with the basis elements.
We arrange the relevant definitions into four parts, related to: the index set; the lattice, the dual lattice; and the indeterminates.
Then we show how all of this data arises from the choice of an extended exchange matrix.

\subsubsection*{Definitions related to the index set}
\begin{itemize}\itemsep=0pt
\item A finite index set $I$.
\item $I_\uf\subseteq I$ (``unfrozen'') and $I_\fr=I\setminus I_\uf$ (``frozen'').
\item Positive integers $(d_i:i\in I)$ with $\gcd(d_i:i\in I)=1$.
\end{itemize}

\subsubsection*{Definitions related to the lattice}
\begin{itemize}\itemsep=0pt
\item A lattice $N$ with basis $(e_i:i\in I)$.
\item The finite-index sublattice $N^\circ\subseteq N$ spanned by $(d_ie_i:i\in I)$.
\item The sublattice $N_\uf$ spanned by $(e_i:i\in I_\uf)$.
\item $N^{0+}_\uf=\bigl\{\sum_{i\in I_\uf}a_ie_i:a_i\in\integers,\,a_i\ge0\bigr\}$ and $N^+_\uf=N^{0+}_\uf\setminus \{0\}$, the nonnegative and positive parts of $N_\uf$.
\item $V=N_\uf\otimes\reals$, the ambient vector space of $N_\uf$ (not of $N$).
\item $\set{\,\cdot\,,\,\cdot\,}\colon N\times N\to \rationals$, a skew-symmetric bilinear form, chosen so that $\set{N_\uf,N^\circ}\subseteq\integers$ and $\set{N,N_\uf\cap N^\circ}\subseteq\integers$.
\item $\epsilon_{ij}=\set{e_i,d_je_j}$ (integers except possibly when $i,j\in I_\fr$).
\item An element of a lattice is \emph{primitive} if it is not a positive integer multiple of some \emph{other} element of the lattice.
Given a primitive element $n\in N$, let $n^\circ$ be the primitive element of $N^\circ$ that is a positive multiple of $n$.
\item We partially order the lattice $N$ \emph{componentwise} relative to the basis $(e_i:i\in I)$.
That is, $n\le n'\in N$ if and only if $n=\sum_{i\in I}a_ie_i$ and $n'=\sum_{i\in I}a'_ie_i$ with $a_i\le a'_i$ for all $i\in I$.
\end{itemize}

\subsubsection*{Definitions related to the dual}
\begin{itemize}\itemsep=0pt
\item $M=\Hom(N,\integers)$, the dual lattice to $N$, with basis $(e^*_i:i\in I)$ dual to $(e_i:i\in I)$.
\item $M^\circ=\Hom(N^\circ,\integers)$, the finite-index superlattice of $M$ spanned by \mbox{$(f_i:i\in I)$}, where $f_i=d_i^{-1}e_i^*$.
\item \mbox{$\br{\,\cdot\,,\,\cdot\,}\colon M^\circ\times N\to\rationals$}, the natural pairing.
\item The sublattice $M^\circ_\uf$ spanned by $(f_i:i\in I_\uf)$.
\item $V^*$, the dual vector space to $V$, spanned by $\set{e^*_i:i\in I_\uf}$.
\item $\br{\,\cdot\,,\,\cdot\,}\colon V^*\times V\to\reals$, the natural pairing.
\item The elements $\sum_{j\in I}\epsilon_{ij}f_j$ for $i\in I_\uf$ are required to be linearly independent, but this is really a condition on $\set{\,\cdot\,,\,\cdot\,}$.
\item We partially order $M^\circ$ componentwise relative to the basis $(f_i:i\in I)$.
\end{itemize}

\subsubsection*{Definitions related to indeterminates}
\begin{itemize}\itemsep=0pt
\item Indeterminates $(z_i:i\in I)$.
\item $z^m$ means $\prod_{i\in I}z_i^{c_i}$ for $m=\sum_{i\in I}c_if_i\in M^\circ$.
\item Laurent monomials $\sigma_i=\prod_{j\in I_\fr}z_j^{\epsilon_{ij}}$ and $\zeta_i=\prod_{j\in I}z_j^{\epsilon_{ij}}$ for $i\in I_\uf$.
\item $\zeta^n$ means $\prod_{i\in I_\uf}\zeta_i^{a_i}$ and $\sigma^n$ means $\prod_{i\in I_\uf}\sigma_i^{a_i}$ for $n=\sum_{i\in I_\uf}a_i e_i\in N_\uf$.
\item $\k[[\zeta]]=\k[[\zeta_i:i\in I]]$, the ring of formal power series in the $\zeta_i$, over a field~$\k$ of characteristic~$0$.
\item Given $n\in N_\uf$, we write $nB$ for the vector in $M^\circ_\uf$ such that $\zeta^n=z^{nB}\sigma^n$.
In other words, we write the $e_i$-coordinate vector of $n$ as a row vector, apply the matrix $B=[\epsilon_{ij}]_{i,j\in I_\uf}$ on the right, and interpret the resulting row vector as the $f_i$-coordinate vector of $nB\in M^\circ_\uf$.
\end{itemize}

We now explain how \emph{all of the essential underlying data for a scattering diagram amounts to the choice of an extended exchange matrix.}

Given an indexing set $I$ and $I_\uf\subseteq I$, define $I_\fr$ as above.
Let $\tB=[\epsilon_{ij}]_{i\in I_\uf,j\in I}$ be an integer matrix with linearly independent rows, and write $B$ for the square submatrix $[\epsilon_{ij}]_{i,j\in I_\uf}$.
We require that there exist integers $(d_i:i\in I_\uf)$ such that $d_i\epsilon_{ij}=-d_j\epsilon_{ji}$ for all $i,j\in I_\uf$.
We can assume that $\gcd(d_i:i\in I)=1$.
Thus $B$ is a skew-symmetrizable integer matrix---an \emph{exchange matrix}---and $\tB$ is an \emph{extended exchange matrix}, extending $B$.
(Contrary to the conventions in~\cite{ca4}, here we have extended~$B$ to make it \emph{wide} rather than \emph{tall}.
In~\cite{afftheta}, we give the translations of this paper's results into the tall extended matrix setting.)
The definition of cluster scattering diagrams and theta functions below requires that the rows of $\tB$ be linearly independent, but we will see in Section~\ref{coeff sec} that there is a reasonable definition of theta functions for arbitrary $\tB$.

All of the underlying data described above can be constructed from $\tB$ as follows.
Let $(e_i:i\in I)$ be formal symbols, take $N$ to be the lattice of formal $\integers$-linear combinations of $\set{e_i:i\in I}$, and define $N^\circ$, $N_\uf$, $N^+$, and $V$ as above.

There is a choice to be made before constructing the remaining data.
Specifically, we must choose additional entries $\epsilon_{ij}$ with $i\in I_\fr$, $j\in I$ and additional integers $(d_i:i\in I_\fr)$ such that the additional entries $\epsilon_{ij}$ are integers when $j\in I_\uf$ and such that the matrix $[\epsilon_{ij}:i,j\in I]$ has $d_i\epsilon_{ij}=-d_j\epsilon_{ji}$ for all $i,j\in I$.
However, this choice is inconsequential for our purposes because the data $[\epsilon_{ij}]_{i\in I_\fr,j\in I}$ and $(d_i:i\in I_\fr)$ do not appear in the rest of the paper.
(One way to make the choice is to set $d_i=1$ for all $i\in I_\fr$, set $\epsilon_{ij}=-d_j\epsilon_{ji}$ for $i\in I_\fr$ and $j\in I_\uf$, and set $\epsilon_{ij}=0$ for $i,j\in I_\fr$.)
We emphasize again that \emph{all of the essential data is all contained in $\tB$}.

Define a bilinear form $\set{\,\cdot\,,\,\cdot\,}\colon N\times N\to \rationals$ according to the rule $\set{e_i,d_je_j}=\epsilon_{ij}$ for $i,j\in I$.
This is skew-symmetric because $d_i\epsilon_{ij}=-d_j\epsilon_{ji}$ for all $i,j\in I$.
The form has $\set{N_\uf,N^\circ}\subseteq\integers$ and $\set{N,N_\uf\cap N^\circ}\subseteq\integers$ because $\tB$ is an integer matrix.

Finally, construct $M$, $M^\circ$, and $V^*$ and define
indeterminates and Laurent monomials as above.
The requirement that the vectors $\bigl(\sum_{j\in I}\epsilon_{ij}f_j:i\in I_\uf\bigr)$ are linearly independent is precisely the linear independence of the rows of $\tB$.

The extended exchange matrix $\tB$ (and by extension, the underlying data for the cluster scattering diagram) is said to have \emph{principal coefficients} if there is a bijection $\pi\colon I_\uf\to I_\fr$ such that $\epsilon_{i\pi(j)}=\delta_{ij}$ (Kronecker delta).
Thus~$\tB$ is~$B$, extended by adjoining a permutation of an identity matrix.

\subsection{Scattering diagrams}\label{scat sec}
We assume the most basic notions of scattering diagrams from~\cite{GHKK}, and we follow the treatment~in~\cite{scatfan}.
In particular, we leave out, from the scattering diagram, the extra dimensions related to frozen variables, and modify some definitions/constructions accordingly.
(See \cite[Remarks~\mbox{2.12--2.13}]{scatfan}.)

A \emph{wall} is a pair $(\d,f_\d)$, where $\d$ is a codimension-$1$ polyhedral cone in $V^*$ and~$f_\d$ is in $\k[[\zeta]]$.
More specifically, the cone $\d$ is orthogonal to a vector $n\in N^+_\uf$ that is primitive in~$N$, and the \emph{scattering term} $f_\d$ is a univariate power series in $\zeta^n$ with constant term~$1$.
We sometimes write $(\d,f_\d(\zeta^n))$ as a way of naming $n$ explicitly.
A \emph{scattering diagram} is a collection $\D$ of walls, satisfying a finiteness condition that amounts to the requirement that all relevant computations are valid as limits of formal power series.

The \emph{wall-crossing automorphism} $\p_{\bl,\d}$ for a path $\bl$ crossing a wall $(\d,f_\d(\zeta^n))$ acts on Laurent monomials in the $z_i$ by
\begin{align*}
\p_{\bl,\d}(z^m)&=z^m f_\d^{\br{m,\pm n^\circ}},
\end{align*}
for $m\in M^\circ_\uf$, taking $+n^\circ$ when crossing $\d$ \emph{against} the direction of $n$ or taking $-n^\circ$ when crossing \emph{in} the direction of $n$.

As usual, \emph{path-ordered products} are compositions of wall-crossing automorphisms along gener\-ic paths, or more precisely, limits of such compositions in the sense of formal power series.
A~scattering diagram $\D$ is \emph{consistent} if a path-ordered product depends only on the starting point and ending point of the path, and two scattering diagrams $\D$ and $\D'$ are \emph{equivalent} if they determine the same path-ordered products.
In this paper, we do not need to compute any path-ordered products explicitly,
except in connection with the proof of Proposition~\ref{mut theta}, where we compute a~single wall-crossing automorphism.

A wall $(\d,f_\d(\zeta^n))$ is \emph{outgoing} if the vector $nB\in M^\circ_\uf$ is not contained in $\d$.
The \emph{cluster scattering diagram} $\Scat\bigl(\tB\bigr)$ is the unique (up to equivalence) consistent scattering diagram consisting of walls $\set{(e_i^\perp,1+\zeta_i): i\in I_\uf}$ together with additional walls, all of which are outgoing.
The existence and uniqueness of $\Scat\bigl(\tB\bigr)$ is \cite[Theorem~1.12]{GHKK}.

\subsection{Theta functions}\label{thet sec}
We now define theta functions, quoting and reinterpreting definitions and results of \cite{GHKK} in the special case of $\Scat\bigl(\tB\bigr)$.
Theta functions depend on the choice of a point $Q\in V^*$ and a vector $m\in M^\circ_\uf$, with $Q$ required to not be contained in any hyperplane $n^\perp$ for $n\in N_\uf$.
We have $\thet_{Q,0}=1$, and for $m\neq0$, $\thet_{Q,m}$ is a sum of weights on broken lines, as we now describe.
A \emph{broken line} for $m$ with endpoint $Q$, relative to $\Scat\bigl(\tB\bigr)$ is a piecewise linear path $\bl\colon (-\infty,0]\to V^*$ with finitely many domains of linearity (abbreviated ``domains'' in what follows), satisfying certain conditions.
The first two conditions are directly on~$\bl$:
\begin{enumerate}[label=\rm(\roman*), ref=(\roman*)]\itemsep=0pt
\item \label{brok endpoint}
$\bl(0)=Q$.
\item \label{brok generic}
$\bl$ does not intersect the relative boundary of any wall of $\Scat\bigl(\tB\bigr)$ and does not intersect any intersection of walls of $\Scat\bigl(\tB\bigr)$ (unless those walls are in the same hyperplane).
\end{enumerate}
The remaining conditions are formulated in terms of an assignment of a Laurent monomial $\const_Lz^{m_L}\sigma^{n_L}$ (with $\const_L\in\k$, $m_L\in M^\circ_\uf$, and $n_L\in N_\uf$) to each domain $L$ of~$\bl$.
The Laurent monomials must satisfy the following conditions, which amount to conditions on~$\bl$, and which allow us (if the conditions can be satisfied) to recover the Laurent monomials uniquely from the path $\bl$.
\begin{enumerate}[label=\rm(\roman*), ref=(\roman*)]\itemsep=0pt
\addtocounter{enumi}{2}
\item \label{brok slope}
In each domain $L$, the derivative $\bl'$ of $\bl$ is constantly equal to $-m_L$.
\item \label{brok unbounded}
 $\const_Lz^{m_L}\sigma^{n_L}=z^m$ when $L$ is the unbounded domain of $\bl$. (That is, $\const_L=1$, $m_L=m$, and $n_L=0$.)
\item \label{brok change slope}
Suppose $t$ is a point of nonlinearity of $\bl$, adjacent to domains $L_1$ and $L_2$, with~$L_1$ being of the form $[a,t]$ or $(-\infty,t]$ and $L_2$ being of the form $[t,b]$.
Then~$\bl(t)$ is required to be contained in some wall.
By condition \ref{brok generic}, above, there exists $n\in N_\uf$ such that every wall containing $\bl(t)$ is in $n^\perp$.
We can choose $n$ to be primitive in $N$ and to have ${\br{m_{L_1},n}>0}$.
Let $f$ be the product of the $f_\d$ for all walls $(\d,f_\d)$ with $\bl(t)\in\d$.
Then $\const_{L_2}z^{m_{L_2}}\sigma^{n_{L_2}}$ is required to be $\const_{L_1}z^{m_{L_1}}\sigma^{n_{L_1}}$ times a term in $f^{\br{m_{L_1},n^\circ}}$.
(We say that $\bl$ \emph{bends} at $t$, and
$(\const_{L_2}z^{m_{L_2}}\sigma^{n_{L_2}})/(\const_{L_1}z^{m_{L_1}}\sigma^{n_{L_1}})$ is the \emph{contribution} at this bend.)
\end{enumerate}
Since each scattering term $f_\d$ is a formal power series in $\k[[\zeta]]$, these conditions imply that each~$n_L$ is in $N_\uf^{0+}$, so that each $z^{m_L}\sigma^{n_L}$ is a Laurent monomial in $(z_i:i\in I_\uf)$ times an ordinary monomial in $(\sigma_i:i\in I_\uf)$, and in fact each $z^{m_L}\sigma^{n_L}$ is $z^m$ times an ordinary monomial in $(\zeta_i:i\in I_\uf)$.

Write $\const_\bl z^{m_\bl}\sigma^{n_\bl}$ for the Laurent monomial on the domain containing $0$.
Then the \emph{theta function} $\thet_{Q,m}$ is $\sum \const_\bl z^{m_\bl}\sigma^{n_\bl}$, where the sum is over broken lines for~$m$ with endpoint~$Q$.
We see that $\thet_{Q,m}\in z^m \k[[\zeta]]$.
(That is, $\thet_{Q,m}$ is $z^m$ times a formal power series in $(\zeta_i:i\in I_\uf)$.)

The most important theta functions, from the point of view of cluster algebras, are those such that $Q$ is chosen to be in the interior of the positive orthant $D=\bigcap_{i=1}^n\set{x\in V^*: \br{x,e_i}\ge 0}$.
The theta function does not depend on the exact choice of $Q$ in the interior of $D$, and thus we write $\thet_m$ to mean $\thet_{Q,m}$ with $Q$ in the interior of $D$.
Each theta function $\thet_m$ is $z^m$ times a formal power series $F_m\in\k[[\zeta]]$.
We will call $F_m$ the \emph{$F$-series} of $\thet_m$.
A~broken line $\bl$ for $m$ has $n_\bl=0$ if and only if $\bl$ has exactly one domain of linearity.
There is exactly one such broken line for each $m$ and $Q$, and it has $\const_\bl z^{m_\bl}\sigma^{n_\bl}=z^m$.
Thus $F_m$ has constant term $1$.

\begin{Example}\label{theta g2 ex}
(To understand this example, one must be careful to distinguish between~$n$ and~$n^\circ$ in the definitions of walls and broken lines.)
Consider the exchange matrix \smash{$B=\left[\begin{smallmatrix} 0& -3\\ 1& \hphantom{-}0\end{smallmatrix}\right]$} and assume that $\tB$ is some extension of~$B$.
We take $I_\uf=\set{1,2}$ indexing $B$ in the usual way and write $[m_1,m_2]$ for the $f_i$-coordinates of vectors in $M^\circ_\uf$.
The black lines in Figure~\ref{theta g2 fig} are the walls of the cluster scattering diagram.
The figure is drawn so that $f_1$ is the horizontal coordinate and $f_2$ is the vertical, and $f_1$ and $f_2$ are shown with the same length.
The skew-symmetrizing constants are $d_1=3$ and $d_2=1$, so $e_1^*=3f_1$ and $e_2^*=f_2$.
We have $\zeta_1=\sigma_1z_2^{-3}$ and $\zeta_2=\sigma_2z_1$.
\begin{figure}[!t]\centering
\scalebox{0.85}{
\includegraphics{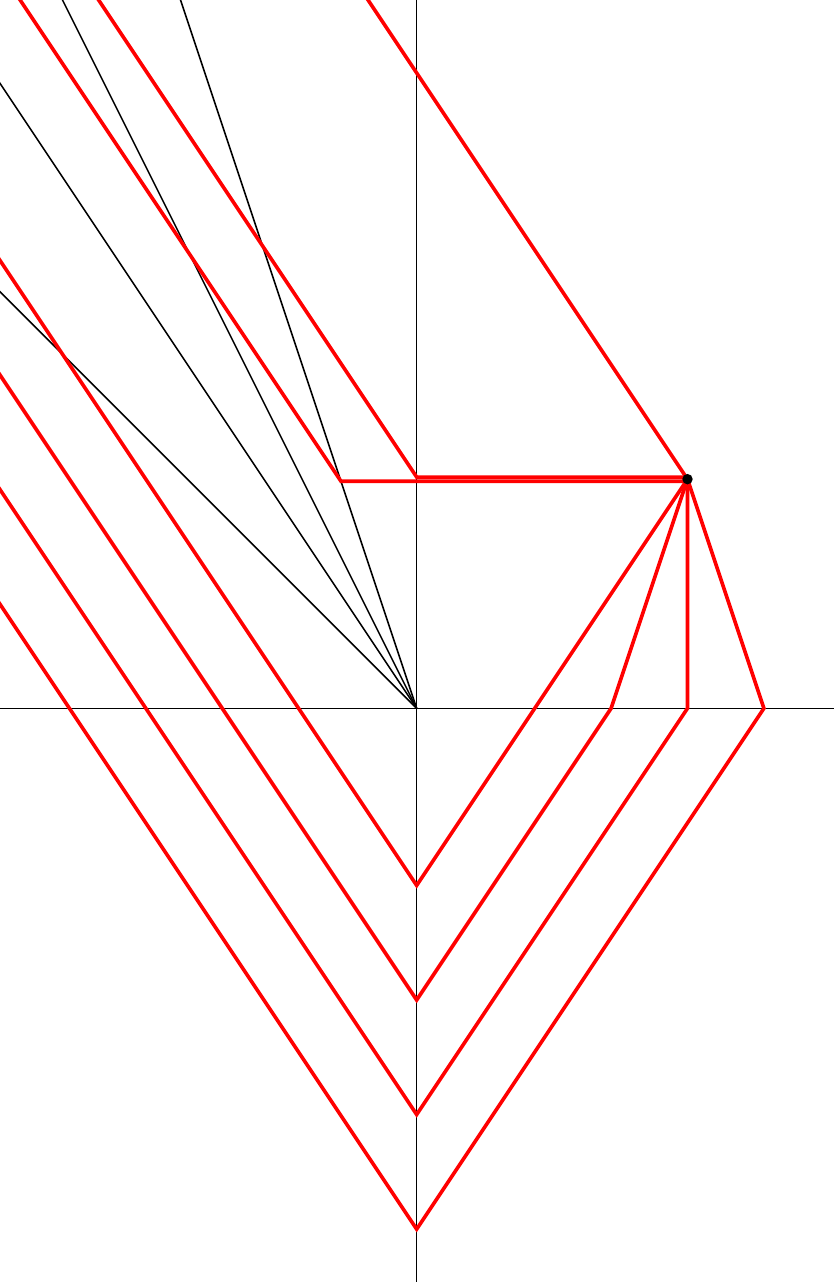}
\begin{picture}(0,0)(200,-275)
\put(-1,333){$1+\zeta_1$}
\put(-200,-10){$1+\zeta_2$}
\put(-1,-272){$1+\zeta_1$}
\put(168,-10){$1+\zeta_2$}
\put(100,158){\textcolor{red}{\rotatebox{-56}{$z_1^{-2}z_2^3$}}}
\put(40,116){\textcolor{red}{$2\sigma_1z_1^{-2}$}}
\put(44,8){\textcolor{red}{\rotatebox{56}{$\sigma_1^2z_1^{-2}z_2^{-3}$}}}
\put(78,4){\textcolor{red}{\rotatebox{71.5}{$3\sigma_1^2\sigma_2z_1^{-1}z_2^{-3}$}}}
\put(129,50){\textcolor{red}{\rotatebox{-90}{$3\sigma_1^2\sigma_2^2z_2^{-3}$}}}
\put(147,60){\textcolor{red}{\rotatebox{-71.5}{$\sigma_1^2\sigma_2^3z_1z_2^{-3}$}}}
\put(30,98){\textcolor{red}{$3\sigma_1\sigma_2z_1^{-1}$}}

\put(-42,176){\textcolor{red}{\rotatebox{-56}{$z_1^{-2}z_2^3$}}}
\put(-88,188){\textcolor{red}{\rotatebox{-56}{$z_1^{-2}z_2^3$}}}
\put(-202,219){\textcolor{red}{\rotatebox{-56}{$z_1^{-2}z_2^3$}}}
\put(-202,164){\textcolor{red}{\rotatebox{-56}{$z_1^{-2}z_2^3$}}}
\put(-202,109){\textcolor{red}{\rotatebox{-56}{$z_1^{-2}z_2^3$}}}
\put(-202,54){\textcolor{red}{\rotatebox{-56}{$z_1^{-2}z_2^3$}}}

\put(0,-111){\textcolor{red}{\rotatebox{56}{$\sigma_1^2z_1^{-2}z_2^{-3}$}}}
\put(0,-166){\textcolor{red}{\rotatebox{56}{$\sigma_1^2z_1^{-2}z_2^{-3}$}}}
\put(0,-221){\textcolor{red}{\rotatebox{56}{$\sigma_1^2z_1^{-2}z_2^{-3}$}}}

\put(-112,335){\rotatebox{-71.5}{$1+\zeta_1\zeta_2$}} 
\put(-167,335){\rotatebox{-63}{$1+\zeta_1^2\zeta_2^3$}} 
\put(-200,302){\rotatebox{-56}{$1+\zeta_1\zeta_2^2$}} 
\put(-150,153){\rotatebox{-45}{$1+\zeta_1\zeta_2^3$}} 

\put(130,114){$Q$}
\end{picture}
}
\caption{Broken lines for $\thet_{[-2,3]}=\thet_{Q,[-2,3]}$.}\label{theta g2 fig}
\end{figure}
The figure also shows, in red, the broken lines that arise in the computation of~$\thet_{[-2,3]}$, for a~particular choice of $Q$ in the positive quadrant.
The theta function is the sum of the monomials assigned to the domains incident to~$Q$:
\begin{gather*}
\thet_{[-2,3]} =
z_1^{-2}z_2^3+2\sigma_1z_1^{-2}+3\sigma_1\sigma_2z_1^{-1}+\sigma_1^2z_1^{-2}
z_2^{-3}\\
\hphantom{\thet_{[-2,3]} =}{}
+3\sigma_1^2\sigma_2z_1^{-1}z_2^{-3}+3\sigma_1^2\sigma_2^2z_2^{-3}+\sigma_1^2\sigma_2^3z_1z_2^{-3}
\\
\hphantom{\thet_{[-2,3]}}{} =z_1^{-2}z_2^3\bigl(1+2\zeta_1+3\zeta_1\zeta_2+\zeta_1^2+3\zeta_1^2\zeta_2+3\zeta_1^2\zeta_2^2+\zeta_1^2\zeta_2^3\bigr).
\end{gather*}
For further details, compare with Example~3.20 in \cite{scatcomb}.
\end{Example}

\begin{Remark}\label{useless dimensions}
As mentioned at the beginning of Section~\ref{scat sec}, we have followed \mbox{\cite[Section~2]{scatfan}} in leaving out some unnecessary dimensions from the definition of scattering diagrams and modifying the definition of theta functions accordingly.
This choice is explained and justified in \cite[Remark~2.1]{scatfan}, \cite[Remark~2.12]{scatfan}, and \cite[Remark~5.1]{scatfan}.
Leaving out the extra dimensions amounts to ``demoting'' the ``frozen variables'' $(z_i:i\in I_\fr)$ to the status of ``coefficients''.
While the difference may be a matter of taste and convenience, we find that leaving out the extra dimensions leads to better intuition on what is actually happening, for example when we change coefficients.
(See Section~\ref{coeff sec}.)

We now describe the difference more specifically.
For the purposes of this remark, temporarily define $V^*_\fr$ to be the real span of $\set{f_i:i\in I_\fr}$, so that $M\otimes \reals$ can be identified with $V^*\oplus V^*_\fr$.
Starting with $\Scat\bigl(\tB\bigr)$ in $V^*$ as defined here, we recover the scattering diagram in $M\otimes \reals$ (as in~\cite{GHKK}) by replacing each wall $(\d,f_\d)$ of $\Scat\bigl(\tB\bigr)$ with $(\d\oplus V^*_\fr,f_\d)$.
Our definition constrains broken lines to $V^*$, but their ``directions'' in $V^*\oplus V^*_\fr$ are recorded on each domain by the Laurent monomials~$z^{m_L}\sigma^{n_L}$.
Broken lines for $m\in M^\circ_\uf$ in our sense are easily seen to correspond bijectively to broken lines in the larger space in the sense of~\cite{GHKK}, and the Laurent monomials labeling domains are the same.
\end{Remark}

\begin{Remark}\label{useless theta functions}
In~\cite{GHKK}, the definition of theta functions allows $m\in M^\circ$, rather than the more restrictive condition $m\in M^\circ_\uf$ in our definition.
However, if we allow $m\in M^\circ$ in our definition and adjust Condition~\ref{brok unbounded} in the obvious way, we obtain the same theta functions as in \cite{GHKK}.
In either definition, the only interesting theta functions are the theta functions for $m\in M^\circ_\uf$.
Indeed, if $m\in M^\circ$ is written as $m_0+m_1$ with $m_0\in M^\circ_\uf$ and $m_1$ in the span of $\set{f_i:i\in I_\fr}$, then $\thet_m=\thet_{m_0}z^{m_1}$.
Thus we restrict our attention to theta functions $\thet_m$ with $m\in M^\circ_\uf$.
See also Section~\ref{GHKK compare}.
\end{Remark}

\subsection{Structure constants}\label{struct sec}
We now explain a result of \cite{GHKK} that gives the structure constants for multiplication of theta functions $\thet_m$, reinterpreting this result in our setting where unnecessary dimensions have been removed.
Specifically, we quote a result of \cite{GHKK} that is stated there under the hypothesis of principal coefficients, as defined in Section~\ref{context sec}.
In Section~\ref{coeff sec}, we will show (as Proposition~\ref{struct plus}) that the result quoted in this section holds without the hypothesis of principal coefficients.

Suppose $p_1,p_2,m\in M^\circ_\uf$ and suppose $Q\in V^*$ is not contained in any wall of $\Scat\bigl(\tB\bigr)$.
Define
\[a_Q(p_1,p_2,m)=\sum_{(\bl_1,\bl_2)}\const_{\bl_1}\const_{\bl_2} \sigma^{n_{\bl_1}+n_{\bl_2}},\]
where the sum is over all pairs $(\bl_1,\bl_2)$ of broken lines for $p_1$ and $p_2$ respectively, both having endpoint $Q$, and with $m_{\bl_1}+m_{\bl_2}=m$.
The definition of broken lines implies that each $n_{\bl_i}$ is a nonnegative combination of $\set{e_i:i\in I_\uf}$.
Each $\sigma^{n_{\bl_1}+n_{\bl_2}}$ is a Laurent monomial in $(z_i:i\in I_\fr)$.
By \cite[Definition-Lemma~6.2]{GHKK}, each Laurent monomial in $(z_i:i\in I_\fr)$ appears at most finitely many times, and with nonnegative coefficients, in the sum.
Therefore, each monomial in $(\sigma_i:i\in I_\uf)$ appears at most finitely many times and with nonnegative coefficients.
In particular, $a_Q(p_1,p_2,m)$ is well defined as a formal power series in $(\sigma_i:i\in I_\uf)$ with nonnegative integer coefficients.
For an arbitrary sequence of points $Q\in V^*$, disjoint from the walls of $\Scat\bigl(\tB\bigr)$ and approaching $m$, define
\[a(p_1,p_2,m)=\lim_{Q\to m}a_Q(p_1,p_2,m).\]
This limit is valid as a limit of formal power series, by the finiteness condition in the definition of a scattering diagram.
The following is part of \cite[Proposition~6.4]{GHKK} specialized to the present setting.

\begin{Proposition}\label{struct}
Assume principal coefficients and suppose $p_1,p_2\in M^\circ_\uf$.
Then for every $m\in M^\circ_\uf$, the formal power series $a(p_1,p_2,m)$ does not depend on the sequence of points $Q$ approaching $m$.
Furthermore,
\begin{equation}\label{struct eq}
\thet_{p_1}\cdot\thet_{p_2}=\sum_{m\in M^\circ_\uf}a(p_1,p_2,m)\,\thet_m.
\end{equation}
\end{Proposition}
The sum in Proposition~\ref{struct} may not reduce to a finite sum, but \cite[Definition-Lemma~6.2]{GHKK} and \cite[Proposition~6.4]{GHKK} imply that it makes sense as a~convergent sum of formal power series.
More precisely, for each term $c\sigma^n$ of $a(p_1,p_2,m)$, the expression $c\sigma^n\thet_m$ can be rewritten as $c\zeta^nz^{-nB}z^mF_m$.
Since $\zeta^n$ is the product of the contributions of all bends of some pair of broken lines for $p_1$ and $p_2$, we have $m-nB=p_1+p_2$.
Thus $c\sigma^n\thet_m$ is $z^{p_1+p_2}$ times a formal power series $c\zeta^nF_m\in\k[[\zeta]]$.
Proposition~\ref{struct} (i.e., \cite[Proposition~6.4]{GHKK}) says that the sum of these formal power series converges to $z^{-p_1-p_2}\thet_{p_1}\cdot\thet_{p_2}$ so that the sum in \eqref{struct eq} converges to $\thet_{p_1}\cdot\thet_{p_2}$.

Roughly following \cite[Section~7]{GHKK}, we define $\Theta\subseteq M^\circ_\uf$ to be the set of vectors $m\in M^\circ_\uf$ such that only finitely many broken lines figure into the definition of $\thet_m$.
In particular, for $m\in\Theta$, in the theta function $\thet_m=z^m\cdot F_m$, the $F$-series $F_m$ is a polynomial (the \emph{$F$-polynomial}), rather than a more general formal power series.

The following is part of \cite[Theorem~7.5]{GHKK}, rephrased to use the conventions of this paper.
\begin{Theorem}\label{Theta facts}
If $p_1,p_2\in\Theta$, then the sum in \eqref{struct eq} has finitely many nonzero terms, every $a(p_1,p_2,m)$ is a polynomial, and every $m$ with $a(p_1,p_2,m)\neq0$ has $m\in\Theta$.
\end{Theorem}

\subsection{Matrix mutation, mutation maps, and the mutation fan}\label{mat mut sec}
Of primary importance in this paper is the notion of mutation.
For each $k\in I_\uf$, the \emph{mutation} of $\tB$ in direction~$k$ is the matrix $\mu_k\bigl(\tB\bigr)=[\epsilon'_{ij}]$ given by
\begin{equation*}
\epsilon_{ij}'= \begin{cases}
-\epsilon_{ij}&\mbox{if }i=k\mbox{ or }j=k,\\
\epsilon_{ij}+[-\epsilon_{ik}]_+\epsilon_{kj}+\epsilon_{ik}[\epsilon_{kj}]_+&\mbox{otherwise.}
\end{cases}
\end{equation*}
Here, $[x]_+$ means $\max(0,x)$.
Mutation of the smaller matrix $B$ is defined by the same formula.

The notation $\kk=k_q\cdots k_1$ stands for a sequence of indices in $I_\uf$, so that $\mu_\kk$ means the composition $\mu_{k_q}\circ\mu_{k_{q-1}}\circ\cdots\circ\mu_{k_1}$.

Given an exchange matrix~$B$, there is a \emph{mutation map} $\eta_\kk^B\colon V^*\to V^*$ for each sequence $\kk$ of indices.
To compute $\eta_\kk^B(v)$ for $v=\sum_{i\in I_\uf}a_if_i\in V^*$, we extend $B$ by appending a single row $(a_i:i\in I_\uf)$ below $B$, apply $\mu_\kk$, and read off the row $(a'_i:i\in I_\uf)$ below $\mu_\kk(B)$ in the mutated matrix.
Then $\eta_\kk^B(v)$ is defined to be $\sum_{i\in I_\uf}a'_if_i\in V^*$.
The map $\eta_\kk^B\colon V^*\to V^*$ is a piecewise linear homeomorphism.
It is a composition
\begin{equation}\label{eta comp}
\eta_\kk^B=\eta^B_{k_q,k_{q-1},\dots,k_1}=\eta_{k_q}^{B_{q}}\circ\eta_{k_{q-1}}^{B_{q-1}}\circ\cdots\circ\eta_{k_1}^{B_1}
\end{equation}
of mutation maps for singleton sequences, where $B_1=B$ and $B_{i+1}=\mu_{k_i}(B_i)$ for $i=1,\dots,q$.

For $v=\sum_{i\in I_\uf}a_if_i$, the notation $\vsgn(v)$ stands for the vector $(\sgn(a_i):i\in I_\uf)$, where $\sgn(a)\in\set{-1,0,1}$ is the sign of $a$ as usual.
Two vectors $m,p\in V^*$ are \emph{$B$-equivalent} (written $m\equiv^Bp$) if and only if $\vsgn\bigl(\eta^B_\kk(m)\bigr)=\vsgn\bigl(\eta^B_\kk(p)\bigr)$ for all sequences $\kk$.
The $\equiv^B$-equivalence classes are called \emph{$B$-classes} and the closures of $B$-classes are called \emph{$B$-cones}.
Each $B$-cone is a~closed convex cone (i.e., it is closed in the usual sense and closed under nonnegative scaling and addition).
The set consisting of all $B$-cones and their faces is a complete fan called the \emph{mutation fan}~$\F_{B}$ \cite[Theorem~5.13]{universal}.

\subsection{The choice of coefficients}\label{coeff sec}
We take the point of view that the exchange matrix $B$ is the most important initial data for a cluster scattering diagram, while the remaining entries of $\tB$ are a choice of ``coefficients''.
We now discuss to what extent, and in what way, the cluster scattering diagram, its theta functions, and their structure constants depend on the choice of coefficients.
This discussion leads us to introduce the hypothesis on coefficients that is needed for our main results.
In the other direction, this discussion leads us to define some theta functions in greater generality by relaxing the requirement that $\tB$ has linearly independent rows.
We describe in Proposition~\ref{struct plus plus} how results on structure constants for multiplying theta functions (e.g., those in Sections~\ref{sym sec} and~\ref{dom sec}) can apply to this more general definition of theta functions.

If $\tB$ and $\tB'$ both extend $B$ (and both have linearly independent rows), then one can construct the underlying data as in Section~\ref{context sec} for both $\tB$ and $\tB'$ with the same $(e_i:i\in I_\uf)$ both times, so that the lattices $N_\uf$ and $M^\circ_\uf$ and the vector spaces $V$ and $V^*$ are the same for $\tB$ and $\tB'$.
Distinguishing the indeterminates and Laurent monomials by placing primes on the indeterminates and Laurent monomials for $\tB'$, \mbox{\cite[Proposition~2.6]{scatfan}} is the statement that a consistent scattering diagram for $\tB$ can be made into a consistent scattering diagram for $\tB'$ by replacing each wall $(\d,f_\d(\zeta^{n_0}))$ by $(\d,f_\d((\zeta')^{n_0}))$, where $f_\d((\zeta')^{n_0})$ denotes the formal power series in $(\zeta')^{n_0}$ obtained from $f_\d(\zeta^{n_0})$ by replacing $\zeta^{n_0}$ with $(\zeta')^{n_0}$ throughout.
As an immediate consequence, the cluster scattering diagram $\Scat(\tB')$ is obtained from $\Scat\bigl(\tB\bigr)$ by making the same replacements of walls.

Now, looking at the definition of broken lines, we might try to conclude that the theta functions for $\Scat(\tB')$ are obtained from the theta functions for \smash{$\Scat\bigl(\tB\bigr)$} by replacing~$\zeta_i$ by~$\zeta'_i$ for each $i\in I_\uf$.
We might also try to conclude that the structure constants (given in Proposition~\ref{struct}/Proposition~\ref{struct plus}) for multiplying theta functions for \smash{$\Scat\bigl(\tB'\bigr)$} are obtained from the structure constants for \smash{$\Scat\bigl(\tB\bigr)$} by replacing each~$\sigma_i$ by~$\sigma'_i$.
However, there is a subtlety in the notion of ``replacing each $\zeta_i$ (or $\sigma_i$) by $\zeta'_i$ (or~$\sigma'_i$)'' that must be addressed.
This subtlety leads to our hypothesis on coefficients.

As mentioned just after Proposition~\ref{struct}, the definition via broken lines describes a theta function $\thet_m$ as $z^m$ times a formal power series in $(\zeta_i:i\in I_\uf)$.
Taking that expression, we can replace each $\zeta_i$ by $\zeta'_i$ to obtain the corresponding theta function for \smash{$\Scat\bigl(\tB'\bigr)$}.
However, the Laurent monomials $(\sigma_i:i\in I_\uf)$ may not be linearly independent in the lattice of Laurent monomials.
If they are not, then, when we write $\thet_m$ as a sum of Laurent monomials in $(z_i:i\in I)$, we may not be able to look at a single Laurent monomial $x$ and write it uniquely as $z^m$ times a monomial in $(\zeta_i:i\in I_\uf)$.
(Nevertheless, if we know $m$, then since $\zeta_i=z^{\sum_{j\in I}\epsilon_{ij}f_j}$ and the vectors $\sum_{j\in I}\epsilon_{ij}f_j$ are linearly independent, we can divide $x$ by $z^m$ and uniquely write the result as a monomial in $(\zeta_i:i\in I_\uf)$.)

The situation is even worse for structure constants.
Each $a(p_1,p_2,m)$ is a formal power series in $(\sigma_i:i\in I_\uf)$, and if we obtain a structure constant $a(p_1,p_2,m)$ for $\Scat\bigl(\tB\bigr)$ in this form as in Proposition~\ref{struct}, we can replace each $\sigma_i$ by $\sigma_i'$ to obtain the structure constants for theta functions for $\Scat\bigl(\tB'\bigr)$.
However, if we have $a(p_1,p_2,m)$ as a sum of Laurent monomials in $(z_i:i\in I)$ and if the Laurent monomials $(\sigma_i:i\in I_\uf)$ are not linearly independent, then we cannot uniquely recover monomials in $(\sigma_i:i\in I_\uf)$ and thus cannot ``replace each $\sigma_i$ by $\sigma'_i$'' to find the corresponding structure constant for \smash{$\Scat\bigl(\tB'\bigr)$}.

These considerations motivate an additional condition on $\tB$.
We say that $\tB$ has \emph{nondegenerate coefficients} if, when $B$ is deleted from $\tB$, what remains has linearly independent rows.
Equivalently, if $\sigma^n=1$ for some $n\in N_\uf$, then $n=0$.
This is stronger than requiring that the rows of $\tB$ are linearly independent.

Defining Laurent monomials $\zeta_i$ and $\zeta'_i$ for $i\in I_\uf$ as before, and appealing as suggested above to \cite[Proposition~2.6]{scatfan} and to the definition of theta functions, the replacements indicated in the following proposition are well defined, because of the hypothesis of nondegenerate coefficients.\looseness=1

\begin{Proposition}\label{where prin}
Suppose $\tB$ and $\tB'$ are extensions of~$B$, both with linearly independent rows, and suppose that $\tB$ has nondegenerate coefficients.
Then each theta function~$\thet'_m$ defined in terms of $\Scat\bigl(\tB'\bigr)$ is obtained from the theta function $\thet_m$ defined in terms of $\Scat\bigl(\tB\bigr)$ by replacing each~$\zeta_i$ by~$\zeta'_i$ $($or equivalently by replacing each~$\sigma_i$ by~$\sigma'_i)$.
\end{Proposition}

\begin{Remark}\label{not afoul}
Readers familiar with the ``separation of additions'' formula \cite[Corollary~6.3]{ca4} may be suspicious of Proposition~\ref{where prin}.
The source of confusion is that, under our conventions, theta functions ``are'' cluster monomials only in the case of principal coefficients.
For other choices of coefficients, a cluster monomial may be a theta function times a monomial in the frozen variables.
For details, see \cite[Theorem~5.2]{scatfan} and \cite[Remark~5.3]{scatfan}.
\end{Remark}

We pause to point out another direct consequence of \cite[Proposition~2.6]{scatfan} and the definition of broken lines.

\begin{Proposition}\label{Theta B}
The set $\Theta$ of vectors $m\in M^\circ_\uf$ such that only finitely many broken lines figure into the definition of $\thet_m$ depends only on $B$, not on the choice of an extension $\tB$ with linearly independent rows.
\end{Proposition}

As an immediate consequence of Proposition~\ref{where prin}, we can now keep the promise that we made earlier to weaken the hypotheses of Proposition~\ref{struct}, which was proved in \cite{GHKK} under the hypothesis of principal coefficients.
Since principal coefficients are in particular nondegenerate, Proposition~\ref{where prin} applies when $\tB$ has principal coefficients.
Since theta functions for $\tB'$ can be obtained from theta functions for $\tB$ by a well-defined substitution, the same is true for structure constants.

\begin{Proposition}\label{struct plus}
Proposition~{\rm \ref{struct}} holds, more generally, without the hypothesis of principal coefficients $($but retaining the hypothesis that the rows of~$\tB$ are linearly independent$)$.
\end{Proposition}

In Proposition~\ref{struct plus}, we must retain the hypothesis that the rows of $\tB$ are linearly independent because we may not be able to define theta functions in the case where the rows of $\tB$ are not independent.
However, Proposition~\ref{where prin} suggests how we might attempt to define theta functions for an arbitrary extension $\tB$ of $B$, without the requirement of linearly independent rows:
We find any extension $\tB'$ of~$B$ with nondegenerate coefficients, define theta functions for those nondegenerate coefficients, and then obtain theta functions for $\tB$ by the replacement described above.
This is independent of the choice of $\tB'$ in light of Proposition~\ref{where prin}.
However, it may fail for a different reason, namely because some specializations of formal power series are not well defined.
Therefore, we extend the definition of a \emph{theta function} $\thet_m$ to arbitrary extensions $\tB$ in this way only when $m\in\Theta$.

\begin{Proposition}\label{struct plus plus}
Suppose $\tB$ and $\tB'$ are extensions of $B$ such that $\tB$ has nondegenerate coefficients, and write $a(p_1,p_2,m)$ for the structure constants for theta functions defined in terms of $\tB$.
Assume either that $\tB'$ has linearly independent rows or that the vectors $p_1$ and $p_2$ are in~$\Theta$.
Then replacing each $\thet$ by $\thet'$ in \eqref{struct eq} and replacing each $\sigma_i$ by $\sigma'_i$ in each $a(p_1,p_2,m)$ yields a valid relation among theta functions $\thet'$ defined in terms of $\tB'$.
\end{Proposition}

For some results, we need an even stronger condition than nondegenerate coefficients.
 The main results of this paper describe how theta functions change when we mutate $\tB$.
We will also need to use Proposition~\ref{where prin} on exchange matrices obtained from $\tB$ by mutation.
However, it is possible that $\tB$ has nondegenerate coefficients but some mutation of $\tB$ fails to have nondegenerate coefficients.
For example, \smash{$\tB=\left[\begin{smallmatrix}0&-1&-1&\hphantom{-}1\\1&\hphantom{-}0&\hphantom{-}0&-1\end{smallmatrix}\right]$} has nondegenerate coefficients, but \smash{$\mu_1\bigl(\tB\bigr)=\left[\begin{smallmatrix}\hphantom{-}0& 1& 1&-1\\-1&0&0&\hphantom{-}0\end{smallmatrix}\right]$} does not.
Thus we want $\tB$ to have \emph{nondegenerating coefficients}, meaning that $\mu_\kk\bigl(\tB\bigr)$ has nondegenerate coefficients for every sequence~$\kk$ of indices (including the empty sequence).

To simplify mutation formulas, we also want a certain sign condition on each $\mu_\kk\bigl(\tB\bigr)$:
We say that $\tB$ has \emph{signed-nondegenerating coefficients} if, for every sequence $\kk$ of indices (including the empty sequence), the submatrix \smash{$\bigl[\epsilon^{(\kk)}_{ij}\bigr]_{i\in I_\uf, j\in I_\fr}$} of \smash{$\mu_\kk\bigl(\tB\bigr)=\bigl[\epsilon^{(\kk)}_{ij}\bigr]_{i\in I_\uf, j\in I}$} has linearly independent rows, and each row has a sign, meaning that it consists of either nonnegative entries or nonpositive entries.
Equivalently, writing \smash{$\bigl(\sigma^{(\kk)}_i:i\in I_\uf\bigr)$} for the Laurent monomials associated to $\mu_\kk\bigl(\tB\bigr)$, each \smash{$\sigma^{(\kk)}_i$} is either an ordinary monomial in $(z_i:i\in I_\fr)$ or an ordinary monomial in \smash{$\bigl(z_i^{-1}:i\in I_\fr\bigr)$}.
As a motivating example of this definition, if~$\tB$ is obtained by any sequence of mutations from an extended exchange matrix with principal coefficients, then $\tB$ has signed-nondegenerating coefficients.
(Up to a transpose, this is ``sign-coherence of $\c$-vectors'' \cite[Corollary~5.5]{GHKK}, combined with ``reciprocity between $\c$-vectors and $\g$-vectors \cite[Theorem~1.2]{Nakanishi11a} and ``linear independence of $\g$-vectors'', which follows from the construction of the $\g$-vector fan in~\cite{GHKK}.)

If $\tB$ has signed-nondegenerating coefficients, then for each $i\in I_\uf$, we write $\sgn(\sigma_i)\in\set{\pm1}$ for the sign (nonnegative or nonpositive) of the exponents in $\sigma_i$ (the sign of entries $(\epsilon_{ij}:j\in I_\fr)$).

The Laurent monomials $(\sigma_i:i\in I_\uf)$ span a sublattice of the lattice of Laurent monomials in $(z_i:i\in I_\fr)$.
Under the assumption of signed-nondegenerating coefficients, as an easy consequence of the definition of matrix mutation, the Laurent monomials \smash{$\bigl(\sigma^{(\kk)}_i:i\in I_\uf\bigr)$} span (and are a basis of) the same sublattice, independent of which sequence~$\kk$ of indices in $I$ is chosen.

\begin{Remark}\label{how to make snd coeffs}
The notion of signed-nondegenerating coefficients is closely related to the mutation fan.
(See Section~\ref{mat mut sec}.)
In light of \cite[Proposition~5.3]{universal} and \cite[Proposition~5.30]{universal}, choosing signed-nondegenerating coefficients for $B$ means choosing a full-dimensional cone $C$ of the mutation fan for $B^{\mathsf{T}}$, choosing vectors in $C$ that span the ambient space, and using those vectors as columns to extend~$B$.
It is algebraically tidier to choose a set that not only spans but is a~basis for the ambient space.
As explained in Section~\ref{GHKK compare}, choosing a larger spanning set only adds additional indeterminates that are uninteresting in the small canonical algebra.
\end{Remark}

\section{The small canonical algebra}\label{canon sec}
In this section, we define the small canonical algebra associated to $\tB$.
For a comparison of our definition with the definition in \cite{GHKK}, see Section~\ref{GHKK compare}.
Briefly, the small canonical algebra is the algebra generated by all theta functions.
We now define it more carefully.

\subsection{Definition of the small canonical algebra}\label{def can sec}
Assume that $\tB$ has linearly independent rows.
Let $\k(\!(\zeta)\!)=\bigcup_{n\in N_\uf}\zeta^n\k[[\zeta]]$ be the ring of formal Laurent series in $(\zeta_i:i\in I_\uf)$.
The elements of $\k(\!(\zeta)\!)$ are the formal series $\sum_{n\in N_\uf}c_n\zeta^n$ such that $\set{n\in N_\uf:c_n\neq0}$ has a componentwise lower bound in $N_\uf$.
(The lower bound need not be an element of $\set{n\in N_\uf:c_n\neq0}$.)
The map $n\mapsto\zeta^n$ is one-to-one because $\tB$ has linearly independent rows, so each element of $\k(\!(\zeta)\!)$ is well defined as a formal series of Laurent monomials in $(z_i:i\in I)$ with coefficients in $\k$, indexed by $N$.

Let $\k\bigl[z^{\pm1}\bigr](\!(\zeta)\!)$ be the set of formal series $\sum_{n\in N_\uf}c_n\zeta^n$, with $c_n\in\k$, that can be written (not necessarily uniquely) as a finite sum of elements of the form
\[z^m\cdot(\text{formal Laurent series in }(\zeta_i:i\in I_\uf))\]
with $m$ varying in $M^\circ_\uf$.
The set $\k\bigl[z^{\pm1}\bigr](\!(\zeta)\!)$ is closed under the obvious addition of series of Laurent monomials in $(z_i:i\in I)$.
We can also multiply elements of $\k\bigl[z^{\pm1}\bigr](\!(\zeta)\!)$:
Since the rows of~$\tB$ are linearly independent, each monomial $z^n$ with $n\in N$ appears at most finitely many times as a term in a finite product of elements of $\k\bigl[z^{\pm1}\bigr](\!(\zeta)\!)$.
Furthermore, this multiplication can be written in terms of the usual multiplication of Laurent monomials in $(z_i:i\in I)$ and of formal Laurent series to show that $\k\bigl[z^{\pm1}\bigr](\!(\zeta)\!)$ is closed under multiplication.
We see that $\k\bigl[z^{\pm1}\bigr](\!(\zeta)\!)$ is a ring, and since each $\sigma_i$ is \smash{$\zeta_i\prod_{j\in I_\uf} z_j^{-\epsilon_{ij}}$}, it is also a $\k\bigl[\sigma^{\pm1}\bigr]$-algebra, where $\k\bigl[\sigma^{\pm1}\bigr]$ is the ring of Laurent polynomials in $(\sigma_i:i\in I_\uf)$.

Continuing to assume that the rows of $\tB$ are linearly independent, we define the \emph{small canonical algebra} $\can\bigl(\tB\bigr)$ to be the $\k\bigl[\sigma^{\pm1}\bigr]$-subalgebra of $\k\bigl[z^{\pm1}\bigr](\!(\zeta)\!)$ generated by theta functions $\thet_m$ for $m\in M^\circ_\uf$.
Thus $\can\bigl(\tB\bigr)$ is the set of finite $\k\bigl[\sigma^{\pm1}\bigr]$-linear combinations of finite products of theta functions.
(See Remark~\ref{Theta all OK} for a definition of $\can\bigl(\tB\bigr)$ for certain $B$, with no conditions on the extension~$\tB$.)

\begin{Remark}\label{invert y}
In the definition of $\can\bigl(\tB\bigr)$, we have inverted the elements $(\sigma_i:i\in I_\uf)$.
This is convenient because it will allow us to write elements of $\can\bigl(\tB\bigr)$ as elements of $\can\bigl(\tB'\bigr)$, where~$\tB'$ is obtained from~$\tB$ by mutation (Theorem~\ref{2 muts}).
In other contexts, one may want a~smaller algebra: the $\k[\sigma]$-subalgebra of $\k\bigl[z^{\pm1}\bigr](\!(\zeta)\!)$ generated by theta functions.
The proofs of the results in Section~\ref{sym sec} on structure constants rely on the ability to mutate the initial matrix~$\tB$.
However, these results imply the analogous results in the smaller algebra, because the structure constants $a(p_1,p_2,m)$ in~\eqref{struct eq} are formal power series in $(\sigma_i:i\in I_\uf)$ and don't involve inverses of the~$\sigma_i$.
\end{Remark}

We emphasize that the multiplication in $\can\bigl(\tB\bigr)$ is the obvious multiplication of Laurent monomials in $(z_i:i\in I_\uf)$ and Laurent series in $(\zeta_i:i\in I_\uf)$.
However, the multiplication is also given by \eqref{struct eq}.
We will say that~$B$ is \emph{well behaved} if the sum in \eqref{struct eq} is finite for all $p_1$ and $p_2$ and the coefficients $a(p_1,p_2,m)$ are polynomials (rather than formal power series) for all $p_1$, $p_2$, and $m$.
In this case, $\can\bigl(\tB\bigr)$ is the set of finite $\k$-linear combinations of elements~$\sigma^n\thet_m$ for~$(n,m)\in N_\uf\oplus M^\circ_\uf$.
As mentioned just after Proposition~\ref{struct}, $\tB$ can fail to be well behaved.
The well behaved case is the case where the set $\Theta$ defined in Section~\ref{struct sec} is all of~$M^\circ_\uf$.
In light of Proposition~\ref{Theta B}, the property of being well behaved depends on $B$, not on the choice of extension~$\tB$.
We will see in Remark~\ref{Theta all OK} that in the well-behaved case, we can define the canonical algebra for arbitrary extensions $\tB$, with no requirement of linearly independent rows.

In order to define notions of convergence and basis in the small canonical algebra, aside from Remark~\ref{Theta all OK}, \textit{for the rest of Section~{\rm \ref{canon sec}}, we assume that~$\tB$ has nondegenerate coefficients} in the sense of Section~\ref{coeff sec}.
Under that assumption, the set $\set{z_i:i\in I_\uf}\cup\set{\zeta_i:i\in I_\uf}$ is linearly independent.

Consider the associative $\k$-algebra $\bigoplus_{m\in M^\circ_\uf}z^m\k(\!(\zeta)\!)$.
For each $m\in M^\circ_\uf$, the summand $z^m\k(\!(\zeta)\!)$ is a $\k$-vector space consisting of expressions of the form
\[z^m\cdot(\text{formal Laurent series in }(\zeta_i:i\in I_\uf)).\]
The direct sum $\bigoplus_{m\in M^\circ_\uf}z^m\k(\!(\zeta)\!)$ is a $\k$-algebra with the obvious multiplication (multiplying the Laurent monomials in $z$ and multiplying the formal Laurent series).
Because of nondegenerate coefficients, the algebra $\k[z^{\pm1}](\!(\zeta)\!)$ is naturally isomorphic to $\bigoplus_{m\in M^\circ_\uf}z^m\k(\!(\zeta)\!)$.
Otherwise, two different elements of $\bigoplus_{m\in M^\circ_\uf}z^m\k(\!(\zeta)\!)$ could represent the same series of Laurent monomials in $(z_i:i\in I)$.

The algebra $\bigoplus_{m\in M^\circ_\uf}z^m\k(\!(\zeta)\!)$ is graded by the lattice $M^\circ_\uf$, with each direct summand $z^m\k(\!(\zeta)\!)$ having degree $m$.
We call this the \emph{$\g$-vector grading} (using the name from \cite{ca4} despite the difference in conventions).
If $\v$ is an element of $z^m\k(\!(\zeta)\!)$, then we write $\g(\v)=m$.
Thus $\g(z^m)=m$ for all $m\in M^\circ_\uf$ and $\g(\zeta^n)=0$ for all $n\in N_\uf$.
Since each $\sigma_i$ is $\zeta_i\cdot\prod_{j\in I_\uf} z_j^{-\epsilon_{ij}}$, setting $\g(\zeta_i)=0$ amounts to setting $\g(\sigma_i)=-\sum_{j\in I_\uf}\epsilon_{ij}f_j$, so that $\g(\sigma^n)=-nB$ for all $n\in N_\uf$.
This definition of the $\g$-vector would be problematic without the assumption of nondegenerate coefficients, which allows us to independently assign a $\g$-vector to each $\sigma_i$.

There is a natural notion of convergence that makes $\bigoplus_{m\in M^\circ_\uf}z^m\k(\!(\zeta)\!)$ into a complete vector space over $\k$.
First, we say that a sequence of formal Laurent series in $\k(\!(\zeta)\!)$ converges if (1) there exists a vector in $N_\uf$ that is simultaneously a componentwise lower bound on the exponent vectors of all nonzero terms of all Laurent series in the sequence and (2) for every $n\in N_\uf$, the coefficient of $\zeta^n$ eventually stabilizes in the sequence.

Now, given $\v=\sum_{m\in M^\circ_\uf}z^m H_m$ in $\bigoplus_{m\in M^\circ_\uf}z^m\k(\!(\zeta)\!)$, let the symbol $S(\v)$ stand for the finite set $\set{m\in M^\circ_\uf: H_m\neq0}$.
A sequence $\v_0,\v_1,\dots$ of elements of $\bigoplus_{m\in M^\circ_\uf}z^m\k(\!(\zeta)\!)$ \emph{converges} if the set $S=\bigcup_{k\ge0}S(\v_k)$ is finite and, writing $\v_k=\sum_{m\in S}z^m H_{k,m}$ for each $k\ge 0$, the sequence $H_{0,m},H_{1,m},\dots$ converges in $\k(\!(\zeta)\!)$ for each $m\in S$.
It is easy to see that if a sequence converges then every rearrangement of the sequence also converges.
Taking the usual definition of an infinite sum as the limit of the sequence of partial sums, we note also that if a sum converges then every rearrangement of the sum converges.
Thus we will discuss convergence of sums indexed by countable sets, without specifying an enumeration of the indexing set.

The following is an immediate consequence of Proposition~\ref{where prin}.

\begin{Corollary}\label{where prin cor}
Suppose $\tB$ and $\tB'$ are extensions of $B$, both with nondegenerate coefficients.
Then the isomorphism from $\bigoplus_{m\in M^\circ_\uf}z^m\k(\!(\zeta)\!)$ to $\bigoplus_{m\in M^\circ_\uf}z^m\k(\!(\zeta')\!)$ that fixes each $z_i$ and sends each $\sigma_i$ to $\sigma'_i$
\begin{enumerate}[label=\rm(\arabic*), ref=\arabic*]\itemsep=0pt
\item
restricts to an isomorphism from $\can\bigl(\tB\bigr)$ to $\can\bigl(\tB'\bigr)$, and
\item
sends convergent series to convergent series.
\end{enumerate}
\end{Corollary}

\begin{Remark}\label{Theta all OK}
When $B$ has the property that $\Theta$ is all of $M^\circ_\uf$, Theorem~\ref{Theta facts} and Corollary~\ref{where prin cor} allow us to define a small canonical algebra for an arbitrary extension~$\tB$, with no requirement of linearly independent rows:
Given any extension~$\tB'$ of $B$ with nondegenerate coefficients, the \emph{small canonical algebra} $\can\bigl(\tB\bigr)$ is the algebra obtained from $\can\bigl(\tB'\bigr)$ by replacing each~$\sigma'_i$ with~$\sigma_i$.
This makes sense because the definition of $\Theta$ and Theorem~\ref{Theta facts} together imply that the elements of $\can\bigl(\tB'\bigr)$ are all \emph{finite} sums of Laurent monomials in $(z_i:i\in I)$.
This definition of $\can\bigl(\tB\bigr)$ is also independent of the choice of extension~$\tB'$ because of Corollary~\ref{where prin cor}.
The small canonical algebra $\can\bigl(\tB\bigr)$ is the $\k[\sigma^{\pm1}]$-algebra generated by the theta functions $\thet_m$ for~$\tB$, which were generalized in Section~\ref{coeff sec} to lift the requirement that $\tB$ have linearly independent rows when $m\in\Theta$.
\end{Remark}

\subsection[Comparison with the Gross--Hacking--Keel--Kontsevich approach]{Comparison with the Gross--Hacking--Keel--Kontsevich approach}\label{GHKK compare}
We now discuss two differences between our construction and the construction of Gross, Hacking, Keel, and Kontsevich~\cite{GHKK} (called ``GHKK'' in this section).
We also explain why, and to what extent, our results apply to the GHKK construction.
(The constructions and results of GHKK vary as to hypotheses such as principal coefficients, etc.
In this short comparison of our construction with GHKK, we will not carefully distinguish the various hypotheses of the GHKK results, but rather refer the reader to \cite{GHKK} for such details.)

The first and most significant difference is that GHKK define a canonical algebra only in the well behaved case, and under the hypothesis that $\tB$ has linearly independent rows.
They define the canonical algebra as the vector space of formal linear combinations of theta functions and give it an algebra structure using the structure constants determined in \cite[Definition-Lemma~6.2]{GHKK} and \cite[Proposition~6.4]{GHKK} (adapted here in Section~\ref{struct sec}).
The algebra they define is finitely generated and contains the upper cluster algebra (in the sense of~\cite{ca3}).
See \cite[Theorem~0.12]{GHKK}.

In the general (not necessarily well behaved) case, GHKK also define a smaller algebra called the \emph{middle cluster algebra}.
They consider the subset $\Theta$ of $M^\circ$ discussed in Section~\ref{struct sec}.
Theorem~\ref{Theta facts} implies that the vector space of formal linear combinations of the theta functions $\thet_m$ such that $m$ is in $\Theta$ has an algebra structure given by the structure constants $a(p_1,p_2,m)$, which are polynomials.
This is the middle cluster algebra.
It contains the \emph{ordinary cluster algebra} (the usual cluster algebra generated by the cluster variables) and is contained in the upper cluster algebra.
See \cite[Theorem~0.3]{GHKK}.

By contrast, here we define a canonical algebra in every case (well behaved or not), with the hypothesis of nondegenerate coefficients.
We can do this because, at worst, each theta function~$\thet_m$ is a Laurent monomial~$z^m$ times a formal power series in $\k[[\zeta]]$, and there is no obstacle to multiplying such things in $\bigoplus_{m\in M^\circ_\uf}z^m\k(\!(\zeta)\!)$.
In the well-behaved case, we also define a canonical algebra with no hypotheses on the extension of $\tB$.

Outside of the well-behaved case, our $\can\bigl(\tB\bigr)$ can have severe disadvantages as an algebra, related to the fact that the structure constants $a(p_1,p_2,m)$ can be formal power series rather than polynomials and, worse, the fact that the sum in \eqref{struct eq} can have infinitely many nonzero terms.
We would like for the theta functions to be a basis for $\can\bigl(\tB\bigr)$, but since products of theta functions may expand as infinite sums of theta functions, we can only hope to have a~basis of theta functions with respect to the notion of convergent sums defined above.
However, as explained later in Remark~\ref{completeness problem}, because $\can\bigl(\tB\bigr)$ is not complete in the topology implied by our notion of convergence, the theta functions may still fall short of what we would want from a~basis, even in this sense of convergent sums.

Despite these drawbacks, it is useful to define $\can\bigl(\tB\bigr)$, the ``algebra generated by theta functions'' with no restrictions on~$B$, and define an imperfect notion of ``basis'' later in Section~\ref{bases sec}.
The main advantage is that we can prove results about structure constants for theta functions without worrying about whether~$B$ is well behaved.
In the well behaved case, our results in particular describe structure constants for the canonical algebra in the sense of GHKK.
Furthermore, our results describe structure constants in the middle cluster algebra, whether or not~$\tB$ is well behaved.

The second difference between our construction and the GHKK construction is that we have in fact defined an algebra that is smaller than the GHKK canonical algebra (in the well behaved case):
To recover the GHKK ``large'' canonical algebra from our small canonical algebra $\can\bigl(\tB\bigr)$, one may have to put in some additional coefficients.
Strictly from the point of view of defining an algebra generated by theta functions, these additional coefficients are uninteresting, but from another point of view, they are important: The small canonical algebra might need to be enlarged in order to contain the cluster algebra associated to $\tB$.
That is because, as mentioned in Remark~\ref{not afoul}, a cluster monomial may not equal the corresponding theta function, but rather may be the theta function times a monomial in the frozen variables.
(Cluster monomials are equal to theta functions, for example, in the case of principal coefficients and in the coefficient-free case.)

As already discussed in Remark~\ref{useless theta functions}, the construction in \cite{GHKK} provides a theta function $\thet_m$ for any $m\in M^\circ$ (rather than only for $m\in M^\circ_\uf$), but if we write $m$ as $m_0+m_1$ with $m_0\in M^\circ_\uf$ and $m_1$ in the span of $\set{f_i:i\in I_\fr}$, then $\thet_m=\thet_{m_0}z^{m_1}$.
The GHKK canonical algebra is (when $B$ is well behaved) the $\k$-subalgebra of $\bigoplus_{m\in M^\circ}z^m\k(\!(\zeta)\!)$ generated by theta functions $\thet_m$ for $m\in M^\circ$.
When $|I_\fr|>|I_\uf|$, our definition of $\can\bigl(\tB\bigr)$ as the $\k\bigl[\sigma^{\pm1}\bigr]$-algebra generated by $\set{\thet_m:m\in M^\circ_\uf}$ is strictly smaller than the GHKK canonical algebra.
But the larger algebra is easily obtained by adjoining some additional Laurent monomials in $(z_i:i\in I_\fr)$.
Specifically, we choose $|I_\fr|-|I_\uf|$ of the indeterminates $(z_i:i\in I_\fr)$ in such a way that the $(\sigma_i:i\in I_\uf)$ and the additional indeterminates $z_i$ span the entire lattice of Laurent monomials in $(z_i:i\in I_\fr)$.
The GHKK canonical algebra is the tensor product over $\k$ of $\can\bigl(\tB\bigr)$ with the ring of Laurent polynomials in these additional indeterminates.

\subsection{Bases and reduced bases for the small canonical algebra}\label{bases sec}
We continue to assume that $\tB$ has nondegenerate coefficients.
A countable subset~$\overline{\U}$ of \smash{$\can\bigl(\tB\bigr)$} is a \emph{basis} for \smash{$\can\bigl(\tB\bigr)$} if, for every \smash{$\v\in \can\bigl(\tB\bigr)$}, there is a unique function $\u\mapsto a_\u$ from $\overline{\U}$ to $\k$ such that $\sum_{\u\in\U} a_\u\u$ converges to $\v$, in the sense of convergence in $\bigoplus_{m\in M^\circ_\uf}z^m\k(\!(\zeta)\!)$ described in Section~\ref{def can sec}.

\begin{Remark}\label{completeness problem}
Our definition of a basis for \smash{$\can\bigl(\tB\bigr)$} has an important drawback:
The vector space $\can\bigl(\tB\bigr)$ is not complete under convergence in \smash{$\bigoplus_{m\in M^\circ_\uf}z^m\k(\!(\zeta)\!)$}.
(Indeed, one can easily~convince oneself that the completion of \smash{$\can\bigl(\tB\bigr)$} is all of \smash{$\bigoplus_{m\in M^\circ_\uf}z^m\k(\!(\zeta)\!)$}.)
Thus a basis, as defined above, is not a so-called \emph{Schauder basis} for \smash{$\can\bigl(\tB\bigr)$}.
Nevertheless, this notion of basis is useful.
(By analogy, it is useful to describe invertible linear transformations in an $n$-dimensional vector space by giving matrix entries, even though not every choice of $n^2$ entries describes an invertible transformation.)
Most importantly for our purposes, this notion of basis allows us to state and prove results without restricting to the well behaved case.
\end{Remark}

In accordance with the philosophy of treating Laurent monomials in $(\sigma_i:i\in I_\uf)$ as coefficients (discussed in Remarks~\ref{useless dimensions} and~\ref{useless theta functions} and in Section~\ref{coeff sec}), we want to pass to a subset of the basis and recover the whole basis using the coefficients.
In the process, we will also put fairly strict conditions on the form of basis elements.

A \emph{reduced basis} for $\can\bigl(\tB\bigr)$ is a subset $\U=\set{\u_m:m\in M^\circ_\uf}$ of $\can\bigl(\tB\bigr)$ such that each $\u_m$ is~$z^m$ times a formal power series in $\k[[\zeta]]$ \emph{with constant coefficient~$1$} and such that
\[
\overline{\U}=\set{\sigma^n\u:n\in N_\uf,\u\in\U}=\set{\sigma^n\u_m:n\in N_\uf,m\in M^\circ_\uf}
\]
 is a basis.
In this case, we say that the basis $\overline{\U}$ is \emph{reducible} and \emph{reduces} to~$\U$.
In fact, we will see in Theorem~\ref{pointed set basis} that the requirement that $\overline{\U}$ is a basis is redundant:
As long as each~$\u_m$ is~$z^m$~times a formal power series in $\k[[\zeta]]$ with constant coefficient~$1$, the set $\U=\set{\u_m:m\in M^\circ_\uf}$ of \smash{$\can\bigl(\tB\bigr)$} is a reduced basis.
For the moment, we retain the additional hypothesis.

The theta functions are the motivating example of a reduced basis for~\smash{$\can\bigl(\tB\bigr)$}.
We will call this basis the \emph{theta basis} later in the paper.

\begin{Proposition}\label{theta basis}
Suppose $\tB$ has nondegenerate coefficients.
Then $\set{\thet_m:m\in M^\circ_\uf}$ is a reduced basis for $\can\bigl(\tB\bigr)$.
\end{Proposition}
\begin{proof}
Each $\thet_m$ is $z^mF_m$ for a formal power series $F_m$ in $\k[[\zeta]]$ with constant term~$1$.
We must show that $\set{\sigma^n\thet_m:n\in N_\uf,m\in M^\circ_\uf}$ is a basis.

First, we show how to express every element of \smash{$\can\bigl(\tB\bigr)$} as a convergent sum
\[
\sum_{m\in M^\circ_\uf}\sum_{n\in N_\uf}c_{m,n}\,\sigma^n\thet_m .
\]
Since every element of \smash{$\can\bigl(\tB\bigr)$} is a finite $\k\bigl[\sigma^{\pm1}\bigr]$-linear combination of finite products of theta functions, it is enough to show that every finite product $\thet_{p_1}\cdots\thet_{p_k}$ of theta functions admits such an expression.
Since each $a(p_1,p_2,m)$ in Proposition~\ref{struct} is in $\k[[\sigma]]$, a simple induction shows that $\thet_{p_1}\cdots\thet_{p_k}$ is \smash{$\sum_{m\in M^\circ_\uf}a(p_1,\dots,p_k,m)\thet_m$} for power series $a(p_1,\dots,p_k,m)\in\k[[\sigma]]$.
Thus $\thet_{p_1}\cdots\thet_{p_k}$ can be expressed as a sum \smash{$\sum_{m\in M^\circ_\uf}\sum_{n\in N_\uf^{0+}}c_{m,n}\,\sigma^n\thet_m$}, where each $c_{m,n}$ is an element of $\k$.
Recall from Section~\ref{context sec} (under the heading \textit{definitions related to the dual}) the definition of $nB\in M_\uf^\circ$ and recall that $\sigma^n=z^{-nB}\zeta^n$.
Writing each theta function $\thet_m$ as $z^mF_m$ where $F_m\in\k[[\zeta]]$ is the $F$-series of $\thet_m$, we rewrite the sum as
\smash{$\sum_{m\in M^\circ_\uf}\sum_{n\in N_\uf^{0+}}c_{m,n}\,\zeta^nz^{m-nB}F_m$}.
Write $p$ for $p_1+\cdots+p_k$.
Since multiplication in \smash{$\can\bigl(\tB\bigr)$} is graded by the $\g$-vector, we have $m-nB=p$ whenever $c_{m,n}\neq 0$, so we can further rewrite the sum as \smash{$z^p\sum_{n\in N_\uf^{0+}}c_{p+nB,n}\,\zeta^nF_{p+nB}$}.
Since each $F_{p+nB}$ is a formal power series in $\k[[\zeta]]$, the sum converges in $\k[[\zeta]]$, and we conclude that the sum \smash{$\sum_{m\in M^\circ_\uf}\sum_{n\in N_\uf^{0+}}c_{m,n}\,\sigma^n\thet_m$} converges in \smash{$\bigoplus_{m\in M^\circ_\uf}z^m\k(\!(\zeta)\!)$}, as desired.

If some element admits two expressions as a convergent sum, then the difference of the two expressions is a sum $\sum_{m\in M^\circ_\uf}\sum_{n\in N_\uf}c_{m,n}\,\sigma^n\thet_m$ that converges to $0$.
We will show that $c_{m,n}=0$ for all $m\in M^\circ_\uf$ and $n\in N_\uf$.
Again writing $\thet_m$ as $z^mF_m$ and using $\sigma^n=z^{-nB}\zeta^n$, we rewrite the sum as \smash{$\sum_{m\in M^\circ_\uf}\sum_{n\in N_\uf}c_{m,n}\,\zeta^nz^{m-nB}F_m$}
Since the sum converges, among all the partial sums, there are only finitely many different Laurent monomials $z^{m-nB}$ appearing in the sum with $c_{m,n}\neq0$.

Suppose for the sake of contradiction that one or more coefficients $c_{m,n}$ is nonzero and write~$q$ for $m-nB$.
Write $T_q=\set{n\in N_\uf:c_{q+nB,n}\neq0}$.
Again because the sum converges, there is a~componentwise lower bound for $T_q$ in $N_\uf$, and therefore~$T_q$ contains at least one element $r$ with componentwise minimal $f_i\text{-coordinates}$.
We obtain a contradiction:
By construction, $c_{q+rB,r}$ is nonzero, so since the constant term of $F_r$ is $1$, the summand indexed by $m=q+rB$ and $n=r$ in \smash{$\sum_{m\in M^\circ_\uf}\sum_{n\in N_\uf}c_{m,n}\,\zeta^nz^{m-nB}F_m$} includes a nonzero term $c_{q+rB,r}\zeta^rz^q$.
But because~$r$ was chosen to be minimal in $T_q$, no other choice of $m$ and $n$ gives any terms involving $\zeta^rz^q$.
This contradicts the fact that \smash{$\sum_{m\in M^\circ_\uf}\sum_{n\in N_\uf}c_{m,n}\,\sigma^n\thet_m$} converges to~$0$, and we conclude that $c_{m,n}=0$ for all $m\in M^\circ_\uf$ and $n\in N$.

This contradiction shows that each element of \smash{$\can\bigl(\tB\bigr)$} admits only one expression as a convergent sum.
\end{proof}

If $B$ is well behaved, then every element of \smash{$\can\bigl(\tB\bigr)$} is a finite $\k$-linear combination of elements~$\sigma^n\thet_m$ for $n\in N_\uf$ and $m\in M^\circ_\uf$.
Thus we have the following special case of Proposition~\ref{theta basis}.
It will be apparent from Lemma~\ref{point el} that the corollary also holds with the theta basis replaced by any reduced basis.

\begin{Corollary}\label{well be theta}
Suppose $B$ is well behaved and $\tB$ has nondegenerate coefficients.
Then the reduced basis $\set{\thet_m:m\in M^\circ_\uf}$ is a~$($Hamel$)$ $\k\bigl[\sigma^{\pm1}\bigr]$-basis for \smash{$\can\bigl(\tB\bigr)$}.
Equivalently, the set $\set{\sigma^n\thet_m:n\in N_\uf,m\in M^\circ_\uf}$ is a~$($Hamel$)$ $\k$-basis for \smash{$\can\bigl(\tB\bigr)$}.
\end{Corollary}

Indeed, Corollary~\ref{well be theta} essentially recovers the Gross--Hacking--Keel--Kontsevich construction of the canonical algebra as the space of formal linear combinations of theta functions.
(See Section~\ref{GHKK compare}.)

\begin{Remark}\label{not reduced basis}
One might hope to extend Corollary~\ref{well be theta} beyond the well behaved case by taking a larger base ring than $\k\bigl[\sigma^{\pm1}\bigr]$.
For example, one might wonder whether the theta functions form a (Hamel) $\k(\!(\sigma)\!)$-basis for $\can\bigl(\tB\bigr)$.
But in fact, $\bigoplus_{m\in M^\circ_\uf}z^m\k(\!(\zeta)\!)$ is not even a $\k(\!(\sigma)\!)$-algebra, because $\k(\!(\sigma)\!)$ is not even a subset of $\bigoplus_{m\in M^\circ_\uf}z^m\k(\!(\zeta)\!)$.
(A formal Laurent series in $(\sigma_i:i\in I_\uf)$ can be nonzero in infinitely many direct summands of $\bigoplus_{m\in M^\circ_\uf}z^m\k(\!(\zeta)\!)$.)
\end{Remark}

To conclude this section, we give a characterization of general reduced bases in terms of the theta basis.
Recall from Section~\ref{context sec} the definition of the vector $nB\in M^\circ_\uf$, given $n\in N_\uf$.
Recall also from Section~\ref{def can sec} that the $\g$-vector of $\sigma^n$ is $\g(\sigma^n)=-nB$.
We begin with the following lemma.

\begin{Lemma}\label{point el}\samepage
Suppose $\tB$ has nondegenerate coefficients.
Let $\U=\set{\u_p:p\in M^\circ_\uf}$ be a reduced basis for $\can\bigl(\tB\bigr)$, indexed with $\g(\u_p)=p$ for all $p\in M^\circ_\uf$.
Let~$\v\in\can\bigl(\tB\bigr)$ be of the form $z^mH$ for some formal power series $H\in \k[[\zeta]]$ with constant term $1$.
If the expression for $\v$ in terms of $\U$ is \smash{$\v=\sum_{p\in M^\circ_\uf}\sum_{n\in N_\uf}c_{p,n}\,\sigma^n \u_p$}, then
\begin{enumerate}[label=\rm(\arabic*), ref=\arabic*]\itemsep=0pt
\item \label{c g}
All nonzero terms $c_{p,n}\,\sigma^n \u_p$ have $\g$-vector $m$.
\item \label{c is 1}
$c_{m,0}=1$.
\item \label{c is 0}
$c_{p,0}=0$ for $p\in M^\circ_\uf\setminus\set{m}$.
\item \label{c B n}
If $c_{p,n}\neq0$ for some $p\in M^\circ_\uf$ and $n\in N_\uf$, then $n\in N_\uf^{0+}$ and $p=m+nB$.
\end{enumerate}
In particular, $\v=\u_m+\sum_{n\in N^+_\uf}(c_{m+nB,n})\sigma^n\u_{m+nB}$.
\end{Lemma}

\begin{proof}
Since \smash{$\sum_{p\in M^\circ_\uf}\sum_{n\in N_\uf}c_{p,n}\,\sigma^n \u_p$} converges to $\v$, which is homogeneous with $\g$-vector $m$, the restriction of \smash{$\sum_{p\in M^\circ_\uf}\sum_{n\in N_\uf}c_{p,n}\,\sigma^n\u_p$} to terms whose $\g$-vector is not~$m$ converges to~$0$.
Since $\U$ is a reduced basis, the trivial expression is the only expression for $0\in\can\bigl(\tB\bigr)$ as a~convergent sum of terms $\sigma^n\u_p$.
Assertion~(\ref{c g}) follows.

Suppose $c_{p,n}\neq0$.
The $\g$-vector of $\sigma^n \u_p$ is $-nB+p$, which equals $m$ by assertion~(\ref{c g}).
To prove assertion~(\ref{c B n}), it remains to show that $n\in N_\uf^{0+}$.
The term $c_{p,n}\sigma^n\u_p$ is $z^m\zeta^n$ times a formal power series in $\k[[\zeta]]$ with constant term $1$.
In particular, if $c_{p,n}\neq0$, then $c_{p,n}\sigma^n\u_p$ contributes a~nonzero term $c_{p,n}z^m\zeta^n$ plus constant multiples of $z^m\zeta^{n'}$ such that $n'\ge n$ in the componentwise order on~$N_\uf$.
The fact that \smash{$\sum_{p\in M^\circ_\uf}\sum_{n\in N_\uf}c_{p,n}\,\sigma^n \u_p$} converges implies that there is a componentwise lower bound on $\set{n\in N_\uf:c_{m+nB,n}\neq0}$.
Therefore also $\set{n\in N_\uf:c_{m+nB,n}\neq0}$ contains at least one element that is minimal in the componentwise order.
Choose $n$ to be minimal.
Since $\v$ is $z^m$ times a formal power series in $\k[[\zeta]]$ with constant term~$1$, if $n\not\in N_\uf^{0+}$, the term $c_{p,n}z^m\zeta^n$ must cancel with some other term.
Each $\sigma^{n'}\u_{p'}$ is $z^m\zeta^{n'}$ times a formal power series in $\k[[\zeta]]$, so $\sigma^{n'}\u_{p'}$ can only provide a term that cancels with $c_{p,n}z^m\zeta^n$ if $n'\le n$ componentwise.
Since $n$ was chosen to be minimal, we see that the term $c_{p,n}z^m\zeta^n$ cannot be canceled by another term.
By this contradiction, we conclude that $n\in N_\uf^{0+}$, and we have proved assertion~(\ref{c B n}).
Assertion~(\ref{c is 0}) follows immediately.

By hypothesis, $\v$ is $z^m$ times a formal power series $H$ in $\k[[\zeta]]$ with constant term~$1$.
Also, assertion~(\ref{c B n}) says that \smash{$\v=\sum_{n\in N^{0+}_\uf}(c_{m+nB,n})\sigma^n\u_{m+nB}$}.
Each $\sigma^n\u_{m+nB}$ is $z^m\zeta^n$ times a formal power series in $\k[[\zeta]]$ with constant term $1$, so $\sigma^n\u_{m+nB}$ does not contribute to the constant term of~$H$ for $n\neq0$.
Thus the constant term of~$H$ is $c_{m,0}$, which is therefore~$1$, and we have proved assertion~(\ref{c is 1}).

These assertions together say that $\v=\u_m+\sum_{n\in N^+_\uf}(c_{m+nB,n})\sigma^n\u_{m+nB}$.
\end{proof}

The following theorem is the desired characterization of reduced bases.
We emphasize an essential hypothesis of the theorem: We must know \textit{a priori} that $\U$ is a subset of $\can\bigl(\tB\bigr)$.

\begin{Theorem}\label{pointed set basis}
If $\tB$ has nondegenerate coefficients and $\U=\set{\u_m:m\in M^\circ_\uf}$ is a subset of $\can\bigl(\tB\bigr)$, then the following are equivalent:
\begin{enumerate}[label=\rm(\roman*), ref=(\roman*)]\itemsep=0pt
\item \label{u point}
$\U$ is a reduced basis for $\can\bigl(\tB\bigr)$, indexed so that $\g(\u_p)=p$ for all $p\in M^\circ_\uf$.
\item \label{each point}
Each $\u_m$ is $z^m$ times a formal power series in $\k[[\zeta]]$ with constant coefficient~$1$.
\item \label{u thet}
Each $\u_m$ is \smash{$\thet_m+\sum_{n\in N_\uf^+}c^{(m)}_n\,\sigma^n\thet_{m+nB}$} for some constants $c^{(m)}_n\in\k$.
\item \label{u v}
If $\V=\set{\v_p:p\in M^\circ_\uf}$ is a reduced basis for $\can\bigl(\tB\bigr)$, then for each $m\in M^\circ_\uf$, the element~$\u_m$ is \smash{$\v_m+\sum_{n\in N_\uf^+}c^{(m)}_n\,\sigma^n\v_{m+nB}$} for some constants $c^{(m)}_n\in\k$.
\end{enumerate}
\end{Theorem}
\begin{proof}
By definition, \ref{u point} implies \ref{each point}.
If~\ref{each point} holds, then~\ref{u v} holds by Lemma~\ref{point el}.

Now suppose \ref{u v} so that, taking $c^{(m)}_0=1$, each $\u_m$ is \smash{$\sum_{n\in N_\uf^{0+}}c^{(m)}_n\,\sigma^n\v_{m+nB}$}.
For each $p\in M^\circ_\uf$ and $r\in N_\uf^{0+}$, let $d^{(p)}_r$ be \smash{$\sum(-1)^k\prod_{i=1}^kc^{(p+n_{i-1}B)}_{n_i-n_{i-1}}$}, a finite sum indexed by all chains of the form $0=n_0<n_1<\cdots<n_k=r$ in the componentwise order.
Each $k$ must be positive, except when $r=0$, in which case $k$ must be zero and we interpret the empty product as~$1$, so that $d^{(p)}_0=1$.
For each $p\in M^\circ_\uf$, we compute (taking $s=n+r$):
\begin{align*}
\sum_{r\in N_\uf^{0+}}d^{(p)}_r\,\sigma^r\u_{p+rB}
&=\sum_{r\in N_\uf^{0+}}d^{(p)}_r\,\sigma^r\sum_{n\in N_\uf^{0+}}c^{(p+rB)}_n\,\sigma^n\v_{p+rB+nB}\\
&=\sum_{s\in N_\uf^{0+}}\sigma^s\v_{p+sB}\sum_{0\le r\le s}d^{(p)}_rc^{(p+rB)}_{s-r}\\
&=\sum_{s\in N_\uf^{0+}}\sigma^s\v_{p+sB}\sum_{0\le r\le s}c^{(p+rB)}_{s-r}\sum(-1)^k\prod_{i=1}^kc^{(p+n_{i-1}B)}_{n_i-n_{i-1}}\,,
\end{align*}
where the third sum is over chains $0=n_0<n_1<\cdots<n_k=r$.
The sum over $r$ and the sum over chains can be combined, and we obtain
\begin{align*}
\sum_{r\in N_\uf^{0+}}d^{(p)}_r\,\sigma^r\u_{p+rB}
&=\sum_{s\in N_\uf^{0+}}\sigma^s\v_{p+sB}\sum(-1)^k\prod_{i=1}^{k+1}c^{(p+n_{i-1}B)}_{n_i-n_{i-1}}\,,
\end{align*}
where the sum is over chains $0=n_0<n_1<\cdots<n_k\le n_{k+1}=s$ (and $n_k$ is the~$r$ from the earlier expression).
That sum is $1$ if $s=0$ or $0$ if $s>0$.
(If $s>0$ then each chain in the interval $[0,s]$ appears exactly twice, once with $n_k=n_{k+1}$ and once with $n_k<n_{k+1}$.
These have the same weight, except with opposite signs, because \smash{$c^{(p+n_kB)}_0=1$}.)
Thus \smash{$\sum_{r\in N_\uf^{0+}}d^{(p)}_r\,\sigma^r\u_{p+rB}$} converges to $\v_p$.

Now write each $\v_m$ as $z^mH_m$ for $H_m\in\k[[\zeta]]$.
Since also $\sigma^n=z^{-nB}\zeta^n$, we see that each $\u_m=\smash{\sum_{n\in N_\uf^{0+}}c^{(m)}_n\,\sigma^n\v_{m+nB}}$ is \[
\sum_{n\in N_\uf^{0+}}c^{(m)}_n\,z^{-nB}\zeta^nz^{m+nB}H_{m+nB}=z^mK_m
\] for some $K_m\in\k[[\zeta]]$.

Suppose $\w\in\can\bigl(\tB\bigr)$.
Express $\w$ as a convergent linear combination of Laurent monomials in $\sigma$ times elements of $\V$ and replace each $\v_p$ in the expression by \smash{$\sum_{r\in N_\uf^{0+}}d^{(p)}_r\,\sigma^r\u_{p+rB}$}.
The result is an expression for $\w$ as a convergent linear combination of Laurent monomials in $\sigma$ times elements of~$\U$.
The argument that such an expression is unique is the same as in Proposition~\ref{theta basis}, using the expressions $\u_m=z^mK_m$ in place of the expressions $\thet_m=z^mF_m$.
Thus $\U$ is a reduced basis.
We have shown that \ref{u point}, \ref{each point}, and \ref{u v} are all equivalent.

In light of Proposition~\ref{theta basis},~\ref{u v} implies~\ref{u thet}.
Specializing the argument above that~\ref{u v} implies~\ref{each point}, we see that~\ref{u thet} implies~\ref{each point}.
\end{proof}

\section{Mutation of theta functions}\label{mut sec}
In this section, we discuss the notion of ``mutating'' theta functions.
Crucially, \textit{throughout Section~{\rm \ref{mut sec}}, we assume that $\tB$ has signed-nondegenerating coefficients}.

There are several notions of what it should mean to ``mutate'' theta functions.
The key technical observation underlying this paper describes the relationship between two of them.

The first, and simpler, notion is to compute theta functions relative to \smash{$\Scat\bigl(\tB\bigr)$}, pass from $\tB$ to \smash{$\mu_k\bigl(\tB\bigr)$}, otherwise keeping everything the same, and then separately compute theta functions relative to \smash{$\Scat\bigl(\mu_k\bigl(\tB\bigr)\bigr)$}.
Here is a more precise description of ``otherwise keeping everything the same'':
If we have constructed the underlying data from $\tB$ as described near the end of Section~\ref{context sec}, then we construct the underlying data again from \smash{$\mu_k\bigl(\tB\bigr)$}, \emph{using the same integers~$d_i$ and the same symbols $e_i$ and~$e^*_i$}.
Thus, all of the data in Section~\ref{context sec} is the same, except for the entries~$\epsilon_{ij}$ and the skew-symmetric bilinear form $\set{\,\cdot\,,\,\cdot\,}$.

\begin{Remark}\label{our mut their mut}
In this paper, we will not use the notion of mutation in~\cite{GHKK}, which differs from the notion described above because different things change and different things stay the same.
In \cite{GHKK}, everything stays the same, including $\set{\,\cdot\,,\,\cdot\,}$, except for $\tB$ and the choice of bases $(e_i:i\in I)$ for $N$ and $(f_i:i\in I)$ for~$M^\circ$.
A new basis $(e'_i:i\in I)$ is chosen for $N$ such that $\set{e'_i,d_je'_j}$ are the entries of the mutated matrix (and then $f'_i$ is defined to be $d_i^{-1}(e')^*_i$).
The mutation results that we quote and prove here have analogs in \cite{GHKK} for this different notion of mutation.
\end{Remark}

The second notion of mutation of theta functions uses mutation of cluster variables as follows.
The triple $(B,(z_i:i\in I_\uf),(\sigma_i:i\in I_\uf))$ constitutes a seed of geometric type in the sense of cluster algebras \cite{ca4}.
(However, to conform to the conventions of \cite{GHKK} as already described, we continue with a ``wide'' extended exchange matrix $\tB=[\epsilon_{ij}]_{i\in I_\uf,j\in I}$ rather than a ``tall'' one.
Thus the construction is related to the constructions in \cite{ca4} by a global transpose.)

Mutation at $k$ creates a new seed $(\mu_k(B),(z'_i:i\in I_\uf),(\sigma'_i:i\in I_\uf))$.
The relationship between the two seeds is given by matrix mutation and the exchange relation.
As part of the assumption that $\tB$ has signed-nondegenerating coefficients, each $\sigma_i$ has a $\sgn(\sigma_i)\in\set{\pm1}$, and similarly each $\sigma'_i$ has a sign, which simplifies the relationships between primed and unprimed quantities to the equations below.
One may take equation~\eqref{exch rel} to be the definition of the exchange relation in this case, or refer to \cite[(3.11)]{ca4}.
The other relations below follow from the definition of matrix mutation.
(Compare \mbox{\cite[Proposition~3.9]{ca4}} and its proof.)
\begin{align}
\label{exch rel}
&z_kz'_k=(1+\zeta_k)\sigma_k^{-[-\sgn(\sigma_k)]_+}\prod_{j\in I_\uf}z_j^{[-\epsilon_{kj}]_+},\\
&z_i'=z_i\quad\text{for }i\in I_\uf\setminus\set{k},\\
&\sigma'_k=\sigma_k^{-1},\\
&\sigma'_i=\sigma_i(\sigma_k)^{[\sgn(\sigma_k)\epsilon_{ik}]_+}\quad\text{for }i\in I_\uf\setminus\set{k},\\
&\zeta'_k=\zeta_k^{-1},\\
\label{hy mut}
&\zeta'_i=\zeta_i(\zeta_k)^{[\epsilon_{ik}]_+}(1+\zeta_k)^{-\epsilon_{ik}}\quad\text{for }i\in I_\uf\setminus\set{k}.
\end{align}

The second notion of mutation of theta functions is to use the above relations between primed and unprimed variables to write each theta function relative to \smash{$\Scat\bigl(\tB'\bigr)$} in terms of the unprimed variables.
We will prove the following theorem relating the two notions.

\begin{Theorem}\label{2 muts}
Suppose $\tB$ has signed-nondegenerating coefficients and, for some $k\in I_\uf$, write~$\tB'$ for \smash{$\mu_k\bigl(\tB\bigr)$}.
For $m\in M^\circ_\uf$, write $\thet_{m}^\tB$ for a theta function defined in the unprimed variables using $\tB$, write $m'=\eta_k^B(m)$, and write $\thet_{m'}^{\tB'}$ for a theta function defined in the primed variables using $\tB'$.
Relating the primed and unprimed variables as in \eqref{exch rel}--\eqref{hy mut}, we have $\thet_{m'}^{\tB'}=\thet_{m}^\tB\cdot(\sigma_k)^{-[\sgn(\sigma_k)\br{m,d_ke_k}]_+}$.
\end{Theorem}

In interpreting Theorem~\ref{2 muts}, it is useful to remember that $\sigma'_k=\sigma_k^{-1}$.
Thus the conclusion of the theorem can be rewritten as
\[
\thet_{m'}^{\tB'}=\thet_{m}^\tB\cdot(\sigma'_k)^{[\sgn(\sigma_k)\br{m,d_ke_k}]_+} \qquad \text{or}\qquad \thet_{m}^{\tB}=\thet_{m'}^{\tB'}\cdot(\sigma'_k)^{-[\sgn(\sigma_k)\br{m,d_ke_k}]_+}.
\]

Before proving Theorem~\ref{2 muts}, we point out a corollary, which is immediate by induction on the length of a sequence~$\kk$.
\begin{Corollary}\label{mut can}
Suppose $\tB$ has signed-nondegenerating coefficients.
For any sequence $\kk$ of indices in $I_\uf$, if we identify indeterminates related to $\tB$ with indeterminates related to $\mu_\kk\bigl(\tB\bigr)$ by iterations of \eqref{exch rel}--\eqref{hy mut}, then $\can\bigl(\mu_\kk\bigl(\tB\bigr)\bigr)=\can\bigl(\tB\bigr)$.
\end{Corollary}

We now move to the proof of Theorem~\ref{2 muts}.
The proof consists of expressing both notions of mutation of theta functions in terms of formal symbolic operations and then comparing.
First, \eqref{exch rel}--\eqref{hy mut} can be rephrased as the following proposition.
\big(To get the correct replacement for $z_k$, it is useful to rewrite $\bigl(1+(\zeta_k')^{-1}\bigr)$ as $(1+\zeta_k')(\zeta_k')^{-1}$.\big)

\begin{Proposition}\label{mutate subs}
Suppose $\tB$ has signed-nondegenerating coefficients.
Let $\v$ be an expression in $(z_i:i\in I_\uf)$, $(\sigma_i:i\in I_\uf)$, and $(\zeta_i:i\in I_\uf)$.
Then $\v$ can be expressed in terms of $(z'_i:i\in I_\uf)$, $(\sigma'_i:i\in I_\uf)$, and $(\zeta'_i:i\in I_\uf)$ by simultaneously making the following substitutions:
\begin{alignat*}{2}
\text{replace }&z_k&&\text{ by}\quad \frac{1+\zeta'_k}{z'_k(\sigma'_k)^{[\sgn(\sigma_k)]_+}}\prod_{j\in I_\uf}(z'_j)^{[\epsilon_{kj}]_+},\\
\text{replace }&z_i&&\text{ by}\quad z'_i \quad\text{for }i\in I_\uf\setminus\set{k},\\
\text{replace }&\sigma_k&&\text{ by}\quad (\sigma_k')^{-1},\\
\text{replace }&\sigma_i&&\text{ by}\quad\sigma'_i(\sigma'_k)^{[\sgn(\sigma_k)\epsilon_{ik}]_+} \quad\text{for }i\in I_\uf\setminus\set{k},\\
\text{replace }&\zeta_k&&\text{ by}\quad (\zeta_k')^{-1},\\
\text{replace }&\zeta_i&&\text{ by}\quad \zeta_i'(\zeta'_k)^{[-\epsilon_{ik}]_+}(1+\zeta'_k)^{\epsilon_{ik}} \quad\text{for }i\in I_\uf\setminus\set{k}.
\end{alignat*}
\end{Proposition}

Our next goal is to prove the following proposition.

\begin{Proposition}\label{mut theta}
Continuing hypotheses and notation from Theorem~{\rm \ref{2 muts}}, $\thet_{m'}^{\tB'}$ is obtained from $\thet^\tB_m$ by substitution and then multiplication, as follows:
First, simultaneously make the following substitutions:
\begin{alignat*}{2}
\text{replace }&z_k&&\text{ by}\quad \frac{1+\zeta'_k}{z'_k(\sigma'_k)^{[\sgn(\sigma_k)]_+}}\prod_{j\in I_\uf}(z'_j)^{[\epsilon_{kj}]_+},\\
\text{replace }&z_i&&\text{ by}\quad z'_i \quad\text{for }i\in I_\uf\setminus\set{k},\\
\text{replace }&\sigma_k&&\text{ by}\quad (\sigma_k')^{-1},\\
\text{replace }&\sigma_i&&\text{ by}\quad\sigma'_i(\sigma'_k)^{[\sgn(\sigma_k)\epsilon_{ik}]_+} \quad\text{for }i\in I_\uf\setminus\set{k}.
\end{alignat*}
Then multiply by $(\sigma'_k)^{[\sgn(\sigma_k)\br{m,d_ke_k}]_+}$.
\end{Proposition}

The combination of Propositions~\ref{mut theta} and~\ref{mutate subs} immediately implies Theorem~\ref{2 muts}.
Specifically, we use the fact that $\zeta^n=z^{nB}\sigma^n$ for all $n\in N_\uf$ to write \smash{$\thet_m^{\tB}$} in terms of the symbols $z$ and $\sigma$ only, without the symbols~$\zeta$.
Then, using Proposition~\ref{mutate subs} to write $\thet_m^\tB$ in terms of the symbols~$z'$ and~$\sigma'$ and comparing with Proposition~\ref{mut theta}, we see that $\thet_{m'}^{\tB'}=\thet_m^{\tB}\cdot(\sigma'_k)^{[\sgn(\sigma_k)\br{m,d_ke_k}]_+}$.

We now proceed to prove Proposition~\ref{mut theta}.
We begin by quoting \cite[Theorem~4.2]{scatfan}, which is a reinterpretation of \cite[Theorem~1.24]{GHKK}.
The theorem explains how to construct the cluster scattering diagram \smash{$\Scat\bigl(\mu_k\bigl(\tB\bigr)\bigr)$} in terms of $(z'_i:i\in I_\uf)$ and $(\zeta'_i:i\in I_\uf)$, given the cluster scattering diagram for $\tB$ in terms of $(z_i:i\in I_\uf)$ and $(\zeta_i:i\in I_\uf)$.
(The difference between the result in~\cite{scatfan} and the result in~\cite{GHKK} is that, in~\cite{scatfan}, a linear map is applied to change bases of $V^*$.)

Up to equivalence of scattering diagrams, we can assume that there is no wall $(\d,f_\d(\zeta^n))$ of $\Scat\bigl(\tB\bigr)$ such that $\d$ contains points strictly on opposite sides of the hyperplane $e_k^\perp$ in $V^*$.
Thus, we can apply the piecewise-linear map $\eta_k^B$ to whole walls, because it is linear on each halfspace defined by $e_k^\perp$.
The following is \cite[Theorem~4.2]{scatfan}.

\begin{Theorem}\label{mut thm}
$\Scat\bigl(\mu_k\bigl(\tB\bigr)\bigr)$ (in the primed variables) is obtained from $\Scat\bigl(\tB\bigr)$ $($in the unprimed variables$)$ by changing the wall $(e_k^\perp,1+\zeta_k)$ to $(e_k^\perp,1+\zeta'_k)$ and altering each other wall $(\d,f_\d(\zeta^n))$ as follows:
Replace $\d$ by $\eta_k^B(\d)$ and make the following substitutions in $f_\d(\zeta^n)$:
\begin{alignat*}{2}
\text{replace }&\zeta_k&&\text{ by}\quad (\zeta'_k)^{-1},\\
\text{replace }&\zeta_i&&\text{ by}\quad
\begin{cases}
\zeta_i'(\zeta'_k)^{[\epsilon_{ik}]_+} &\text{if }\d\subseteq \set{p\in V^*:\br{p,e_k}\le0},\\
\zeta_i'(\zeta'_k)^{[-\epsilon_{ik}]_+} &\text{if }\d\subseteq \set{p\in V^*:\br{p,e_k}\ge0},
\end{cases}
\qquad\text{for }i\in I_\uf\setminus\set{k}.
\end{alignat*}
\end{Theorem}

\begin{Example}\label{theta g2 eta1 ex}
This example continues Example~\ref{theta g2 ex}.
We take \smash{$B=\left[\begin{smallmatrix}0& -3\\1&\hphantom{-}0\end{smallmatrix}\right]$} and $\tB$ as before and consider mutation in position~$1$, so that \smash{$\mu_1(B)=-B=\left[\begin{smallmatrix}\hphantom{-}0& 3\\-1& 0\end{smallmatrix}\right]$}.
The black lines and black labels~in Figure~\ref{theta g2 eta1 fig} show $\Scat(\mu_1(\tB))$, in primed variables, for~$\tB$ as in Example~\ref{theta g2 ex}.
\begin{figure}[!t]\centering
\scalebox{0.85}{
\includegraphics{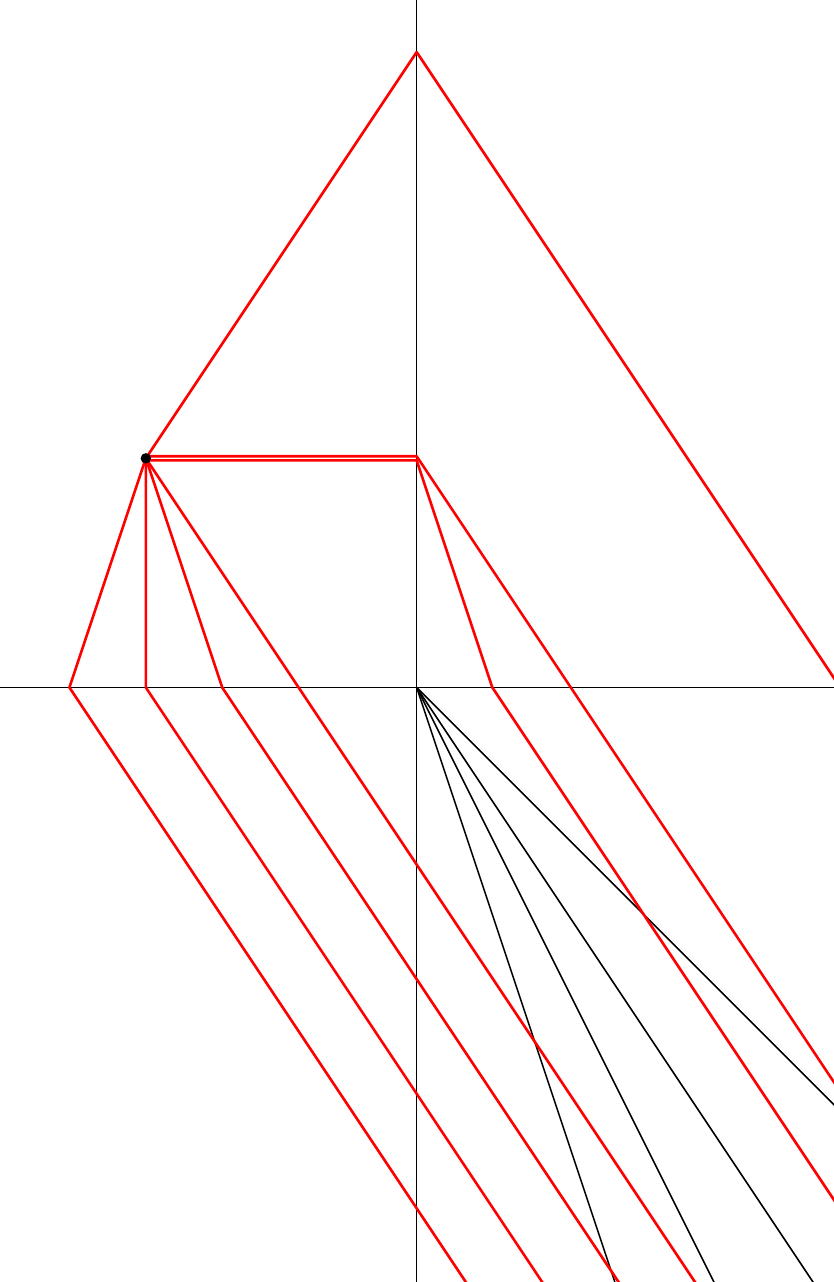}
\begin{picture}(0,0)(200,-285)
\put(-1,323){$1+\zeta'_1$}
\put(-200,-10){$1+\zeta'_2$}
\put(-33,-282){$1+\zeta'_1$}
\put(168,-10){$1+\zeta'_2$}

\put(102,158){\textcolor{red}{\rotatebox{-56}{$(z'_1)^2(z'_2)^{-3}$}}}

\put(44,-8){\textcolor{red}{\rotatebox{-56}{$(z'_1)^2(z'_2)^{-3}$}}}
\put(82,-8){\textcolor{red}{\rotatebox{-56}{$(z'_1)^2(z'_2)^{-3}$}}}

\put(-2,-80){\textcolor{red}{\rotatebox{-56}{$(z'_1)^2(z'_2)^{-3}$}}}
\put(-2,-135){\textcolor{red}{\rotatebox{-56}{$(z'_1)^2(z'_2)^{-3}$}}}
\put(-2,-190){\textcolor{red}{\rotatebox{-56}{$(z'_1)^2(z'_2)^{-3}$}}}
\put(-2,-245){\textcolor{red}{\rotatebox{-56}{$(z'_1)^2(z'_2)^{-3}$}}}

\put(-96,158){\textcolor{red}{\rotatebox{56}{$(\sigma'_1)^2(z'_1)^2(z'_2)^3$}}}

\put(-80,118){\textcolor{red}{$2\sigma'_1(z'_1)^2$}}

\put(18,55){\textcolor{red}{\rotatebox{-71.5}{$3\sigma'_2z'_1(z'_2)^{-3}$}}}

\put(-80,98){\textcolor{red}{$3\sigma'_1\sigma'_2z'_1$}}

\put(-112,56){\textcolor{red}{\rotatebox{-71.5}{$3\sigma'_2z'_1(z'_2)^{-3}$}}}
\put(-132,10){\textcolor{red}{\rotatebox{90}{$3(\sigma'_2)^2(z'_2)^{-3}$}}}
\put(-167,5){\textcolor{red}{\rotatebox{71.5}{$(\sigma'_2)^3(z'_1)^{-1}(z'_2)^{-3}$}}}

\put(78,-235){\rotatebox{-71.5}{$1+\zeta'_1\zeta'_2$}} 
\put(109,-215){\rotatebox{-63}{$1+(\zeta'_1)^2(\zeta'_2)^3$}} 
\put(155,-230){\rotatebox{-56}{$1+\zeta'_1(\zeta'_2)^2$}} 
\put(154,-173){\rotatebox{-45}{$1+\zeta'_1(\zeta'_2)^3$}} 

\put(-165,114){$\eta_1^B(Q)$}
\end{picture}
}
\caption{$\Scat(\mu_1(\tB))$ and broken lines for $\thet^{\mu_1(\tB)}_{\eta_1^B(Q),\eta_1^B([-2,3])}$.}\label{theta g2 eta1 fig}
\end{figure}
The broken lines in the figure are explained in Example~\ref{theta g2 eta1 broken ex}.
\end{Example}

The following simple fact will be useful in applying Theorem~\ref{mut thm}.
Recall that for a primitive element $n\in N$, the primitive element of $N^\circ$ that is a positive multiple of $n$ is written as $n^\circ$.

\begin{Lemma}\label{transforming coroots}
Suppose $n$ is a primitive element of $N_\uf$.
Starting with $\zeta^n$, replace $\zeta_k$ by $(\zeta_k)^{-1}$ and each $\zeta_i$ by $\zeta_i(\zeta_k)^{[\pm \epsilon_{ik}]_+}$ as in Theorem~{\rm \ref{mut thm}}, and call the result $(\zeta)^{n'}$, so that
\[n'=n-2\br{e^*_k,n}e_k+\sum_{i\in I_\uf}[\pm\epsilon_{ik}]_+\br{e^*_i,n}e_k.\]
Then $n'$ is a primitive element of $N_\uf$ and \smash{$(n')^\circ\!=\!n^\circ-2\br{f_k,n^\circ}d_ke_k+\sum_{i\in I_\uf}\br{f_i,n^\circ}[\mp \epsilon_{ki}]_+d_ke_k$}.
\end{Lemma}
\begin{proof}
For either rule in Theorem~\ref{mut thm} for replacing $\zeta_i$, the map $n\mapsto n'$ is an automorphism of the lattice $N_\uf$, so it sends primitive elements to primitive elements.
Since~$N^\circ$ is a sublattice of~$N$, the map also sends primitive elements of~$N^\circ$ to primitive elements of $N^\circ$.
Thus the maps $n\mapsto n'$ and $n\mapsto n^\circ$ commute, so $(n')^\circ$ is obtained simply by applying the same linear map to~$n^\circ$.
We compute that map on the basis $(d_ie_i:i\in I_\uf)$.
The map sends $d_ke_k$ to $-d_ke_k$, and for $i\in I_\uf\setminus\set{k}$, sends $d_ie_i$ to $d_i(e_i+[\pm\epsilon_{ik}]_+e_k)$, but since the $d_i$ are the skew-symmetrizing constants, this is $d_ie_i+[\pm d_i\epsilon_{ik}]_+e_k=d_ie_i+[\mp d_k\epsilon_{ki}]_+e_k=d_ie_i+[\mp\epsilon_{ki}]_+d_ke_k$.
\end{proof}

Lemma~\ref{mut broken line}, below, relates broken lines in the scattering diagrams $\Scat(\mu_k(\tB))$ and $\Scat\bigl(\tB\bigr)$.
It is a version of \cite[Proposition~3.6]{GHKK}, but we need to modify that result in the same way that Theorem~\ref{mut thm} is a modification of \cite[Theorem~1.24]{GHKK}.
Rather than explaining how to modify \cite[Proposition~3.6]{GHKK}, we prove the modification directly using Theorem~\ref{mut thm} and the same kind of argument given in the proof of \cite[Proposition~3.6]{GHKK}.

Call a curve $\bl\colon (-\infty,0]\to V^*$ a \emph{prospective broken line} if it is piecewise-linear and has finitely many domains of linearity, each labeled by Laurent monomials in $(z_i^\pm:i\in I_\uf)$ and $(\sigma_i:i\in I_\uf)$, with the infinite domain of linearity labeled by a Laurent monomial in the $z_i$.
We extend each mutation map $\eta_\kk^B$ to a map on the set of prospective broken lines.
We begin by defining the map when $\kk$ is a singleton $k$; the map $\eta_\kk^B$ for arbitrary sequences $\kk$ is given by \eqref{eta comp}.

Given a prospective broken line $\bl$ with its infinite domain of linearity labeled $z^m$ for $m\in M^\circ$, define $\eta_k^B(\bl)$, as a curve, to be $\eta_k^B\circ\bl$.
We break all domains of $\bl$ and $\eta_k^B(\bl)$ at $e_k^\perp$ so that we can assume that no domain of either curve crosses $e_k^\perp$.
Thus every domain $L'$ of $\eta_k^B(\bl)$ is $\eta_k^B(L)$ for some domain $L$ of $\bl$.
We label $L'$ with the Laurent monomial obtained from $\const_Lz^{m_L}\sigma^{n_L}$ by substitution and multiplication as follows.
Simultaneously make these substitutions:
\begin{alignat*}{2}
\text{replace }&z_k&&\text{ by}\quad
\begin{cases}\displaystyle
(z_k')^{-1}(\sigma'_k)^{[-\sgn(\sigma_k)]_+}\prod_{j\in I_\uf}(z_j')^{[-\epsilon_{kj}]_+}&\!\!\!\text{if }L\subseteq \set{p\in V^*:\br{p,e_k}\le0},\\\displaystyle
(z_k')^{-1}(\sigma'_k)^{-[\sgn(\sigma_k)]_+}\prod_{j\in I_\uf}(z_j')^{[\epsilon_{kj}]_+}&\!\!\!\text{if }L\subseteq \set{p\in V^*:\br{p,e_k}\ge0},
\end{cases}\\
\text{replace }&z_i&&\text{ by}\quad z'_i \quad\text{for }i\in I_\uf\setminus\set{k},\\
\text{replace }&\sigma_k&&\text{ by}\quad (\sigma_k')^{-1},\\
\text{replace }&\sigma_i&&\text{ by}\quad\sigma'_i(\sigma'_k)^{[\sgn(\sigma_k)\epsilon_{ik}]_+} \quad\text{for }i\in I_\uf\setminus\set{k}.
\end{alignat*}
Then multiply by $(\sigma'_k)^{[\sgn(\sigma_k)\br{m,d_ke_k}]_+}$, noting that it is $m$, not $m_L$, that appears in the expression for the exponent.

\begin{Lemma}\label{mut broken line}
For $m\in M^\circ_\uf$ and $Q\in V^*$ and a sequence $\kk$, a prospective broken line~$\bl$ is a~broken line, relative to $\Scat\bigl(\tB\bigr)$, for $m$ with endpoint $Q$ if and only if~$\eta_\kk^B(\bl)$ is a broken line, relative to \smash{$\Scat\bigl(\mu_\kk\bigl(\tB\bigr)\bigr)$}, for $\eta_\kk^B(m)$ with endpoint $\eta_\kk^B(Q)$.
\end{Lemma}

\begin{Example}\label{theta g2 eta1 broken ex}
This example continues Examples~\ref{theta g2 ex} and~\ref{theta g2 eta1 ex}.
We take $B$ as before and assume, for convenience, that $\tB$ has principal coefficients, so that in particular it has signed-nondegenerating coefficients and $\sgn(\sigma_1)=\sgn(\sigma_2)=1$.
Taking $m=[-2,3]$ as before, $\eta^B_1(m)=[2,-3]$.
The broken lines, relative to $\Scat(\mu_k(\tB))$, for $\eta_k^B(m)$ with endpoint $\eta_k^B(Q)$ are shown in red in Figure~\ref{theta g2 eta1 fig}.
These are precisely the $\eta_1^B(\bl)$ for the broken lines $\bl$ shown in Figure~\ref{theta g2 fig}.
\end{Example}

\begin{proof}[Proof of Lemma~\ref{mut broken line}]
It is enough to prove the case where $\kk$ is a singleton sequence~$k$.
Write $m'=\eta_k^B(m)$ and $Q'=\eta_k^B(Q)$ and $\bl'=\eta_k^B(\bl)$.
Up to the symmetry of swapping $\bl$ and $\bl'$, it is enough to show that if $\bl$ is a broken line then $\bl'$ is a broken line.
Conditions \ref{brok endpoint}, \ref{brok slope}, and \ref{brok unbounded} in the definition of a broken line are equivalent for $\bl$ and $\bl'$.
By Theorem~\ref{mut thm}, Condition~\ref{brok generic} is also equivalent for $\bl$ and~$\bl'$.
It remains to check condition~\ref{brok change slope}.
In checking this condition, we can leave out the post-multiplication of Laurent monomials by $(\sigma'_k)^{[\sgn(\sigma_k)\br{m_L,d_ke_k}]_+}$, because it happens on every domain and because coefficients don't affect how broken lines bend at points of nonlinearity.

Suppose $L_1$ and $L_2$ are adjacent domains of $\bl$, with $L_1$ \emph{before} $L_2$ as we follow~$\bl$ from $-\infty$ to~$0$.
Write $\const_1$, $m_1$, and $n_1$ for $\const_{L_1}$, $m_{L_1}$, and $n_{L_1}$, and similarly $\const_2$, $m_2$, and $n_2$.
The domains $L_1$ and $L_2$ meet inside a wall, or a collection of walls with the same normal vector.
Suppose $n\in N^+_\uf$ is a primitive normal vector to this wall or walls, and that the product of the scattering terms on those walls is $\tau(\zeta^n)$, a formal power series in $\zeta^n$.
Write $\lambda$ for $|\br{m_1,n^\circ}|$.
Thus, we assume that there is a nonzero term $a\cdot(\zeta^n)^\nu$ of $(\tau(\zeta^n))^\lambda$ such that $\const_2z^{m_2}\sigma^{n_2}=\const_1z^{m_1}\sigma^{n_1}\cdot a\cdot(\zeta^n)^\nu$.
That is, we assume that there is a term $a\cdot(\zeta^n)^\nu$ such that
\begin{align}
\label{c condition}
&\const_2=a\const_1,\\
\label{m condition}
&m_2=m_1+\nu\sum_{i\in I_\uf}\br{e^*_i,n}\sum_{j\in I_\uf}\epsilon_{ij}f_j,\\
\label{n condition}
&n_2=n_1+\nu n.
\end{align}

Let $L'_1$ and $L'_2$ be the corresponding domains of $\bl'$.
By Theorem~\ref{mut thm}, we can write the scattering term where $L'_1$ and $L'_2$ meet as $\tau\bigl((\zeta')^{n'}\bigr)$ for some $n'$ and for the same $\tau$ as above.
Similarly, we write $\const_1\cdot(z')^{m'_1}(\sigma')^{n'_1}$ and $\const_2\cdot(z')^{m'_2}(\sigma')^{n'_2}$ for the Laurent monomials attached to~$L'_1$ and~$L'_2$, for some $m_1'$, $n_1'$, $m_2'$, and $n_2'$.
Write $\lambda'$ for $|\br{m_1',(n')^\circ}|$.
We want to show that there is a term \smash{$a'\cdot\bigl((\zeta')^{n'}\bigr)^{\nu'}$} of \smash{$\bigl(\tau\bigl((\zeta')^{n'}\bigr)\bigr)^{\lambda'}$} such that \smash{$\const_1\cdot(z')^{m'_1}(\sigma')^{n'_1}=\const_2\cdot(z')^{m'_2}(\sigma')^{n'_2}\cdot a'\cdot\bigl((\zeta')^{n'}\bigr)^{\nu'}$}.
That is, we want a term \smash{$a'\cdot\bigl((\zeta')^{n'}\bigr)^{\nu'}$} such that
\begin{align}
\label{c condition prime}
&\const_2=a'\const_1,\\
\label{m condition prime}
&m_2'=m_1'+\nu'\sum_{i\in I_\uf}\br{e^*_i,n'}\sum_{j\in I_\uf}\epsilon'_{ij}f_j,\\
\label{n condition prime}
&n_2'=n_1'+\nu'n',
\end{align}
where the symbols $\epsilon'_{ij}$ are the entries of $\mu_k(B)$.

The relationship between the primed and unprimed symbols depends on where~$L_1$ and~$L_2$ are relative to $e_k^\perp$.
There are four cases, but we can treat them together in pairs as Case 1 and Case 2.
In every case, we compute $n'$ by Theorem~\ref{mut thm} and compute $m_1'$, $n_1'$, $m_2'$, and $n_2'$ by the definition of $\bl'$, neglecting, as justified above, the post-multiplication.
We will see that $\lambda'=\lambda$ in every case.

\medskip

\noindent
\textbf{Case 1:}
\emph{$L_1$ and $L_2$ are on the same side of $e_k^\perp$}.
We write expressions for the primed quantities using $\pm$ and $\mp$, with the top sign referring to the case where $L_1$ and $L_2$ are on the negative side of $e_k^\perp$ and the bottom sign referring to the case where $L_1$ and $L_2$ are on the positive side:
\begin{align}
\label{n prime}
&n'=n-2\br{e^*_k,n}e_k+\sum_{i\in I_\uf}[\pm \epsilon_{ik}]_+\br{e^*_i,n}e_k,\\
\label{m 1 prime}
&m'_1=m_1-2\br{m_1,d_ke_k}f_k+\br{m_1,d_ke_k}\sum_{j\in I_\uf}[\mp \epsilon_{kj}]_+f_j,\\
&n_1'=n_1-\br{m_1,d_ke_k}\sgn(\sigma_k)[\mp\sgn(\sigma_k)]_+e_k-2\br{e^*_k,n_1}e_k\nonumber\\
&\hphantom{n_1'=}{} +\sum_{i\in I_\uf}\br{e^*_i,n_1}[\sgn(\sigma_k)\epsilon_{ik}]_+e_k,\label{n 1 prime}\\
\label{m 2 prime}
&m'_2=m_2-2\br{m_2,d_ke_k}f_k+\br{m_2,d_ke_k}\sum_{j\in I_\uf}[\mp \epsilon_{kj}]_+f_j,\\
&n_2'=n_2-\br{m_2,d_ke_k}\sgn(\sigma_k)[\mp\sgn(\sigma_k)]_+e_k-2\br{e^*_k,n_2}e_k\nonumber\\
&\hphantom{n_2'=}{}
+\sum_{i\in I_\uf}\br{e^*_i,n_2}[\sgn(\sigma_k)\epsilon_{ik}]_+e_k.\label{n 2 prime}
\end{align}

Using Lemma~\ref{transforming coroots}, we compute $\lambda'=\br{m'_1,(n')^\circ}=\br{m_1,n^\circ}=\lambda$.
Thus $\tau\bigl((\zeta')^{n'}\bigr)^{\lambda'}$ 
has a term $a\cdot\bigl((\zeta')^{n'}\bigr)^{\nu}$, so we will show that conditions \eqref{c condition prime}, \eqref{m condition prime}, and \eqref{n condition prime} hold with $a'=a$ and $\nu'=\nu$.
condition~\eqref{c condition prime} is immediate.

If $\nu=0$, then condition \eqref{m condition prime} holds easily, so assume $\nu\neq0$ and compute
\begin{gather*}
\frac1\nu(m_2'-m_1')
=\frac1\nu(m_2-m_1)-\frac2\nu\br{m_2-m_1,d_ke_k}f_k
+\frac1\nu\br{m_2-m_1,d_ke_k}\sum_{j\in I_\uf}[\mp \epsilon_{kj}]_+f_j.
\end{gather*}
Replacing each $m_2-m_1$ according to \eqref{m condition} and simplifying, we obtain
\begin{gather}
\frac1\nu(m_2'-m_1')=\sum_{i\in I_\uf}\br{e^*_i,n}\sum_{j\in I_\uf}\epsilon_{ij}f_j-2\sum_{i\in I_\uf}\br{e^*_i,n}\epsilon_{ik}f_k\nonumber\\
\hphantom{\frac1\nu(m_2'-m_1')=}{}
+\sum_{i\in I_\uf}\br{e^*_i,n}\epsilon_{ik}\sum_{j\in I_\uf}[\mp \epsilon_{kj}]_+f_j.\label{m condition case 1 simpler}
\end{gather}

We want to show that $\frac1\nu(m_2'-m_1')$ equals $\sum_{i\in I_\uf}\br{e^*_i,n'}\sum_{j\in I_\uf}\epsilon'_{ij}f_j$.
Using the definition of matrix mutation, we rewrite the latter as follows, with expressions like $i\neq k$ standing for $i\in I_\uf\setminus\set{k}$:
\begin{gather}
-\br{e^*_k,n'}\sum_{j\in I_\uf}\epsilon_{kj}f_j
-\sum_{i\in I_\uf}\br{e^*_i,n'}\epsilon_{ik}f_k\nonumber\\
\qquad {} +\sum_{i\neq k}\br{e^*_i,n'}\sum_{j\neq k}(\epsilon_{ij}+[\mp \epsilon_{kj}]_+\epsilon_{ik}+\epsilon_{kj}[\pm \epsilon_{ik}]_+)f_j.\label{m condition case 1 sum}
\end{gather}
Replacing $n'$ in \eqref{m condition case 1 sum} with its formula in terms of $n$, we can simplify to obtain the right side of~\eqref{m condition case 1 simpler}.

Finally, we establish \eqref{n condition prime} in this case.
Again, we may as well assume $\nu\neq0$.
Using \eqref{n 1 prime} and \eqref{n 2 prime}, we calculate
\begin{align*}
&\frac1\nu(n'_2-n'_1)=\frac1\nu(n_2-n_1)-\frac1\nu\br{m_2-m_1,d_ke_k}\sgn(\sigma_k)[\mp\sgn(\sigma_k)]_+e_k\\\nonumber
&\hphantom{\frac1\nu(n'_2-n'_1)=}{}
-\frac2\nu\br{e^*_k,n_2-n_1}e_k+\frac1\nu\sum_{i\in I_\uf}\br{e^*_i,n_2-n_1}[\sgn(\sigma_k)\epsilon_{ik}]_+e_k.
\end{align*}
Now \eqref{m condition} and \eqref{n condition} let us replace $n_2-n_1$ and $m_2-m_1$ with expressions in terms of $n$.
We simplify to obtain $\frac1\nu(n'_2-n'_1)=n'$, as desired.

\medskip

\noindent
\textbf{Case 2:}
\emph{$L_1$ and $L_2$ are on opposite sides of $e_k^\perp$.}
In this case, we have $n'=n=e_k$, and we use $\pm$ and $\mp$ with the top sign for the case where $L_1$ is on the negative side of $e_k^\perp$ and the bottom sign for the case where $L_1$ is on the positive side.
Thus $m_1'$ and $n_1'$ are described by~\eqref{m 1 prime} and~\eqref{n 1 prime}, while $m_2'$ and $n_2'$ are described by~\eqref{m 2 prime} and~\eqref{n 2 prime} with $\pm$ and $\mp$ reversed.
We see that $\br{m_1,d_ke_k}=-\br{m_1',d_ke_k}$, so that $\lambda'=\lambda$ in this case as well.
More specifically, $\lambda=\mp\br{m_1,d_ke_k}$.
We will show that~\eqref{c condition prime}, \eqref{m condition prime}, and~\eqref{n condition prime} hold with $\nu'=\lambda-\nu$.

Since $\tau(\zeta^n)=1+\zeta_k$, the binomial theorem says that $\nu\in\set{0,\dots,\lambda}$ and $a=\binom{\lambda}{\nu}$.
Similarly, $\tau((\zeta')^{n'})=1+\zeta'_k$, so $\nu'\in\set{0,\dots,\lambda}$ and $a'=\binom{\lambda}{\nu'}$, so~\eqref{c condition prime} holds.

Since $n=e_k$, \eqref{m condition} becomes $m_2=m_1+\nu\sum_{j\in I_\uf}\epsilon_{kj}f_j$.
Thus $\br{m_1,d_ke_k}=\br{m_2,d_ke_k}=\mp \lambda$.
Since $n'=e_k$, and since $\epsilon'_{kj}=-\epsilon_{kj}$ for all $j$, the desired condition \eqref{m condition prime} is the equation $m_2'=m_1'-(\lambda-\nu)\sum_{j\in I_\uf}\epsilon_{kj}f_j$.
Using \eqref{m 1 prime} and \eqref{m 2 prime}, the latter with $\pm$ and $\mp$ reversed, we compute
\begin{align*}
\begin{aligned}
m_2'-m_1'&=(m_2-m_1)\mp \lambda\sum_{j\in I_\uf}[\pm \epsilon_{kj}]_+f_j\pm \lambda\sum_{j\in I_\uf}[\mp \epsilon_{kj}]_+f_j\\\nonumber
&=\nu\sum_{j\in I_\uf}\epsilon_{kj}f_j\mp \lambda\sum_{j\in I_\uf}\bigl([\pm \epsilon_{kj}]_+-[\mp \epsilon_{kj}]_+\bigr)f_j
 =(\nu-\lambda)\sum_{j\in I_\uf}\epsilon_{kj}f_j.
\end{aligned}
\end{align*}

In this case, \eqref{n condition} is $n_2=n_1+\nu e_k$ and \eqref{n condition prime} is $n'_2=n'_1+(\lambda-\nu)e_k$.
Still keeping in mind that $\br{m_1,d_ke_k}=\br{m_2,d_ke_k}=\mp \lambda$, we compute
\begin{align*}
n'_2-n'_1&=(n_2-n_1)\pm \lambda\sgn(\sigma_k)[\pm\sgn(\sigma_k)]_+e_k\mp \lambda\sgn(\sigma_k)[\mp\sgn(\sigma_k)]_+e_k\\\nonumber
&\hspace{0.25in}-2\br{e^*_k,n_2-n_1}e_k+\sum_{i\in I_\uf}\br{e^*_i,n_2-n_1}[\sgn(\sigma_k)\epsilon_{ik}]_+ e_k\\\nonumber
&=\nu e_k+\lambda e_k-2\br{e^*_k,\nu e_k}e_k
 =(\lambda-\nu)e_k.
\end{align*}
We have verified \eqref{c condition prime}, \eqref{m condition prime}, and \eqref{n condition prime} in all cases.
\end{proof}

To prove Proposition~\ref{mut theta}, we need one more ingredient, namely \cite[Theorem~3.5]{GHKK}, which allows us to change the theta functions with endpoint~$Q'$ in Lemma~\ref{mut broken line} to theta functions with endpoint~$Q$.
The general statement of the theorem says that one can change the endpoints of~theta functions by applying a path-ordered product.
There is a path from $Q'$ to $Q$ that only crosses one wall, so the path-ordered product in this case is a single wall-crossing automorphism.~The wall is $\bigl(e_k^\perp,1+\zeta'_k\bigr)$, so the theorem says that \smash{$\thet_{Q,m'}^{\mu_k(\tB)}$} is obtained from \smash{$\thet_{Q',m'}^{\mu_k(\tB)}$} by replacing~$z'_k$ with $z'_k(1+\zeta_k')^{-1}$.

\begin{proof}[Proof of Proposition~\ref{mut theta}]
Choose $Q$ appropriately in the interior of $D$ and write~$Q'$ for $\eta_k^B(Q)$.
Since in particular $Q$ has $\br{Q,d_ke_k}>0$, Lemma~\ref{mut broken line} implies that \smash{$\thet_{Q',m'}^{\mu_k(\tB)}$} is obtained from~\smash{$\thet_{Q,m}^\tB$} as described in Proposition~\ref{mut theta}, except that~$z_k$ is replaced by
\[
(z_k')^{-1}(\sigma'_k)^{-[\sgn(\sigma_k)]_+}\prod_{j\in I_\uf}(z_j')^{[\epsilon_{kj}]_+}
\] (instead of the replacement described in Proposition~\ref{mut theta}).
By \cite[Theorem~3.5]{GHKK}, as explained above, \smash{$\thet_{Q,m'}^{\mu_k(\tB)}$} is now obtained by replacing~$z'_k$ with $z'_k(1+\zeta_k')^{-1}$.
The net effect is to replace~$z_k$~by
\[
(z'_k)^{-1}(1+\zeta_k')(\sigma'_k)^{-[\sgn(\sigma_k)]_+}\prod_{j\in I_\uf}(z_j')^{[\epsilon_{kj}]_+},
\] which is the replacement described in Proposition~\ref{mut theta}.
\end{proof}

As already explained (just after the statement of Proposition~\ref{mut theta}), this completes the proof of Theorem~\ref{2 muts}.

\begin{Example}\label{theta g2 eta1 replace ex}
This example continues and concludes Examples~\ref{theta g2 ex}, \ref{theta g2 eta1 ex}, and \ref{theta g2 eta1 broken ex}.
As before, write $m'=\eta_k^B(m)=[2,-3]$ and $Q'=\eta_k^B(Q)$.
From Example~\ref{theta g2 eta1 broken ex} (Figure~\ref{theta g2 eta1 fig}), we see that
\begin{align*}
&\thet_{Q',m'}=
(z'_1)^2(z'_2)^{-3}+2\sigma'_1(z'_1)^2+(\sigma'_1)^2(z'_1)^2(z'_2)^3\\
& \hphantom{\thet_{Q',m'}=}{}
+3\sigma'_2z'_1(z'_2)^{-3}+3\sigma'_1\sigma'_2z'_1
+3(\sigma'_2)^2(z'_2)^{-3}+(\sigma'_2)^3(z'_1)^{-1}(z'_2)^{-3}\\
&\hphantom{\thet_{Q',m'}}{}=
(z'_1)^2(z'_2)^{-3}\bigl(1+2\sigma'_1(z'_2)^3+(\sigma'_1)^2(z'_2)^6\\
&\hphantom{\thet_{Q',m'}=}{} +3\sigma'_2(z'_1)^{-1}+3\sigma'_1\sigma'_2(z'_1)^{-1}(z'_2)^3
+3(\sigma'_2)^2(z'_1)^{-2}+(\sigma'_2)^3(z'_1)^{-3}\bigr).
\end{align*}
Replacing $z_1'$ with $z_1'(1+\zeta'_1)^{-1}$, we obtain
\begin{align*}
&\thet^{\mu_1(\tB)}_{m'}=
(z'_1)^2(z'_2)^{-3}\biggl(\frac{1+2\sigma'_1(z'_2)^3+(\sigma'_1)^2(z'_2)^6}{(1+\zeta'_1)^2}
+\frac{3\sigma'_2(z'_1)^{-1}+3\sigma'_1\sigma'_2(z'_1)^{-1}(z'_2)^3}{1+\zeta'_1}\\
& \hphantom{\thet^{\mu_1(\tB)}_{m'}=}{}
+3(\sigma'_2)^2(z'_1)^{-2}+(\sigma'_2)^3(z'_1)^{-3}(1+\zeta'_1)\biggr)\\
&\hphantom{\thet^{\mu_1(\tB)}_{m'}}{}=
(z'_1)^2(z'_2)^{-3}\biggl(\frac{1+2\zeta'_1+(\zeta'_1)^2}{(1+\zeta'_1)^2}
+\frac{3\zeta'_2+3\zeta'_1\zeta'_2}{1+\zeta'_1}
+3(\zeta'_2)^2+(\zeta'_2)^3(1+\zeta'_1)\biggr)\\
&\hphantom{\thet^{\mu_1(\tB)}_{m'}}{} =
(z'_1)^2(z'_2)^{-3}\bigl(1+3\zeta'_2
+3(\zeta'_2)^2+(\zeta'_2)^3+\zeta'_1(\zeta'_2)^3\bigr).
\end{align*}
\end{Example}

\section{Applications of mutation of theta functions}\label{app sec}

\subsection{Mutation symmetry} \label{sym sec}
A \emph{mutation symmetry} of an exchange matrix $B$ is a sequence $\kk$ of indices such that $\mu_\kk(B)=B$.
A mutation symmetry induces a symmetry of cluster scattering diagrams.
In this section, we state and prove our main result, which simplifies some structure constant computations in the presence of a mutation-symmetry.
Recall from Section~\ref{struct sec} that $\Theta\subseteq M^\circ_\uf$ is the set of vectors $m\in M^\circ_\uf$ such that only finitely many broken lines figure into the definition of $\thet_m$.

\begin{Theorem}\label{finite orbit}
Suppose $\tB$ has signed-nondegenerating coefficients and suppose~$\kk$ is a mutation symmetry of~$B$.
Let $\v$ be a monomial in a finite set $\set{\thet_p:p\in P}$ of theta functions with $P\subset\Theta$, expressed as \smash{$\v=\sum_{m\in M^\circ_\uf}\sum_{n\in N_\uf}c_{m,n}\,\sigma^n\thet_m$} in the theta basis.
If each $p\in P$ is in a finite $\eta^B_\kk$-orbit but $m\in M^\circ_\uf$ is in an infinite $\eta^B_\kk$-orbit, then $c_{m,n}=0$ for all $n\in N_\uf$.
\end{Theorem}

We prove Theorem~\ref{finite orbit} by way of the following proposition with the same hypotheses except that there is no requirement that $P\subset\Theta$.

\begin{Proposition}\label{orbit union}
Suppose $\tB$ has signed-nondegenerating coefficients and suppose~$\kk$ is a mutation symmetry of~$B$.
Let $\v$ be a monomial in a finite set $\set{\thet_p:p\in P}$ of theta functions and~write \smash{$\v=\sum_{m\in M^\circ_\uf}\sum_{n\in N_\uf}c_{m,n} \sigma^n\thet_m$}.
If there exists $\ell\ge0$ such that \smash{$\bigl(\eta^B_\kk\bigr)^\ell$} fixes each $p\in P$, then $\set{m\in M^\circ_\uf:\exists n\in N_\uf\text{ with }c_{m,n}\neq0}$
is a union of \smash{$\bigl(\eta^B_\kk\bigr)^\ell$}-orbits.
\end{Proposition}

\begin{proof}
Write $\eta$ as a shorthand for \smash{$\bigl(\eta^B_\kk\bigr)^\ell=\eta^B_{\kk^\ell}$} and write $\mu$ as shorthand for $\mu_{\kk^\ell}$.
Continue, from the end of Section~\ref{coeff sec}, the notation $(\sigma^{(\kk)}_i:i\in I_\uf)$ for the Laurent monomials in $\set{z_j:j\in I_\fr}$ associated to \smash{$\mu_\kk\bigl(\tB\bigl)$}, and write ~$r$ for the map that replaces each $\sigma_i$ by \smash{$\sigma^{(\kk^\ell)}_i$}.
The essence of the proof is to combine Proposition~\ref{where prin} with Theorem~\ref{2 muts}.

Because $\mu(B)=B$, Proposition~\ref{where prin} lets us pass between theta functions for $\tB$ and theta functions for $\mu\bigl(\tB\bigr)$ using the map~$r$.
Specifically, \smash{$r(\thet^\tB_m)=\thet^{\mu(\tB)}_m$} for any $m\in M^\circ_\uf$.
(We emphasize that, although $\mu(B)=B$, typically \smash{$\mu\bigl(\tB\bigr)\neq\tB$}.)

Theorem~\ref{2 muts} lets us pass between theta functions for $\tB$ and theta functions for $\mu\bigl(\tB\bigr)$ by applying $\eta$ and multiplying by a Laurent monomial in the $\sigma_i$.
For each $m\in M^\circ_\uf$, we apply Theorem~\ref{2 muts} many times.
Keeping in mind that~$\tB$ has signed-nondegenerating coefficients, we~see that \smash{$\thet^{\mu(\tB)}_{\eta(m)}$} equals \smash{$\thet^\tB_m$} times a Laurent monomial in the $\sigma_i$.

Suppose $p\in P$, so that $\eta(p)=p$ by hypothesis.
Then \smash{$r\bigl(\thet^\tB_p\bigr)$} is $\thet^\tB_p$ times a Laurent monomial in the $\sigma_i$.
Since $\v$ is a monomial in \smash{$\bigl\{\thet^\tB_p:p\in P\bigr\}$}, also $r(\v)$ is~$\v$ times a Laurent monomial in the~$\sigma_i$, specifically $r(\v)=\sigma^{q}\v$ for some $q\in N_\uf$.

Apply $r$ to both sides of the equation \smash{$\v=\sum_{m\in M^\circ_\uf}\sum_{n\in N_\uf}c_{m,n} \sigma^n\thet^\tB_m$} and solve for $\v$, to obtain
\[\v=\sum_{m\in M^\circ_\uf}\sum_{n\in N_\uf}c_{m,n}\bigl(\sigma^{(\kk^\ell)}\bigr)^n\sigma^{-q}\thet^{\mu(\tB)}_m.\]
Since $\eta$ is a permutation of $M^\circ_\uf$, we can reindex the sum as
\[\v=\sum_{m\in M^\circ_\uf}\sum_{n\in N_\uf}c_{\eta(m),n}\bigl(\sigma^{(\kk^\ell)}\bigr)^n\sigma^{-q}\thet^{\mu(\tB)}_{\eta(m)}.\]
Since \smash{$\thet^{\mu(\tB)}_{\eta(m)}$} equals $\thet^\tB_m$ times a Laurent monomial in the $\sigma_i$, this is
\[\v=\sum_{m\in M^\circ_\uf}\sum_{n\in N_\uf}c_{\eta(m),n}\bigl(\sigma^{(\kk^\ell)}\bigr)^n\sigma^{q_m-q}\thet^{\tB}_m,\]
where $q_m$ is a vector in $N_\uf$ (depending on $m$) that need not be specified.
Comparing this formula for $\v$ to the original equation \smash{$\v=\sum_{m\in M^\circ_\uf}\sum_{n\in N_\uf}c_{m,n} \sigma^n\thet^\tB_m$} and remembering that the theta functions are a reduced basis, we conclude, for all $m\in M^\circ_\uf$, that there exists $n$ such that $c_{\eta(m),n}\neq0$ if and only if there exists $n$ such that $c_{m,n}\neq0$.
\end{proof}

Under the hypotheses of Proposition~\ref{orbit union}, if also $P\subset\Theta$, then Theorem~\ref{Theta facts} implies that the set \smash{$\bigl\{m\in M^\circ_\uf:\sum_{n\in N_\uf}c_{m,n} \sigma^n\neq0\bigr\}$} is finite.
Thus we have proved Theorem~\ref{finite orbit}.

\subsection{Mutation of pairs of broken lines}\label{mut pair sec}
We now develop another tool that is useful for computing structure constants for multiplication of theta functions.
(For example, this tool is essential in the case where $B$ is of acyclic affine type, treated in~\cite{afftheta}.)
The point of the following proposition is that it gives conditions under which mutation takes a pair of broken lines that contributes to structure constants to another pair of broken lines that contributes to structure constants.

At each step in applying a mutation map $\eta_\kk^B$ to a vector, there are two different cases.
Specifically, at step $i$, the cases depend on which side of the hyperplane \smash{$e_{k_i}^\perp$} the output of \smash{$\eta^B_{k_{i-1}\cdots k_1}$} is on.
Two vectors are in the same \emph{domain of definition} of \smash{$\eta_\kk^B$} if, at every step, the same case applies to both vectors.
That is, at every step~$i$, the output of \smash{$\eta^B_{k_{i-1}\cdots k_1}$} for the two vectors is weakly on the same side of \smash{$e_{k_i}^\perp$}.
Recall the extension of \smash{$\eta_\kk^B$} to a map on broken lines, defined in connection with Lemma~\ref{mut broken line}.

\begin{Proposition}\label{mut pair}
Suppose $\tB$ has signed-nondegenerating coefficients, let $m\in M_\uf^\circ$, suppose $Q\in V^*$ is not contained in any wall of $\Scat\bigl(\tB\bigr)$, and let $\kk$ be a sequence of indices.
If~$m$ and~$Q$ are in the same domain of definition of~$\eta_\kk^B$, then a pair $(\bl_1,\bl_2)$ of broken lines contributes to $a_Q(p_1,p_2,m)$ if and only if \smash{$\bigl(\eta_\kk^B(\bl_1),\eta_\kk^B(\bl_2)\bigr)$} contributes to \smash{$a_{\eta_\kk^B(Q)}\bigl(\eta_\kk^B(p_1),\eta_\kk^B(p_2),\eta_\kk^B(m)\bigr)$}.
\end{Proposition}
\begin{proof}
It is enough to prove the case where $\kk$ consists of a single index~$k$.
Suppose $p_1,p_2\in M^\circ_\uf$.
Let $\bl_1$ and~$\bl_2$ be broken lines for $p_1$ and $p_2$ respectively, each having endpoint $Q$.
Let~$\bl_1'$ and~$\bl_2'$ be the broken lines for \smash{$\eta^B_k(p_1)$} and \smash{$\eta^B_k(p_1)$}, respectively, each having endpoint $\eta_k^B(Q)$, defined by the construction in Lemma~\ref{mut broken line}.
In light of Lemma~\ref{mut broken line}, the proposition amounts to showing that $m_{\bl_1}+m_{\bl_2}=m$ if and only if \smash{$m_{\bl_1'}+m_{\bl_2'}=\eta_k^B(m)$}.
Because \smash{$\bigl(\eta_k^B\bigr)^{-1}=\eta_k^{\mu_k(B)}$}, it is enough to prove one direction of implication.

Suppose $m_{\bl_1}+m_{\bl_2}=m$.
The map $\eta^B_k$ is piecewise linear, with two domains of linearity separated by $e_k^\perp$.
Since $Q$ is not in any wall of $\Scat\bigl(\tB\bigr)$ and because $e_k^\perp$ is a wall of $\Scat\bigl(\tB\bigr)$, the point $Q$ is not on $e_k^\perp$.
Therefore the domain of linearity of $\bl_1$ containing $0$ has a piece that is strictly on the same side of $e_k^\perp$ as $Q$, and therefore weakly on the same side of $e_k^\perp$ as $m$.
Since $m_{\bl_1}$ is the negative of the derivative of $\bl_1$ on that domain of linearity, $m_{\bl_1'}$ is obtained from $m_{\bl_1}$ by the same linear map that takes~$m$ to \smash{$\eta^B_k(m)$}.
The same is true for $m_{\bl_2'}$, and we conclude that \smash{$m_{\bl_1'}+m_{\bl_2'}=\eta_k^B(m)$}.
\end{proof}

\subsection[Dominance regions, B-cones, and pointed elements]{Dominance regions, $\boldsymbol{B}$-cones, and pointed elements}\label{dom sec}
Given $m\in M^\circ_\uf$ and~$\kk$, define \smash{$\Dom^B_{m,\kk}=\bigl(\eta_{\kk}^B\bigr)^{-1}\bigl\{\eta_\kk^B(m)+n\cdot\mu_\kk(B):n\in N^{0+}_\uf\bigr\}\subseteq M^\circ_\uf$}.
The \emph{integral dominance region} of~$m$ with respect to $B$ is $\Dom^B_m=\bigcap_\kk\Dom^B_{m,\kk}$, where the intersection is over all sequences~$\kk$.
The \emph{$($real$)$ dominance region} of~$m$ with respect to~$B$ is defined in the same way, but replacing $n\in N^{0+}_\uf$ everywhere with a vector in $V$ with nonnegative entries.
The inclusion of $M^\circ_\uf$ into $V^*$ sends the integral dominance region to a subset of the real dominance region.
The definition of the dominance region originated, in the form of a partial order on $M^\circ$, in the work of Fan Qin~\cite{FanQin}
and has been defined and studied in the form of a set of points by Rupel and Stella~\cite{RupelStella}.
The original motivation of the definition was to study bases for the upper cluster algebra, and in Section~\ref{prb sec} we will make some statements in the style of~\cite{FanQin} about dominance regions and certain special bases for the small canonical algebra.
However, our main motivation for discussing dominance regions here is for computing structure constants for theta functions, specifically the following theorem.

\begin{Theorem}\label{B cone prod}
Suppose that $\tB$ has signed-nondegenerating coefficients and that $m_1,\dots,m_\ell$ are all contained in the same $B$-cone.
Write $m=a_1m_1+\cdots+a_\ell m_\ell$ for nonnegative integers $a_1,\dots,a_\ell$.
Then there exist constants $c_{p,n}\in\k$ such that \smash{$\thet_{m_1}^{a_1}\cdots\thet_{m_\ell}^{a_\ell}=\thet_m+\sum_p\sum_nc_{p,n}\sigma^n\thet_p$}, summing over $p\in\Dom^B_m$ and $n\in N^+_\uf$ such that $p=m+nB$.
\end{Theorem}

As stated, Theorem~\ref{B cone prod} emphasizes that the only theta functions occurring in the expansion of $\thet_{m_1}^{a_1}\cdots\thet_{m_\ell}^{a_\ell}$ are $\thet_p$ for $p\in\Dom^B_m$.
Alternatively, the conclusion of the corollary can be written as \smash{$\thet_{m_1}^{a_1}\cdots\thet_{m_\ell}^{a_\ell}=\thet_m+\sum_nc_n\sigma^n\thet_{m+nB}$}, summing over $n\in N^+_\uf$ such that $m+nB\in\Dom^B_m$, with coefficients $c_n\in\k$.

We will prove Theorem~\ref{B cone prod} as a corollary of Theorem~\ref{B cone prod N}, below, which is stronger in the sense that it gives a smaller set of pairs $(n,p)\in N_\uf^+\times M^\circ_\uf$ such that~$\sigma^n\thet_p$ can appear with nonzero coefficient in the expansion of $\thet_{m_1}^{a_1}\cdots\thet_{m_\ell}^{a_\ell}$.
However, Theorem~\ref{B cone prod} has the advantage that it is phrased in terms of the dominance region~$\Dom^B_m$.
The dominance region takes considerable effort to compute, but has been successfully computed when $\tB=B$ is $2\times2$ and in affine type.
(See \cite{RupelStella} and \cite{affdomreg}.)
Theorem~\ref{B cone prod N} is phrased in terms of a subset of $N_\uf$ that is new to this paper and appears to be even more complicated.

We now prepare to state and prove the stronger theorem (Theorem~\ref{B cone prod N}).
One key to the proof is that $\thet_{m_1}^{a_1}\cdots\thet_{m_\ell}^{a_\ell}$ is a pointed element of $\can\bigl(\tB\bigr)$, in the sense that we now define.
Given $\kk$, write \smash{$\bigl(\mu_\kk(B),\bigl(z^{(\kk)}_i:i\in I_\uf\bigr),\bigl(\sigma^{(\kk)}_i:i\in I_\uf\bigr)\bigr)$} for the seed obtained by mutating $(B,(z_i:i\in I_\uf),(\sigma_i:i\in I_\uf))$ along~$\kk$, and define \smash{$\bigl(\zeta^{(\kk)}_i:i\in I_\uf\bigr)$} accordingly.
As explained in Section~\ref{mut sec}, using \eqref{exch rel}--\eqref{hy mut} for each index in the sequence~$\kk$ we can write each $z_i$ and $\zeta_i$ in terms of the \smash{$z_i^{(\kk)}$} and \smash{$\zeta_i^{(\kk)}$}.
An element $\u\in\can\bigl(\tB\bigr)$ is \emph{pointed} if it is $z^m$ (for some $m\in M^\circ_\uf$) times a formal power series in $\k[[\zeta]]$ with constant coefficient~$1$ and, for any sequence $\kk$, it is \smash{$\bigl(z^{(\kk)}\bigr)^{\eta_\kk^{B}(m)}$} times a Laurent monomial in $\sigma^{(\kk)}$ times a formal power series in \smash{$\k\bigl[\bigl[\zeta^{(\kk)}\bigr]\bigl]$} with constant coefficient~$1$.
This is a version of a notion called ``pointed at the tropical point $[g]$'' in~\cite{FanQin}.
The motivating examples of pointed elements are the theta functions.
Iterations of Theorem~\ref{2 muts} imply the following proposition.

\begin{Proposition}\label{theta pointed}
Suppose $\tB$ has signed-nondegenerating coefficients.
Then the theta basis $\set{\thet_m:m\in M^\circ_\uf}$ consists of pointed elements.
\end{Proposition}

Another key to the proof is a precise characterization of pointed elements.
For each sequence $\kk$ and each ${m\in M^\circ_\uf}$, let $(\kappa(m,\kk),\phi(m,\kk))\in M^\circ_\uf\oplus N_\uf$ be such that~$\thet_m^{\tB}$ is \smash{$\bigl(\sigma^{(\kk)}\bigr)^{\phi(m,\kk)}\bigl(z^{(\kk)}\bigr)^{\kappa(m,\kk)}$} times a formal power series in $\k\bigl[\bigl[\zeta^{(\kk)}\bigr]\bigr]$ with constant coefficient~$1$.
The vectors $\kappa(m,\kk)$ and $\phi(m,\kk)$ exist in light of Proposition~\ref{theta pointed}.
Indeed, Theorem~\ref{2 muts} implies that \smash{$\kappa(m,\kk)=\eta_\kk^{B}(m)$}.

We define a map from $N_\uf$ to itself that lets us write a Laurent monomial in $(\sigma_i:i\in I_\uf)$ in terms of $\bigl(\sigma^{(\kk)}_i:i\in I_\uf\bigr)$.
For an index $k$, define
\begin{equation}\label{zeta k}
\psi_k^\tB(n)=n-2\br{e^*_k,n}e_k+\sum_{i\in I_\uf}\br{e^*_i,n}[\sgn(\sigma_k)\epsilon_{ik}]_+e_k,
\end{equation}
and for a sequence $\kk=k_q\cdots k_1$, define
\begin{equation*}
\psi_\kk^\tB=\psi_{k_q}^{\mu_{k_{q-1},\dots,k_1}(\tB)}\circ\psi^{\mu_{k_{q-2},\dots,k_1}(\tB)}_{k_{q-1}}\circ\cdots\circ\psi^{\mu_{k_1}(\tB)}_{k_2}\circ\psi^\tB_{k_1}.
\end{equation*}
When $\kk$ is the empty sequence, $\psi_\kk^\tB$ is the identity map on $N_\uf$.
Proposition~\ref{mutate subs} and an easy induction on $q$ show that \smash{$\sigma^n=\bigl(\sigma^{(\kk)}\bigr)^{\psi_\kk^\tB(n)}$} for any $n\in N_\uf$.
We emphasize that $\psi_\kk^\tB$ depends on~$\tB$, not just~$B$, because $\sgn(\sigma_k)$ appears in~\eqref{zeta k}.
For the same reason, we define $\psi_\kk^\tB$ only when $\tB$ has signed-nondegenerating coefficients.

Given $m\in M^\circ_\uf$ and $n\in N^+_\uf$ and a sequence $\kk$, define
\[\nu^{(m)}_\kk\!(n)=\psi_\kk^\tB(n)+\phi(m+nB,\kk)-\phi(m,\kk).\]
For any $m\in M^\circ_\uf$ and sequence $\kk$, define
\[\N_{m,\kk}^\tB=\sett{n\in N^{0+}_\uf:\,\eta_\kk^B(m+nB)-\eta_\kk^B(m)=\nu^{(m)}_\kk\!(n)\cdot\mu_\kk(B),\,\,\,\nu^{(m)}_\kk\!(n)\in N^{0+}_\uf}\]
Finally, for any $m\in M^\circ_\uf$, define
\[\N_m^\tB=\bigcap_\kk\N_{m,\kk}^\tB\quad\text{(the intersection over all sequences $\kk$ of indices in $I_\uf$).}\]
We will prove the following theorem.

\begin{Theorem}\label{mut point}
Suppose $\tB$ has signed-nondegenerating coefficients and suppose $\u\in\can\bigl(\tB\bigr)$.
The following are equivalent.
\begin{enumerate}[label=\rm(\roman*), ref=(\roman*)]\itemsep=0pt
\item \label{is pointed}
$\u$ is pointed and $\g(\u)=m$.
\item \label{like theta}
For all sequences $\kk$, $\u$ is \smash{$\bigl(\sigma^{(\kk)}\bigr)^{\phi(m,\kk)}\bigl(z^{(\kk)}\bigr)^{\kappa(m,\kk)}$} times a formal power series $\k[[\zeta^{(\kk)}]]$ with constant coefficient~$1$.
\item \label{N}
There exist constants $c_n\in\k$ with $c_0=1$ such that \smash{$\u=\sum_{n\in \N_m^\tB}c_n \sigma^n\thet_{m+nB}$}.
\end{enumerate}
\end{Theorem}

Since \smash{$\kappa(m,\kk)=\eta_\kk^{B}(m)$}, the equivalence of Conditions~\ref{is pointed} and~\ref{like theta} in Theorem~\ref{mut point} amounts to a specification of the monomials in $\sigma^{(\kk)}$ that appear when a pointed element is written as \smash{$\bigl(z^{(\kk)}\bigr)^{\eta_\kk^{B}(m)}$} times a Laurent monomial in $\sigma^{(\kk)}$ times a formal power series in $\k\bigl[\bigl[\zeta^{(\kk)}\bigr]\bigr]$ with constant coefficient~$1$.

Recall that $\can\bigl(\mu_\kk\bigl(\tB\bigr)\bigr)=\can\bigl(\tB\bigr)$ for any sequence $\kk$ by Corollary~\ref{mut can}.
However, the notion of a reduced basis for $\can\bigl(\tB\bigr)$ from Section~\ref{bases sec} implicitly chooses the description of the small canonical algebra as $\can\bigl(\tB\bigr)$ rather than $\can\bigl(\mu_\kk\bigl(\tB\bigr)\bigr)$.
For the proof of Theorem~\ref{mut point}, it will be useful to make that choice explicit.

A subset~$\U\subseteq\can\bigl(\tB\bigr)$ is a \emph{reduced basis at $\kk$} if $\set{\sigma^n\u:n\in N_\uf,\u\in\U}$ is a basis for $\can\bigl(\tB\bigr)$ and $\U$ is of the form $\set{\u_m:m\in M^\circ_\uf}$ such that each~$\u_m$ is $(z^{(\kk)})^m$ times a formal power series in \smash{$\k\bigl[\bigl[\zeta^{(\kk)}\bigr]\bigr]$} with constant coefficient~$1$.
(Because~$\tB$ has signed-nondegenerating coefficients, we can write $\set{\sigma^n\u:n\in N_\uf,\u\in\U}$ or \smash{$\bigl\{\bigl(\sigma^{(\kk)}\bigr)^n\u:n\in N_\uf,\u\in\U\bigr\}$} interchangeably.)
The term ``reduced basis'', without referring to a specific $\kk$, will continue to mean a reduced basis at the initial seed ($\kk=\varnothing$).

An element $\u$ is \emph{pointed at $\kk$} if it is $\bigl(z^{(\kk)}\bigr)^m$ (for some $m\in M^\circ_\uf$) times a formal power series in \smash{$\k\bigl[\bigl[\zeta^{(\kk)}\bigr]\bigr]$} with constant coefficient~$1$ and, for any sequence $\ellell$, it is \smash{$\bigl(z^{(\ellell\kk)}\bigr){}^{\eta_{\ellell}^{\mu_{\kk}(B)}(m)}$} times a~Laurent monomial in $\sigma^{(\ellell\kk)}$ times a formal power series in \smash{$\k\bigl[\bigl[\zeta^{(\ellell\kk)}\bigr]\bigr]$} with constant coefficient~$1$.
The term ``pointed'', not referring to~$\kk$, will continue to mean pointed at $\kk=\varnothing$.

\begin{proof}[Proof of Theorem~\ref{mut point}]
Suppose \ref{is pointed}.
By definition, there is an expression for $\u$ as~$z^m$ times a formal power series in $H\in\k[[\zeta]]$ with constant coefficient~$1$.
Rewrite this expression, using iterations of Proposition~\ref{mutate subs}, to obtain an expression of $\u$ in terms of $z^{(\kk)}$, $\sigma^{(\kk)}$, and $\zeta^{(\kk)}$.
Since~$\u$ is pointed, this is $(z^{(\kk)})^{\eta_\kk^{B}(m)}$ times a Laurent monomial in $\sigma^{(\kk)}$ times a formal power series in $\k\bigl[\bigl[\zeta^{(\kk)}\bigr]\bigr]$ with constant coefficient~$1$.
By induction on the length of~$\kk$, because the exponent on the $z^{(\kk)}$ is dictated by the definition of a pointed element, we see that the exponent vector on~$\sigma^{(\kk)}$ depends only on~$m$ and $\kk$, not on~$H$.
Thus these exponent vectors are the same as if~$\u$ were~$\thet_m$.
That is, they are $\phi(m,\kk)$ and $\kappa(m,\kk)$.
We have shown that~\ref{is pointed} implies~\ref{like theta}.
The converse implication is immediate, so~\ref{is pointed} and~\ref{like theta} are equivalent.

Suppose \ref{is pointed} and \ref{like theta} hold.
Since $\u\in\can\bigl(\tB\bigr)$, we can express $\u$ in the theta basis as \smash{$\u=\sum_{p\in M^\circ_\uf}\sum_{n\in N_\uf}c_{p,n} \sigma^n \thet_p$}.
Suppose $c_{p,n}\neq0$ for some $p\in M^\circ_\uf$ and $n\in N_\uf$.
Lemma~\ref{point el}.\ref{c B n} says that $p=m+nB$.

Choose a sequence $\kk$, so that~\ref{like theta} says that \smash{$\u=\bigl(\sigma^{(\kk)}\bigr)^{\phi(m,\kk)}(z^{(\kk)})^{\kappa(m,\kk)}$} times a formal power series in \smash{$\k\bigl[\bigl[\zeta^{(\kk)}\bigr]\bigr]$} with constant coefficient~$1$.
Write $\u^{(\kk)}$ to stand for \smash{$\u\cdot\bigl(\sigma^{(\kk)}\bigr)^{-\phi(m,\kk)}$}.
Since $\u$ is pointed (at~$\varnothing$), also $\u^{(\kk)}$ is pointed at $\kk$.
Each~$\thet_p$ is \smash{$\bigl(\sigma^{(\kk)}\bigr)^{\phi(p,\kk)}(z^{(\kk)})^{\kappa(p,\kk)}$} times a formal power series in \smash{$\k\bigl[\bigl[\zeta^{(\kk)}\bigr]\bigr]$} with constant coefficient~$1$.
Write \smash{$\thet^{(\kk)}_{\kappa(p,\kk)}$} for \smash{$\bigl(\sigma^{(\kk)}\bigr)^{-\phi(p,\kk)}\thet_p$}.
Theorem~\ref{pointed set basis} says that the set of all these~\smash{$\thet^{(\kk)}_{p^{(\kk)}}$} is a reduced basis at~$\kk$.

Starting from \smash{$\u=\sum_{p\in M^\circ_\uf}\sum_{n\in N_\uf}c_{p,n} \sigma^n \thet_p$}, we replace $\u$ and each $\thet_p$ by appropriate exp\-ressions in terms of $\u^{(\kk)}$ and $\thet^{(\kk)}_{\kappa(p,\kk)}$ and rewrite $\sigma^n$ as \smash{$\bigl(\sigma^{(\kk)}\bigr)^{\psi_\kk^\tB(n)}$} to obtain
\begin{align*}
\u^{(\kk)}&=\sum_{p\in M^\circ_\uf}\sum_{n\in N_\uf}\!c_{p,n}\bigl(\sigma^{(\kk)}\bigr)^{\psi_\kk^\tB(n)+\phi(p,\kk)-\phi(m,\kk)} \thet^{(\kk)}_{\kappa(p,\kk)}
 =\sum_{p\in M^\circ_\uf}\sum_{n\in N_\uf}\!c_{p,n}\bigl(\sigma^{(\kk)}\bigr)^{\nu_\kk^{(m)}\!(n)} \thet^{(\kk)}_{\kappa(p,\kk)}.
\end{align*}
Since $c_{p,n}\neq0$, applying Lemma~\ref{point el}\,(\ref{c B n}) for the element~$\u^{(\kk)}$, which is pointed at~$\kk$ and the set of elements \smash{$\thet^{(\kk)}_{p^{(\kk)}}$}, which forms a reduced basis at~$\kk$, we see that \smash{$\nu_\kk^{(m)}\!(n)\in N_\uf^{0+}$} and
\[
\kappa(p,\kk)=\kappa(m,\kk)+\nu_\kk^{(m)}\!(n)\cdot\mu_\kk(B).
\]
Thus \smash{$\eta_\kk^B(m+nB)-\eta_\kk^B(m)=\nu_\kk^{(m)}\!(n)\cdot\mu_\kk(B)$}.
This is true for any sequence $\kk$, so $n\in\N_m^\tB$.
Setting $c_n=c_{m+nB,n}$ we have \smash{$\u=\sum_{n\in\N_m^\tB}c_n\sigma^n\thet_{m+nB}$}, with $c_0=1$ (by~\ref{like theta} at $\kk=\varnothing$).
We have shown that~\ref{is pointed} and~\ref{like theta} imply \ref{N}.

Finally, suppose \smash{$\u=\sum_{n\in\N_m^\tB}c_n\sigma^n\thet_{m+nB}$} with $c_0=1$ and let $\kk$ be any sequence.
Each~$\thet_{m+nB}$ in the sum is \[
\bigl(\sigma^{(\kk)}\bigr)^{\phi(m+nB,\kk)}\bigl(z^{(\kk)}\bigr)^{\kappa(m+nB,\kk)}
\] times a formal power series \smash{$F_{m+nB}^{(\kk)}$}\! in \smash{$\k\bigl[\bigl[\zeta^{(\kk)}\bigr]\bigr]$}\! with constant coefficient~$1$.\!
Replacing each~$\thet_{m{+}nB}$, replacing $\sigma^n$ by \smash{$\bigl(\sigma^{(\kk)}\bigr){}^{\psi_\kk^\tB(n)}$}, and simplifying, we see that
\[\u=\bigl(\sigma^{(\kk)}\bigr)^{\phi(m,\kk)}\bigl(z^{(\kk)}\bigr)^{\kappa(m,\kk)} \sum_{n\in\N_m^\tB}c_n\bigl(z^{(\kk)}\bigr)^{\eta_\kk^B(m+nB)- \eta_\kk^B(m)}\bigl(\sigma^{(\kk)}\bigr)^{\nu^{(m)}_\kk\!(n)}F_{m+nB}^{(\kk)}.\]
Since each $n$ in the sum is in $\N_m^\tB$ and since \smash{$\bigl(z^{(\kk)}\bigr)^{\nu B}\bigl(\sigma^{(\kk)}\bigr)^\nu=\bigl(\zeta^{(\kk)}\bigr)^\nu$} for all $\nu\in N_\uf$, the sum~is
\smash{$\sum_{n\in\N_m^\tB}c_n\bigl(\zeta^{(\kk)}\bigr){}^{\nu^{(m)}_\kk\!(n)}F_{m+nB}^{(\kk)}$},
a formal power series in \smash{$\k\bigl[\bigl[\zeta^{(\kk)}\bigr]\bigr]$} with constant coefficient~$1$.
Since this is true for all sequences~$\kk$, we have established \ref{like theta} and thus completed the proof that~\ref{is pointed},~\ref{like theta}, and~\ref{N} are all equivalent.
\end{proof}

We now state the stronger theorem that will imply Theorem~\ref{B cone prod}.

\begin{Theorem}\label{B cone prod N}
Suppose that $\tB$ has signed-nondegenerating coefficients and that $m_1,\dots,m_\ell$ are all contained in the same $B$-cone.
Write $m=a_1m_1+\cdots+a_\ell m_\ell$ for nonnegative integers $a_1,\dots,a_\ell$.
Then there exist constants $c_n\in\k$ with $c_0=1$ such that
\[
\thet_{m_1}^{a_1}\cdots\thet_{m_\ell}^{a_\ell}=\sum_{n\in\N_m^\tB}c_n\sigma^n\thet_{m+nB}.
\]
\end{Theorem}

We now discuss the relationship between Theorems~\ref{B cone prod} and~\ref{B cone prod N} by giving first, a lemma that shows that Theorem~\ref{B cone prod N} implies Theorem~\ref{B cone prod} and second, an example that shows that Theorem~\ref{B cone prod N} is strictly stronger.

\begin{Lemma}\label{N in Dom}
For any exchange matrix $B$, if $n\in\N_m^\tB$, then $m+nB\in\Dom^B_m$.
\end{Lemma}
\begin{proof}
Suppose $n\in\N_m^\tB$.
For any $\kk$, $\eta_\kk^B(m+nB)-\eta_\kk^B(m)=\nu^{(m)}_\kk\!(n)\cdot\mu_\kk(B)$ and $\nu^{(m)}_\kk\!(n)\in N^{0+}_\uf$.
Thus
\[m+nB\in\bigl(\eta_{\kk}^B\bigr)^{-1}\bigl\{\eta_\kk^B(m)+\nu\cdot\mu_\kk(B) : \nu\in N^{0+}_\uf\bigr\}=\Dom^B_{m,\kk}.\tag*{\qed}\]
\renewcommand{\qed}{}
\end{proof}

\begin{Example}\label{Markov N}
Consider
\[
B=\left[\begin{matrix}
\hphantom{-}0&\hphantom{-}2&-2\\
-2&\hphantom{-}0&\hphantom{-}2\\
\hphantom{-}2&-2&\hphantom{-}0
\end{matrix}\right] \qquad
\text{and}\qquad
 \tB=\left[\begin{matrix}
\hphantom{-}0&\hphantom{-}2&-2& 1&0&0\\
-2&\hphantom{-}0&\hphantom{-}2&0& 1&0\\
\hphantom{-}2&-2&\hphantom{-}0&0&0& 1
\end{matrix}\right].
\]
This is the signed-adjacency matrix of the once-punctured torus, also known as the \emph{Markov quiver}, with principal coefficients.
Any mutation of $B$ is $\mu_\kk(B)=\pm B$, with the sign given by the parity of the length of $\kk$.
One can check that all mutation maps $\eta^B_\kk$ preserve the sum of the $f_i$-coordinates.
One can also check that the nonnegative span of $B$ or $-B$ is the subspace consisting of vectors whose $f_i$-coordinates sum to~$0$.
From there, it is not difficult to see that $\Dom^B_{m,\kk}$ is the set of all vectors whose $f_i$-coordinates have the same sum as the sum of $f_i$-coordinates of $m$.
Thus $\Dom_m^B$ has the same description, so Theorem~\ref{B cone prod} allows terms $c_n\sigma^n\thet_{m+nB}$ for all $n\in N_\uf^{0+}$.
Theorem~\ref{B cone prod N} allows fewer of these terms.
For example, writing $m=[1,1,1]$ and $n=[n_1,n_2,n_3]$ and taking $\kk$ to be the singleton sequence $1$, we compute
\begin{align*}
&\psi_1^B(n)=[-n_1+2n_3,n_2,n_3],\\
&\phi(m+nB,1)=\bigl[-[1-2n_2+2n_3]_+,0,0\bigr],\\
&\phi(m,1)=[-1,0,0],\\
&\nu_1^{(m)}(n)=
\begin{cases}
[-n_1+2n_3+1,n_2,n_3]&\text{if }n_2>n_3,\\
[-n_1+2n_2,n_2,n_3]&\text{if }n_2\le n_3.
\end{cases}
\end{align*}
The requirement that $\nu_1^{(m)}(n)\in N_\uf^{0+}$ amounts to $n_1\le\max(2n_2,2n_3+1)$, so $\N_{m,1}^\tB$ is strictly smaller than $N_\uf^{0+}$.
Interestingly, the requirement that $\eta_1^B(m+nB)-\eta_1^B(m)$ equals \smash{$\nu^{(m)}_1\!(n)\cdot\mu_1(B)$} is vacuous, so
\[\N_{m,1}^\tB=\bigl\{n=[n_1,n_2,n_3]\in N_\uf^{0+}: n_1\le\max(2n_2,2n_3+1)\bigr\}.\]
\end{Example}

In light of Theorem~\ref{mut point}, we can prove Theorem~\ref{B cone prod N} by showing that the monomial $\thet_{m_1}^{a_1}\cdots\thet_{m_\ell}^{a_\ell}$ is pointed.
That fact is the last of the following three lemmas.

\begin{Lemma}\label{kappas and phis}
Suppose $\tB$ has signed-nondegenerating coefficients.
Given $p\in M^\circ_\uf$, a~sequence~$\kk$ of indices in $I_\uf$, and an index $j\in I_\uf$,
\begin{enumerate}[label=\rm(\arabic*), ref=\arabic*]\itemsep=0pt
\item \label{kappas}
$\kappa(p,j\kk)=\eta^{\mu_{\kk}(B)}_j(\kappa(p,\kk))$, and
\item \label{phis}
$\phi(p,j\kk)=\psi_j^{\mu_\kk(\tB)}(\phi(p,\kk))-\bigl[\sgn\bigl(\sigma^{(\kk)}_j\bigr)\br{\kappa(p,\kk),d_je_j}\bigr]_+e_j$.
\end{enumerate}
\end{Lemma}
\begin{proof}
Assertion~(\ref{kappas}) is immediate from Theorem~\ref{2 muts}.
Assertion~(\ref{phis}) follows from the same theorem, as we now explain.

By Theorem~\ref{2 muts} and the definition of $\phi(p,\kk)$, we have $\thet_p^\tB=\thet_{\kappa(p,\kk)}^{\mu_\kk(\tB)}\cdot\bigl(\sigma^{(\kk)}\bigr)^{\phi(p,\kk)}$.
Then, by Theorem~\ref{2 muts} again,
\[\thet_{\kappa(p,\kk)}^{\mu_\kk(\tB)} =\thet_{\kappa(p,j\kk)}^{\mu_{j\kk}(\tB)}\cdot \bigl(\sigma_j^{(j\kk)}\bigr)^{-[\sgn(\sigma_j^{(\kk)})\br{\kappa(p,\kk),d_je_j}]_+}.\]
Thus
\begin{align*}
\thet_p^\tB
&=\thet_{\kappa(p,j\kk)}^{\mu_{j\kk}(\tB)}\cdot\bigl(\sigma_j^{(j\kk)}\bigr)^{-[\sgn(\sigma_j^{(\kk)}) \br{\kappa(p,\kk),d_je_j}]_+}\cdot\bigl(\sigma^{(\kk)}\bigr)^{\phi(p,\kk)}\\
&=\thet_{\kappa(p,j\kk)}^{\mu_{j\kk}(\tB)}\cdot\bigl(\sigma^{(j\kk)}\bigr)^{-[\sgn(\sigma_j^{(\kk)})\br{\kappa(p,\kk), d_je_j}]_+e_j+\psi_j^{\mu_\kk(\tB)}(\phi(p,\kk))}.
\tag*{\qed}
\end{align*}\renewcommand{\qed}{}
\end{proof}

\begin{Lemma}\label{kap phi lin}
For every sequence $\kk$, the maps $p\mapsto\kappa(p,\kk)$ and $p\mapsto\phi(p,\kk)$ are each linear on every $B$-cone.
\end{Lemma}
\begin{proof}
The linearity of $p\mapsto\kappa(p,\kk)$ follows from the fact that $\kappa(p,\kk)=\eta_\kk^{B}(p)$ and from \cite[Proposition~5.3]{universal}, which says that $\eta_\kk^{B}$ is linear on every $B$-cone.
The linearity of $p\mapsto\phi(p,\kk)$ then follows using Lemma~\ref{kappas and phis} and an easy induction on the length of $\kk$.
\end{proof}

\begin{Lemma}\label{mut point prod}
Suppose $\tB$ has signed-nondegenerating coefficients and $\u_1,\dots,\u_\ell$ are pointed elements with $\g$-vectors $m_1,\dots,m_\ell$ all in the same $B$-cone.
For nonnegative $a_1,\dots,a_\ell$, the monomial $\u_{m_1}^{a_1}\cdots\u_{m_\ell}^{a_\ell}$ is pointed with $\g$-vector $a_1m_1+\cdots+a_\ell m_\ell$.
\end{Lemma}

\begin{proof}
For any sequence $\kk$, Theorem~\ref{mut point} says that
 $\u_{m_1}^{a_1}\cdots\u_{m_\ell}^{a_\ell}$ is
\[
\bigl(\sigma^{(\kk)}\bigr)^{\sum a_i\phi(m_i,\kk)}\bigl(z^{(\kk)}\bigr)^{\sum a_i\kappa(m_i,\kk)}\]
 (both sums from $1$ to $\ell$) times a formal power series in \smash{$\k\bigl[\bigl[\zeta^{(\kk)}\bigr]\bigr]$} with constant coefficient~$1$.
\mbox{Setting} $m=a_1m_1+\cdots+a_\ell m_\ell$, Lemma~\ref{kap phi lin} now says that $\u_{m_1}^{a_1}\cdots\u_{m_\ell}^{a_\ell}$ is equal to \linebreak \smash{$\bigl(\sigma^{(\kk)}\bigr)^{\phi(m,\kk)}\bigl(z^{(\kk)}\bigr)^{\kappa(m,\kk)}$} times a formal power series in $\k\bigl[\bigl[\zeta^{(\kk)}\bigr]\bigr]$ with constant coefficient~$1$.
By Theorem~\ref{mut point}, $\u_{m_1}^{a_1}\cdots\u_{m_\ell}^{a_\ell}$ is pointed with $\g$-vector $m=a_1m_1+\cdots+a_\ell m_\ell$.
\end{proof}

Proposition~\ref{theta pointed}, Lemma~\ref{mut point prod}, and Theorem~\ref{mut point} combine to prove Theorem~\ref{B cone prod N}, which, in light of Lemma~\ref{N in Dom}, implies Theorem~\ref{B cone prod}.

\subsection{Pointed reduced bases}\label{prb sec}
A reduced basis~$\U$ for $\can\bigl(\tB\bigr)$ is \emph{pointed} if every element of~$\U$ is pointed.
More generally, $\U$ is \emph{pointed at $\kk$} if every element of~$\U$ is pointed at~$\kk$.
In this section, we discuss pointed reduced bases and change of basis between them.
Again, the motivating example is the theta basis.
The following proposition is the concatenation of Propositions~\ref{theta basis} and~\ref{theta pointed}.

\begin{Proposition}\label{theta pointed basis}
Suppose $\tB$ has signed-nondegenerating coefficients.
The set $\set{\thet_m : m\in M^\circ_\uf}$ is a pointed reduced basis for $\can\bigl(\tB\bigr)$.
\end{Proposition}

Arbitrary pointed reduced bases are characterized as follows.

\begin{Theorem}\label{point precise}
Suppose $\tB$ has signed-nondegenerating coefficients and $\U=\set{\u_m:m\in M^\circ_\uf}$ is a subset of $\can\bigl(\tB\bigr)$.
The following are equivalent.
\begin{enumerate}[label=\rm(\roman*), ref=(\roman*)]\itemsep=0pt
\item \label{is prb}
$\U$ is a pointed reduced basis for~$\can\bigl(\tB\bigr)$, indexed so that $\g(\u_m)=m$ for all $m\in M^\circ_\uf$.
\item \label{all like theta}
For all sequences $\kk$, each $\u_m$ is \smash{$\bigl(\sigma^{(\kk)}\bigr)^{\phi(m,\kk)}\bigl(z^{(\kk)}\bigr)^{\kappa(m,\kk)}$} times a formal power series in $\k\bigl[\bigl[\zeta^{(\kk)}\bigr]\bigr]$ with constant coefficient~$1$.
\item \label{all N}
For each $m\in M_\uf^\circ$, there exist constants $c^{(m)}_n\in\k$ such that \smash{$\u_m=\sum_{n\in \N_m^\tB}c^{(m)}_n\,\sigma^n\thet_{m+nB}$}, with $c^{(m)}_0=1$.
\item \label{any prb}
If $\V=\set{\v_p:p\in M^\circ_\uf}$ is a pointed reduced basis for $\can\bigl(\tB\bigr)$, indexed so that $\g(\v_p)=p$ for all $p\in M^\circ_\uf$, then for each $m\in M_\uf^\circ$, there exist constants $c^{(m)}_n\in\k$ with $c^{(m)}_0=1$ such that \smash{$\u_m=\sum_{n\in \N_m^\tB}c^{(m)}_n\,\sigma^n\v_{m+nB}$}.
\end{enumerate}
If these conditions hold, then \smash{$\U^{(\kk)}=\bigl\{\bigl(\sigma^{(\kk)}\bigr)^{-\phi(p,\kk)}\u_p:m=\eta_\kk^B(p)\in M^\circ_\uf\bigr\}$} is a pointed reduced basis~at~$\kk$, for every sequence~$\kk$.
\end{Theorem}

\begin{Remark}\label{FQ remark}
Theorem~\ref{point precise} is a version of the motivating result of~\cite{FanQin}, with significant differences.
We adopt the stronger hypothesis of signed-nondegenerating coefficients as opposed to just that $\tB$ has full rank.
We drop the hypothesis of existence of an injective-reachable seed and therefore the conclusion is about the small canonical algebra.
In accordance with our philosophy about coefficients, we characterize reduced bases rather than bases.
(The notion of reduced bases is also contemplated in \cite[Remark~5.1.4]{FanQin}.)
Whereas Theorem~\ref{point precise} is a~characterization of pointed \emph{reduced} bases, Qin's result \cite[Theorem~1.2.1]{FanQin} is a~characterization of pointed \emph{bases} (in an analogous sense) and can be phrased in terms of a higher-dimensional version of dominance regions that also involves frozen variables.
Qin's result also includes the statement that every pointed basis contains all cluster monomials.
We emphasize that~\cite{FanQin} appeared during the early stages of this project and was influential in the development of this paper.
\end{Remark}

\begin{Remark}\label{why N}
Remark~\ref{FQ remark} suggests a question:
Why can Qin's result \cite[Theorem~1.2.1]{FanQin} be phrased in terms of a version of dominance regions while Theorem~\ref{point precise} is apparently more complicated, using the sets $\N_m^\tB$?
The difference is that Theorem~\ref{point precise} amounts to a characterization of \emph{reducible} bases (see Section~\ref{bases sec}) that are pointed in the sense of~\cite{FanQin}, while Qin's result characterizes arbitrary pointed bases, reducible or not.
\end{Remark}

The implication \ref{is prb}$\implies$\ref{all N} in Theorem~\ref{point precise} shows how any pointed reduced basis for $\can\bigl(\tB\bigr)$ must be related to the theta basis.
The implication \ref{is prb}$\implies$\ref{any prb} gives the form of the change of basis between any two pointed reduced bases.
Before proving Theorem~\ref{point precise}, we state weaker but simpler versions of these statements in terms of the dominance region.

\begin{Corollary}\label{all N dom}
Suppose $\tB$ has signed-nondegenerating coefficients and $\U=\set{\u_m : m\in M^\circ_\uf}$ is a pointed reduced basis for~$\can\bigl(\tB\bigr)$, indexed so that $\g(\u_m)=m$ for all $m\in M^\circ_\uf$.
Then for each $m$, there exist constants $c^{(m)}_{p,n}\in\k$ such that \smash{$\u_m=\thet_m+\sum_p\sum_nc^{(m)}_{p,n}\sigma^n\thet_p$}, summing over $p\in\Dom^B_m$ and $n\in N^+_\uf$ such that $p=m+nB$.
\end{Corollary}

\begin{Corollary}\label{any prb dom}
Suppose $\tB$ has signed-nondegenerating coefficients and $\U=\set{\u_m : m\in M^\circ_\uf}$ and $\V=\set{\v_m:m\in M^\circ_\uf}$ are pointed reduced bases for~$\can\bigl(\tB\bigr)$, with $\g(\u_m)=\g(\v_m)=m$ for all $m\in M^\circ_\uf$.
Then for each~$m$, there exist constants $c^{(m)}_{p,n}\in\k$ such that \smash{$\u_m\!=\!\v_m\!+\!\sum_p\sum_nc^{(m)}_{p,n}\sigma^n\v_p$}, summing over $p\in\Dom^B_m$ and $n\in N^+_\uf$ such that $p=m+nB$.
\end{Corollary}

In order to prove Theorem~\ref{point precise}, we make the following augmentation of Theorem~\ref{mut point}.

\begin{Theorem}\label{mut point aug}
Suppose $\tB$ has signed-nondegenerating coefficients and suppose $\u\in\can\bigl(\tB\bigr)$.
The following are equivalent.
\begin{enumerate}[label=\rm(\roman*), ref=(\roman*)]\itemsep=0pt
\item \label{is pointed aug}
$\u$ is pointed and $\g(\u)=m$.
\item \label{like theta aug}
For all sequences $\kk$, $\u$ is \smash{$\bigl(\sigma^{(\kk)}\bigr)^{\phi(m,\kk)}\bigl(z^{(\kk)}\bigr)^{\kappa(m,\kk)}$} times a formal power series in $\k\bigl[\bigl[\zeta^{(\kk)}\bigr]\bigr]$ with constant coefficient~$1$.
\item \label{N aug}
There exist constants $c_n\in\k$ with $c_0=1$ such that $\u=\sum_{n\in \N_m^\tB}c_n\,\sigma^n\thet_{m+nB}$.
\item \label{any}
If $\V=\set{\v_p:p\in M^\circ_\uf}$ is a pointed reduced basis for $\can\bigl(\tB\bigr)$, indexed so that $\g(\v_p)=p$ for all $p\in M^\circ_\uf$, then there exist constants $c_n\in\k$ with $c_0=1$ and $\u=\sum_{n\in \N_m^\tB}c_n\,\sigma^n\v_{m+nB}$.
\end{enumerate}
\end{Theorem}

\begin{proof}
The equivalence of \ref{is pointed aug}, \ref{like theta aug}, and \ref{N aug} is Theorem~\ref{mut point}.
We can show that \ref{is pointed aug} and \ref{like theta aug} are equivalent to \ref{any} by essentially the same proof for \ref{is pointed aug}, \ref{like theta aug}, and \ref{N aug} in Theorem~\ref{mut point}, as we now explain.

Specifically, suppose \ref{is pointed aug} and \ref{like theta aug} hold.
Since $\u\in\can\bigl(\tB\bigr)$ and $\V=\set{\v_p:p\in M^\circ_\uf}$ is a pointed reduced basis for $\can\bigl(\tB\bigr)$, we write \smash{$\u=\sum_{p\in M^\circ_\uf}\sum_{n\in N_\uf}c_{p,n}\,\sigma^n \v_p$}.
If $c_{p,n}\neq0$, then, once again Lemma~\ref{point el}.\ref{c B n} says that $p=m+nB$.

As before, the element \smash{$\u^{(\kk)}=\u\cdot\bigl(\sigma^{(\kk)}\bigr)^{-\phi(m,\kk)}$} is pointed at $\kk$.
Since $\V$ is a pointed reduced basis for $\can\bigl(\tB\bigr)$, indexed so that $\g(\v_p)=p$ for all $p\in M^\circ_\uf$, the equivalence of \ref{is pointed aug} and~\ref{like theta aug} says that each~$\v_p$ is \smash{$\bigl(\sigma^{(\kk)}\bigr)^{\phi(p,\kk)}\bigl(z^{(\kk)}\bigr)^{\kappa(p,\kk)}$} times a formal power series in $\k\bigl[\bigl[\zeta^{(\kk)}\bigr]\bigr]$ with constant coefficient~$1$.
We write $\v^{(\kk)}_{\kappa(p,\kk)}$ for \smash{$\bigl(\sigma^{(\kk)}\bigr)^{-\phi(p,\kk)}\v_p$}.
We establish that \smash{$\u=\sum_{n\in\N_m^\tB}c_n\sigma^n\v_{m+nB}$} by the same argument as in the proof of Theorem~\ref{mut point}.
Thus~\ref{is pointed aug} and \ref{like theta aug} imply \ref{any}.

Conversely, suppose
\[
\u=\sum_{n\in\N_m^\tB}c_n\sigma^n\v_{m+nB} \qquad \text{with $c_0=1$}
\]
 and let $\kk$ be any sequence.
We~already saw that each~$\v_{m+nB}$ is \smash{$\bigl(\sigma^{(\kk)}\bigr)^{\phi(m+nB,\kk)}\bigl(z^{(\kk)}\bigr)^{\kappa(m+nB,\kk)}$} times a formal power series in $\k\bigl[\bigl[\zeta^{(\kk)}\bigr]\bigr]$ with constant coefficient~$1$.
Arguing just as in the proof of Theorem~\ref{mut point}, we establish~\ref{like theta}.
\end{proof}

\begin{proof}[Proof of Theorem~\ref{point precise}]
The following condition is formally weaker than~\ref{is prb}.
\begin{enumerate}[label=\rm(\roman*), ref=(\roman*)]\itemsep=0pt
\item[(i$'$)]
$\U$ is a set of pointed elements of~$\can\bigl(\tB\bigr)$, indexed so that $\g(\u_p)=p$ for all $p\in M^\circ_\uf$.
\end{enumerate}
Theorem~\ref{mut point aug} says that (i$'$), \ref{all like theta}, \ref{all N}, and \ref{any} are equivalent.
But (i$'$) says in particular that each $\u_m$ is $z^m$ times a formal power series in $\k[[\zeta]]$ with constant coefficient~$1$ and thus implies, by Theorem~\ref{pointed set basis}, that $\U$ is a reduced basis for $\can\bigl(\tB\bigr)$.
We see that (i$'$) and \ref{is prb} are equivalent.
\end{proof}

We can now obtain a version of Theorem~\ref{B cone prod N} that replaces the theta basis with an arbitrary pointed reduced basis, under the assumption of signed-nondegenerating coefficients.
Lemma~\ref{mut point prod} and Theorem~\ref{point precise} combine to prove the following theorem.

\begin{Theorem}\label{B cone prod U}
Suppose that $\tB$ has signed-nondegenerating coefficients and that $m_1,\dots,m_\ell$ are all contained in the same $B$-cone.
Let $\V=\set{\v_p : p\in M^\circ_\uf}$ be a pointed reduced basis for $\can\bigl(\tB\bigr)$.
Write $m=a_1m_1+\cdots+a_\ell m_\ell$ for nonnegative integers $a_1,\dots,a_\ell$.
Then there exist coefficients $c_n\in\k$ with $c_0=1$ such that $\v_{m_1}^{a_1}\cdots\v_{m_\ell}^{a_\ell}=\sum_{n\in\N_m^\tB}c_n\sigma^n\v_{m+nB}$.
\end{Theorem}

\subsection{The ray basis}
We highlight another example of a pointed basis called the ray basis that exists in many cases, depending on properties of the mutation fan.
The rational vectors in $V^*$ are the vectors $x\in V^*$ such that there exists an integer $a$ such that $ax\in M_\uf^\circ$.
A closed cone is \emph{rational} if it is the nonnegative $\reals$-linear span of a collection of rational vectors and also \emph{simplicial} if those vectors are linearly independent.
Given a rational simplicial cone $C$, there is a unique linearly independent set of primitive vectors in $C\cap M_\uf^\circ$ whose nonnegative $\reals$-linear span is $C$.
(These are the primitive vectors of $M_\uf^\circ$ in the extreme rays of $C$.)
The cone $C$ is \emph{integral} if the nonnegative $\integers$-linear span of this set of primitive vectors is all of $C\cap M_\uf^\circ$.

We begin by defining the ray basis in the simplest case where it exists.
Suppose, for some~$B$, that the mutation fan $\F_B$ is a rational, simplicial, integral fan, meaning that every cone in $\F_B$ is rational, simplicial, and integral.
For every $m\in M^\circ_\uf$, there is a smallest cone $C_m$ of $\F_B$ containing $m$, because the mutation fan $\F_B$ is complete.
Since $C_m$ is a~rational simplicial cone, writing $m_1,\dots,m_\ell$ for the primitive vectors of $M_\uf^\circ$ in the extreme rays of $C_m$, there is a~unique expression $m=\sum_{i=1}^\ell c_im_i$.
The $c_i$ are all positive.
Since $\F_B$ is integral, the $c_i$ are all integers.
Define $\rho_m$ to be $\prod_{i=1}^\ell\thet_{m_i}^{c_i}$.
We will see below that the set of all the $\rho_m$ is a pointed reduced basis for $\can\bigl(\tB\bigr)$.
But first, we will define the $\rho_m$ under weaker conditions on the mutation fan.

Given a complete fan $\F$, the \emph{rational part} of $\F$, if it exists, is a rational fan $\F^\rationals$ with the following two properties:
First, each cone of $\F^\rationals$ is contained in a cone of $\F$.
Second, for each cone~$C$ of $\F$, there is a unique largest cone $C'$ of $\F^\rationals$ contained in $C$, and $C'$ contains all of the rational vectors in $C$.
(See \cite[Definition~4.9]{unisurface}.)
The rational part of $\F$ is unique if it exists.
When $\F$ is rational, it is its own rational part.

\begin{Remark}\label{rat fan rem}
When $\F^\rationals$ exists, it contains all rational vectors, because $\F$ is complete and because for each cone $C$ of $\F$, there is a cone $C'$ of $\F^\rationals$ that contains all rational vectors of $C$.
However, $\F^\rationals$ is not complete (i.e., does not contain all real vectors) when $\F$ is not rational:
If $C\in\F$ is irrational, the largest cone $C'$ of~$\F^\rationals$ contained in $C$ is strictly smaller than~$C$, because~$C'$ is rational and $C$ is not.
\end{Remark}

Now suppose that the rational part \smash{$\F^\rationals_B$} of the mutation fan $\F_B$ exists and is simplicial and integral.
For every $m\in M^\circ_\uf$, let $C_m$ be the smallest cone of $\F_B$ containing $m$ and let $C'_m$ be the largest cone of \smash{$\F^\rationals_B$} contained in $C_m$.
Since $C'_m$ is rational, simplicial, and integral, if $m_1,\dots,m_\ell$ are the primitive vectors of $M_\uf^\circ$ in the extreme rays of $C'_m$, there is a unique expression \smash{$m=\sum_{i=1}^\ell c_im_i$} and the $c_i$ are all positive integers.
Again in this case, define $\rho_m$ to be \smash{$\prod_{i=1}^\ell\thet_{m_i}^{c_i}$}.

\begin{Theorem}\label{ray basis}
Suppose the rational part of the mutation fan $\F_B$ exists and is simplicial and integral.
Suppose also that $\tB$ is an extension of $B$ with signed-nondegenerating coefficients.
Then $\set{\rho_m:m\in M^\circ_\uf}$ is a pointed reduced basis for $\can\bigl(\tB\bigr)$.
\end{Theorem}
\begin{proof}
The hypothesis on the rational part of $\F_B$ makes it possible to define $\rho_m$ for every $m\in M^\circ_\uf$.
Since each $\rho_m$ is a monomial in theta functions, it is $z^m$ times a formal power series in $\k[[\zeta]]$ with constant coefficient~$1$.
Thus Theorem~\ref{pointed set basis} says that $\set{\rho_m:m\in M^\circ_\uf}$ is a reduced basis for $\can\bigl(\tB\bigr)$.
Since the cone $C'_m$ in the definition of $\rho_m$ is contained in some cone $C_m$ of $\F_B$ (a $B$-cone or a face of a $B$-cone), Lemma~\ref{mut point prod} says that $\set{\rho_m:m\in M^\circ_\uf}$ is a pointed reduced basis.
\end{proof}

We call $\set{\rho_m:m\in M^\circ_\uf}$ the \emph{ray basis} when it exists.
A sufficient condition for the ray basis to exist is that $B$ admits positive universal geometric coefficients over~$\integers$, in the sense of~\cite{universal}.
(That term doesn't appear in \cite{universal}, but \emph{universal geometric coefficients} over $\integers$ exist for $B$ if and only if there is a $\integers$-basis for~$B$.
There is a notion of a positive $\integers$-basis for $B$, and the corresponding universal geometric coefficients may therefore be called positive.
See \cite[Section~6]{universal}, particularly \cite[Proposition~6.7]{universal} and \cite[Proposition~6.11]{universal}.)
We conclude with some examples where the ray basis exists and does not exist.
Based on these examples, it seems reasonable to guess that the ray basis exists precisely in the case of finite mutation type (or for rank $2$, when $B$ is of finite or affine type).

\begin{Example}\label{rank 2}
Suppose \smash{$B=\left[\begin{smallmatrix}0&a\\b&0\end{smallmatrix}\right]$} with $\sgn(b)=-\sgn(a)$.
If $ab\ge-4$, then the ray basis exists by \cite[Proposition~9.8]{universal}.
If $ab<-4$, then the ray basis fails to exist because~$\F_B$ has a~full-dimensional cone that is not rational.
\end{Example}

\begin{Example}\label{finite type}
When $B$ is of finite type, $\F_B$ is a finite, rational, simplicial, integral fan because it coincides with the fan whose cones are the nonnegative $\reals$-span of the $\g$-vectors of compatible sets of cluster variables for $B^{\mathsf{T}}$.
(See \cite[Proposition~9.4]{universal}.)
Thus the ray basis exists in this case.
It coincides with the theta basis, because the theta basis essentially coincides with the cluster monomials.
(See Remark~\ref{not afoul}.)
\end{Example}

\begin{Example}\label{affine type}
When $B$ is of affine type, $\F_B$ is a rational, simplicial, integral fan, as we now explain.
Starting with acyclic affine type, $\F_{B^{\mathsf{T}}}$ coincides with a fan $\nu_c(\Fan_c(\RS))$ \cite[Theorem~2.9]{affscat}.
Here $\Fan_c(\RS)$ is a complete fan in $V$ that is rational and simplicial by construction and integral (relative to the lattice $N_\uf$ in $V$) by \cite[Proposition~5.14\,(2),\,(6)]{affdenom}.
The map $\nu_c$ is piecewise linear, respects the fan structure of $\Fan_c(\RS)$, and takes $N_\uf$ to $M_\uf^\circ$.
Thus $\F_{B^{\mathsf{T}}}$ is rational, simplicial, and integral when $B$ is acyclic of affine type.
We can remove the transpose because $B^{\mathsf{T}}$ is acyclic of affine type if and only if $B$ is.
We can remove the requirement of acyclicity because every exchange matrix of affine type is mutation-equivalent to an acyclic matrix and because each mutation map is an isomorphism of mutation fans and preserves $M_\uf^\circ$.
\end{Example}

\begin{Example}\label{marked surfaces}
The motivating example for the definition of the ray basis is the case where~$B$ is the signed adjacency matrix of a marked surface in the sense of \cite{cats1}.
Mandel and~Qin~\cite{MandelQin} showed that in this case, the theta basis coincides with the bracelets basis of \cite{MSWbases} (except that for the once-punctured torus, the two bases agree up to certain constant multiples.)
Also in most marked surfaces cases, the rational part of the mutation fan is known to be a rational, simplicial fan called the rational quasi-lamination fan, defined in terms of shear coordinates of curves.
(The exceptions are the once-punctured surface without boundary and of genus~$>1$, where this fact is expected to be true, but not yet proved.)
The integrality of the rational lamination fan is a consequence of a theorem of William Thurston, quoted as \cite[Theorem~12.3]{cats2} and rephrased for this purpose as \cite[Theorem~3.10]{unisurface}.
The ray basis in this case is easily seen to coincide with the bangles basis of \cite{MSWbases}.
Thus the ray basis in general should be viewed as a~generalization of the bangles basis.
\end{Example}

\subsection*{Acknowledgements}
Nathan Reading was partially supported by the National Science Foundation under Grant Number DMS-2054489 and by the Simons Foundation under award number 581608. Salvatore Stella was partially supported by PRIN 20223FEA2E - PE1 and by INdAM - GNSAGA. We thank the anonymous referees for comments and suggestions that improved the paper.

\pdfbookmark[1]{References}{ref}
\LastPageEnding


\begin{thebibliography}{99}
\footnotesize\itemsep=0pt

\bibitem{ca3}
Berenstein A., Fomin S., Zelevinsky A., Cluster algebras~{III}: {U}pper bounds
 and double {B}ruhat cells,
 \href{https://doi.org/10.1215/S0012-7094-04-12611-9}{\textit{Duke Math.~J.}}
 \textbf{126} (2005), 1--52,
 \href{https://arxiv.org/abs/math/0305434}{arXiv:math/0305434}.

\bibitem{CGMMRSW}
Cheung M.W., Gross M., Muller G., Musiker G., Rupel D., Stella S., Williams H.,
 The greedy basis equals the theta basis: a rank two haiku,
 \href{https://doi.org/10.1016/j.jcta.2016.08.004}{\textit{J.~Combin. Theory
 Ser.~A}} \textbf{145} (2017), 150--171,
 \href{https://arxiv.org/abs/1508.01404}{arXiv:1508.01404}.

\bibitem{cats1}
Fomin S., Shapiro M., Thurston D., Cluster algebras and triangulated surfaces.
 Part~{I}: {C}luster complexes,
 \href{https://doi.org/10.1007/s11511-008-0030-7}{\textit{Acta Math.}}
 \textbf{201} (2008), 83--146,
 \href{https://arxiv.org/abs/math/0608367}{arXiv:math/0608367}.

\bibitem{cats2}
Fomin S., Thurston D., Cluster algebras and triangulated surfaces. {P}art~{II}:
 {L}ambda lengths, \href{https://doi.org/10.1090/memo/1223}{\textit{Mem. Amer.
 Math. Soc.}} \textbf{255} (2018), v+97~pages,
 \href{https://arxiv.org/abs/1210.5569}{arXiv:1210.5569}.

\bibitem{ca4}
Fomin S., Zelevinsky A., Cluster algebras~{IV}: {C}oefficients,
 \href{https://doi.org/10.1112/S0010437X06002521}{\textit{Compos. Math.}}
 \textbf{143} (2007), 112--164,
 \href{https://arxiv.org/abs/math/0602259}{arXiv:math/0602259}.

\bibitem{GHKK}
Gross M., Hacking P., Keel S., Kontsevich M., Canonical bases for cluster
 algebras, \href{https://doi.org/10.1090/jams/890}{\textit{J.~Amer. Math.
 Soc.}} \textbf{31} (2018), 497--608,
 \href{https://arxiv.org/abs/1411.1394}{arXiv:1411.1394}.

\bibitem{MandelQin}
Mandel T., Qin F., Bracelets bases are theta bases,
 \href{https://arxiv.org/abs/2301.11101}{arXiv:2301.11101}.

\bibitem{Muller}
Muller G., The existence of a maximal green sequence is not invariant under
 quiver mutation, \href{https://doi.org/10.37236/5412}{\textit{Electron.~J.
 Combin.}} \textbf{23} (2016), 2.47, 23~pages,
 \href{https://arxiv.org/abs/1503.04675}{arXiv:1503.04675}.

\bibitem{MSWbases}
Musiker G., Schiffler R., Williams L., Bases for cluster algebras from
 surfaces, \href{https://doi.org/10.1112/S0010437X12000450}{\textit{Compos.
 Math.}} \textbf{149} (2013), 217--263,
 \href{https://arxiv.org/abs/1110.4364}{arXiv:1110.4364}.

\bibitem{Nakanishi11a}
Nakanishi T., Zelevinsky A., On tropical dualities in cluster algebras, in
 Algebraic {G}roups and {Q}uantum {G}roups, \textit{Contemp. Math.}, Vol. 565,
 \href{https://doi.org/10.1090/conm/565/11159}{American Mathematical Society,
 Providence, RI}, 2012, 217--226,
 \href{https://arxiv.org/abs/1101.3736}{arXiv:1101.3736}.

\bibitem{FanQin}
Qin F., Bases for upper cluster algebras and tropical points,
 \href{https://doi.org/10.4171/jems/1308}{\textit{J.~Eur. Math. Soc. (JEMS)}}
 \textbf{26} (2024), 1255--1312,
 \href{https://arxiv.org/abs/1902.09507}{arXiv:1902.09507}.

\bibitem{universal}
Reading N., Universal geometric cluster algebras,
 \href{https://doi.org/10.1007/s00209-013-1264-4}{\textit{Math.~Z.}}
 \textbf{277} (2014), 499--547,
 \href{https://arxiv.org/abs/1209.3987}{arXiv:1209.3987}.

\bibitem{unisurface}
Reading N., Universal geometric cluster algebras from surfaces,
 \href{https://doi.org/10.1090/S0002-9947-2014-06156-4}{\textit{Trans. Amer.
 Math. Soc.}} \textbf{366} (2014), 6647--6685,
 \href{https://arxiv.org/abs/1209.4095}{arXiv:1209.4095}.

\bibitem{scatcomb}
Reading N., A combinatorial approach to scattering diagrams,
 \href{https://doi.org/10.5802/alco.107}{\textit{Algebr. Comb.}} \textbf{3}
 (2020), 603--636, \href{https://arxiv.org/abs/1806.05094}{arXiv:1806.05094}.

\bibitem{scatfan}
Reading N., Scattering fans,
 \href{https://doi.org/10.1093/imrn/rny260}{\textit{Int. Math. Res. Not.}}
 \textbf{2020} (2020), 9640--9673,
 \href{https://arxiv.org/abs/1712.06968}{arXiv:1712.06968}.

\bibitem{affdomreg}
Reading N., Rupel D., Stella S., Dominance regions for affine cluster algebras,
 \href{https://arxiv.org/abs/2512.02218}{arXiv:2512.02218}.

\bibitem{affdenom}
Reading N., Stella S., An affine almost positive roots model,
 \href{https://doi.org/10.4171/jca/37}{\textit{J.~Comb. Algebra}} \textbf{4}
 (2020), 1--59, \href{https://arxiv.org/abs/1707.00340}{arXiv:1707.00340}.

\bibitem{RupelStella}
Rupel D., Stella S., Dominance regions for rank two cluster algebras,
 \href{https://doi.org/10.1007/s00026-023-00636-4}{\textit{Ann. Comb.}}
 \textbf{27} (2023), 873--894,
 \href{https://arxiv.org/abs/2206.04778}{arXiv:2206.04778}.

\bibitem{affscat}
Reading N., Stella S., Cluster scattering diagrams of acyclic affine type,
 \textit{C.~R.~Math. Acad. Sci. Paris}, {t}o appear,
 \href{https://arxiv.org/abs/2205.05125}{arXiv:2205.05125}.

\bibitem{afftheta}
Reading N., Stella S., Theta functions in acyclic affine type,
 \href{https://arxiv.org/abs/2603.23429}{arXiv:2603.23429}.

\end{thebibliography}
\end{document}